\documentclass[english]{extarticle}
\usepackage[T1]{fontenc}
\usepackage[latin9]{inputenc}
\usepackage{float}
\usepackage{amsmath}
\usepackage{amsthm}
\usepackage{amssymb}
\usepackage{graphicx}
\usepackage[numbers]{natbib}

\makeatletter

\providecommand{\tabularnewline}{\\}

\theoremstyle{plain}
\newtheorem{thm}{\protect\theoremname}
\theoremstyle{remark}
\newtheorem{rem}{\protect\remarkname}

\usepackage{babel}

\@ifundefined{showcaptionsetup}{}{%
 \PassOptionsToPackage{caption=false}{subfig}}
\usepackage{subfig}
\makeatother

\usepackage{babel}
\providecommand{\remarkname}{Remark}
\providecommand{\theoremname}{Theorem}

\begin{document}
\title{Nonlinear Model Reduction to \\
Temporally Aperiodic Spectral Submanifolds}
\author{George Haller\thanks{Corresponding author. Email: georgehaller@ethz.ch}
~~and Roshan S. Kaundinya \\
Institute for Mechanical Systems\\
ETH Zürich\\
Leonhardstrasse 76, 8092 Zürich, Switzerland\\
\\
(To appear in \textit{\textbf{Chaos}})}
\date{April 3, 2024}

\maketitle
We extend the theory of spectral submanifolds (SSMs) to general non-autonomous
dynamical systems that are either weakly forced or slowly varying.
Examples of such systems arise in structural dynamics, fluid-structure
interactions and control problems. The time-dependent SSMs we construct
under these assumptions are normally hyperbolic and hence will persist
for larger forcing and faster time dependence that are beyond the
reach of our precise existence theory. For this reason, we also derive
formal asymptotic expansions that, under explicitly verifiable nonresonance
conditions, approximate SSMs and their aperiodic anchor trajectories
accurately for stronger, faster or even temporally discontinuous forcing.
Reducing the dynamical system to these persisting SSMs provides a
mathematically justified model reduction technique for non-autonomous
physical systems whose time dependance is moderate either in magnitude
or speed. We illustrate the existence, persistence and computation
of temporally aperiodic SSMs in mechanical examples under chaotic
forcing.\\
\\

\textbf{Reduced models for complex physical systems are of growing
interest in various areas of applied science and engineering. Mathematically
justifiable reduction approaches, such as spectral submanifold (or
SSM) reduction, seek to identify the internal dynamics on lower-dimensional,
attracting invariant sets in the phase space of the system. These
reduced dynamics then become viable reduced models for general trajectories
that approach the invariant set and synchronize with its inner motions.
SSM reduction also allows for periodic or quasiperiodic time dependence
in the full system, but has been inapplicable to systems with more
general time dependence, such as impulsive, chaotic or discontinuous
forcing. This has hindered applications of SSM reduction to a number
of problems in structural dynamics. Here we remove this limitation
by extending SSM theory of temporally aperiodic dynamical systems.
We obtain exact results for cases of smooth small or smooth slow forcing,
but find that our formulas for SSM-reduced dynamics extend to larger
and faster forcing in physical examples, including even discontinuous
chaotic forcing.}

\section{Introduction}

In its simplest form, a spectral submanifold (SSM) of an autonomous
dynamical system is an invariant manifold $W(E)$ that is tangent
to a spectral subspace $E$ of the linearized system at a fixed point
(\citet{haller16}). Classic examples of spectral submanifolds are
the stable, unstable and center manifolds tangent to spectral subspaces
in which the linearized spectrum has eigenvalues with purely negative,
positive and zero real parts, respectively. The stable and unstable
manifolds are well known to be unique and as smooth as the dynamical
system, while center manifolds are non-unique and not all of them
are guaranteed to be as smooth as the dynamical system (\citet{guckenheimer83,hirsch77}).
Near fixed points that only have eigenvalues with negative and zero
real parts, center manifolds attract all trajectories and hence the
dynamics restricted to it provides a mathematically exact reduced-order
model for the full system (\citet{carr82}, \citet{roberts15}).

Dissipative physical systems, however, generically have hyperbolic
equilibria and hence admit no center manifolds. Instead, near their
stable equilibria, such systems tend to have a set of fastest-decaying
modes that die out quickly, leaving a set of slower decaying modes
to govern the longer-term dynamics. Slow manifolds (i.e., SSMs constructed
over such slower decaying modes) then replace center manifolds as
targets for mathematically justified model reduction. Such slow SSMs
were first targeted via Taylor expansions as nonlinear normal modes
(NNMs) by \citet{shaw93}. These insightful calculations were then
extended by the same authors to periodically and quasi-periodically
forced mechanical systems to approximate forced mechanical response
in a number of settings (see, e.g., the reviews by \citet{kerschen09,mikhlin11,touze21,mikhlin23}).

Later mathematical analysis of general SSMs yielded precise existence,
uniqueness and smoothness results for these manifolds. Specifically,
if the spectral subspace $E$ comprises either only decaying modes
or only growing modes (i.e., $E$ is a like-mode spectral subspace)
with no integer resonance relationships to the modes outside of $E$,
then the slow SSM family $W(E)$ has a unique, primary member, $W^{\infty}(E)$,
that is as smooth as the full dynamical system (\citet{cabre03},
\citet{haller16}). The remaining secondary (or fractional) members
of the SSM family have reduced but precisely known order of differentiability
that depends on the ratio of linearized decay rates outside $E$ to
those inside $E$ (\citet{haller23}). If we further disallow any
integer resonance in the full linearized spectrum, then primary and
fractional SSMs also exist when $E$ is of mixed-mode type, i.e.,
spanned by a combination of stable and unstable linear modes (\citet{haller23}). 

All these results also hold for discrete autonomous dynamical systems,
and hence SSM results also cover time-periodic continuous dynamical
systems when applied to their Poincaré maps. Based on this fact, numerical
implementations of time-periodic SSM calculations for periodically
forced finite-element structures have appeared in \citet{ponsioen2018,ponsioen2019,ponsioen2020,jain2022,vizzaccaro23}.
An open source, equation-driven MATLAB toolbox (SSMTool) with a growing
collection of worked problems is also available for mechanical systems
with general nonlinearities (\citet{jain23}). Data-driven construction
of time-periodic SSMs have been developed and applied to numerical
and experimental data by \citet{cenedese22a,cenedese22b,kaszas22,axas22},
with open-source MATLAB implementations (SSMLearn and fastSSM) available
from \citet{cenedese21}.

Outside autonomous and time-periodic non-autonomous systems, the existence
of like-mode SSMs has only been treated under small, time-quasiperiodic
perturbations of autonomous dynamical systems (\citet{haro06,haller16,opreni23,thurner23}).
In such systems, the role of hyperbolic fixed points as anchor points
for SSMs is taken over by invariant tori whose dimension is equal
to the number of rationally independent frequencies present in the
forcing. An SSM in such a case perturbs from the direct product of
an unperturbed invariant torus with an underlying spectral subspace
$E$. Such SSMs have been shown to exist for like-mode spectral subspaces
$E$ under small quasi-periodic perturbations, provided that the real
part of the linearized spectrum within $E$ has no integer relationships
with the real part of the spectrum outside $E$ (\citet{haro06,haro16,haller16}). 

Related work by \citet{fontich06} covers the persistence and smoothness
of primary SSMs emanating from an arbitrary attracting\emph{ }orbit
of an autonomous dynamical system. While all non-autonomous systems
become autonomous when their phase space is extended with the time
variable, this theory does not apply under such an extension. The
reason is that all attracting orbits lose hyperbolicity in the extended
phase space due to the presence of the neutrally stable time direction. 

In summary, while available SSM results have proven highly effective
in equation-driven and data-driven reduced-order modeling of autonomous,
time-periodic and time-quasiperiodic systems, they offer no theoretical
basis or computational scheme for physical systems with general aperiodic
time-dependence. Yet the aperiodic setting is clearly of importance
in a number of problems, including turbulent fluid-structure interactions,
civil engineering structures subject to benchmark aperiodic forcing
(such as those mimicking earthquakes) and control of robot motion.

In this paper, we extend available SSMs results to systems with general
time dependence. While several powerful linearization results imply
the existence of invariant manifolds for such non-autonomous dynamical
systems, these results guarantee either no smoothness for SSM-type
invariant manifolds (see, e.g., \citet{palmer73}) or rely on conditions
involving the Lyapunov spectra, Sacker\textendash Sell spectra or
dichotomy spectra of some associated non-autonomous linear systems
of ODEs (see, e.g., \citet{yomdin88}, \citet{poetzsche10} and \citet{cuong19}).
The latter types of conditions are intuitively clear but not readily
verifiable, especially not in a data-driven setting. Here our objective
is to conclude the existence of time-aperiodic SSMs under directly
computable conditions that also lead to explicitly computable SSM-reduced
models in equation-driven and data-driven applications.

To this end, we consider two settings that arise frequently in practice:
weak and additive non-autonomous time dependence and slowly varying
(or adiabatic) time dependence. The first setting of weak non-autonomous
external forcing is common in structural vibrations, wherein a structure's
steady-state response is of interest under various moderate loading
scenarios. So far, related studies have been restricted to temporally
periodic or quasiperiodic forcing (see, e.g., \citet{ponsioen2018,ponsioen2019,ponsioen2020,li22a,li22b,jain2022,vizzaccaro22,opreni23}),
because the existence and exact form of a steady state and its associated
SSM have been unknown for more general forcing profiles. The second
setting of slow time dependence arises, for instance, in controls
applications wherein the intended motion of a structure is generally
much slower than the characteristic time scales of its internal vibrations.
In those applications, the lack of an adiabatic SSM theory has so
far confined model-reduction studies to small-amplitude trajectories
along which a single, autonomous SSM computed at a nearby fixed point
was used for modeling purposes (\citet{alora23,alora23b}).

In both of these non-autonomous settings, we use, modify or extend
prior invariant manifold results and techniques to conclude the existence
of weakly aperiodic or adiabatic SSMs in the limit of small enough
or slow enough time dependence, respectively. We then derive explicit
recursive formulas for the arising non-autonomous SSMs and the aperiodic
anchor trajectories to which they are attached. These formulas also
cover and extend temporally periodic and aperiodic SSM computations
to arbitrarily high order of accuracy. Using simple mechanical examples
subjected to chaotic excitation, we illustrate that the new asymptotic
formulas yield accurate reduced-order models even for larger and faster
forcing. 

\section{Non-autonomous SSMs under weak forcing\label{sec:Non-Autonomous-SSM}}

\subsection{Set-up \label{sec:set-up}}

Consider a non-autonomous dynamical system of the form
\begin{equation}
\dot{x}=Ax+f_{0}(x)+f_{1}(x,t),\quad x\in\mathbb{R}^{n},\quad A\in\mathbb{R}^{n\times n},\quad f_{0}\in C^{r}(U),\quad f_{0}(x)=o\left(\left|x\right|\right),\label{eq: nonlinear system}
\end{equation}
for some integer $r\geq0$ and with
\begin{equation}
\left\Vert f_{1}\right\Vert _{U}=\sup_{(x,t)\in U\times\mathbb{R}}\left|f_{1}(x,t)\right|<\infty\label{eq:f_1 uniformly bounded}
\end{equation}
on a compact neighborhood $U\subset\mathbb{R}^{n}$. We consider system
(\ref{eq: nonlinear system}) a perturbation of the autonomous system
\begin{equation}
\dot{x}=Ax+f_{0}(x),\label{eq:unperturbed nonlinear system}
\end{equation}
that has a fixed point at $x=0$ by our assumptions on $f_{0}(x)$.
Both $A$ and $f_{0}(x)$ may additionally depend on parameters, which
we will suppress here for notational simplicity but will point out
in the statement of our main results. Note that the time-dependent
term $f_{1}(x,t)$ is also allowed to depend on the phase space variable
$x$. Consequently, in addition to describing purely time-dependent
external forcing, $f_{1}(x,t)$ can also capture what is commonly
called parametric forcing in the structural vibrations literature,
i.e., time-dependence in the internal structure of the system.

We assume that the origin is a hyperbolic fixed point of (\ref{eq:unperturbed nonlinear system}),
i.e., all eigenvalues in the spectrum of $A,$
\begin{equation}
\mathrm{spect}\left(A\right)=\left\{ \lambda_{1},\ldots,\lambda_{n}\right\} ,\label{eq:spect(A)}
\end{equation}
listed in the order
\begin{equation}
\mathrm{Re}\lambda_{n}\leq\mathrm{Re}\lambda_{n-1}\leq\ldots\leq\mathrm{Re}\lambda_{2}\leq\mathrm{Re}\lambda_{1},\label{eq:eigenvalue ordering}
\end{equation}
satisfy 
\begin{equation}
\mathrm{Re}\lambda_{j}\neq0,\quad j=1,\ldots,n.\label{eq:hyperbolicity assumption}
\end{equation}
For simplicity of exposition, we assume that $A$ is semisimple and
hence has $n$ eigenvectors $e_{1},\ldots,e_{n}\in\mathbb{C}^{n}$
corresponding to the eigenvalues listed in (\ref{eq:spect(A)}). The
$k^{th}$ \emph{eigenspace }$E_{k}$ of $A$ is then the linear span
of the real and imaginary parts of the eigenvectors corresponding
to the eigenvalue $\lambda_{k}$, i.e., 
\[
E_{k}=\mathrm{\underset{\text{A\ensuremath{e_{j}}=\ensuremath{\lambda_{k}e_{j}}}}{span}}\left\{ \mathrm{Re\,}e_{j},\mathrm{Im}\,e_{j}\right\} .
\]
Note that any such eigenspace in an invariant subspace under the dynamics
of the linearized ODE 
\begin{equation}
\dot{x}=Ax.\label{eq:linearized ODE}
\end{equation}

A \emph{spectral subspace} $E$ is the direct sum of a selected group
of $\ell$ eigenspaces, i.e.
\begin{equation}
E=E_{j_{1}}\oplus\ldots\oplus E_{j_{\ell}}.\label{eq:spectral subspace}
\end{equation}
 Two important spectral subspaces are the \emph{stable subspace} $E^{s}$
and the \emph{unstable subspace} $E^{u}$, defined as 
\begin{equation}
E^{s}=\underset{\mathrm{Re}\lambda_{j}<0}{\oplus}E_{j},\qquad E^{u}=\underset{\mathrm{Re}\lambda_{j}>0}{\oplus}E_{j},\qquad E^{s}\oplus E^{u}=\mathbb{R}^{n}.\label{eq:E^s and E^u}
\end{equation}
At least one of $E^{s}$ and $E^{u}$ is nonempty due to the hyperbolicity
assumption (\ref{eq:hyperbolicity assumption}). As a consequence
of this assumption, there exist also constants $K,\kappa>0$ such
that 
\begin{equation}
\left\Vert e^{At}\vert_{E^{s}}\right\Vert \leq Ke^{-\kappa t},\quad\left\Vert e^{-At}\vert_{E^{u}}\right\Vert \leq Ke^{-\kappa t},\quad t\geq0.\label{eq:dichotomy for autonomous ODE}
\end{equation}
This property of $A$ is usually referred to as exponential dichotomy.
For $\kappa$, we can select any positive number satisfying
\[
0<\kappa<\min_{1\leq j\leq n}\left|\mathrm{Re}\lambda_{j}\right|.
\]
In contrasts, the choice of $K$ depends on the eigenvector geometry
of $A$. Specifically, if $A$ is a normal operator and hence has
an orthogonal eigenbasis, we can select $K=1$. 

\subsection{Existence and computation of a non-autonomous anchor trajectory}

We first state general results on the fate of the $x=0$ fixed point
under the non-autonomous forcing term $f_{1}(x,t)$. Specifically,
we give conditions under which this fixed point perturbs for moderately
large $\left|f_{1}(x,t)\right|$ into a unique nearby hyperbolic trajectory
$x^{*}(t)$ of (\ref{eq: nonlinear system}). This distinguished trajectory
remains uniformly bounded for all times and has the same stability
type as the $x=0$ fixed point of system (\ref{eq:unperturbed nonlinear system}).
\begin{thm}
\label{thm:x^* under conditions}\textbf{\emph{{[}Existence of anchor
trajectory for non-autonomous SSMs{]}}} Assume that in a ball $B_{\delta}\subset U$
of radius $\delta>0$ around $x=0$, the functions $f_{0}$ and $f_{1}$
are of class $C^{0}$ and admit Lipschitz constants $L_{0}(\delta)$
and $L_{1}(\delta)$, respectively, in $x$ for all $t\in\mathbb{R}$.
Assume further that the conditions
\begin{equation}
\left|f_{1}(x,t)\right|\leq\frac{\kappa\delta}{2K}-\left|f_{0}(x)\right|,\qquad L_{1}(\delta)\leq\frac{\kappa}{4K}-L_{0}(\delta),\qquad x\in B_{\delta},\quad t\in\mathbb{R},\label{eq:conditions for x^*}
\end{equation}
are satisfied with constants $K,\kappa>0$ satisfying (\ref{eq:dichotomy for autonomous ODE}).
Then the following hold:

(i) System (\ref{eq: nonlinear system}) has a unique, uniformly bounded
trajectory $x^{*}(t)$ that remains is $B_{\delta}$  for all $t\in\mathbb{R}$
and has the same stability type as the $x=0$ fixed point of system
(\ref{eq:unperturbed nonlinear system}). 

(ii) The trajectory $x^{*}(t)$ is as smooth in any parameter as system
(\ref{eq: nonlinear system}).
\end{thm}
\begin{proof}
See Appendix \ref{subsec:proof of existence of anchor trajectory}.
\end{proof}
By statement (i) of Theorem \ref{thm:x^* under conditions}, the anchor
trajectory $x^{*}(t)$ takes over the role of the $x=0$ equilibrium
of the unperturbed system (\ref{eq:unperturbed nonlinear system})
in the forced system (\ref{eq: nonlinear system}): they are both
unique uniformly bounded solutions in the ball $B_{\delta}$ for their
respective systems. Note that if $f_{0}$ and $f_{1}$ are $C^{1}$
in $x$, then the Lipschitz constants in the inequalities (\ref{eq:conditions for x^*})
can be chosen as 
\[
L_{0}(\delta)=\max_{x\in B_{\delta}}\left|D_{x}f_{0}(x)\right|,\quad L_{1}(\delta)=\max_{x\in B_{\delta},t\in\mathbb{R}}\left|D_{x}f_{1}(x,t)\right|.
\]

A unique hyperbolic anchor trajectory $x^{*}(t)$ may well exist even
if the strict bounds listed in (\ref{eq:conditions for x^*}) are
not satisfied. For this reason, in the following theorem, we will
simply assume that such a trajectory $x^{*}(t)$ exists as a perturbation
of the $x=0$ fixed point and provide a recursively implementable,
formal approximation for $x^{*}(t)$ up to any desired order. These
formulas will only assume the uniform-in-time boundedness of $f_{1}$
and its derivatives with respect to $x$ at $x=0$, without assuming
the specific bounds on $f_{1}$ listed in Theorem \ref{thm:x^* under conditions}.
In particular, the non-autonomous term $f_{1}$ does not even have
to be continuous in time for these formulas to be well-defined. This
will enable predictions for anchor trajectories their SSMs in physical
systems even under temporally discontinuous forcing terms or under
specific realizations of bounded random forcing. 

To state our recursively computable formulas for $x^{*}(t)$ in our
upcoming Theorem \ref{thm: expansion for x^*}, we will use a matrix
$T\in\mathbb{R}^{n\times n}$ whose columns comprise the real and
imaginary parts of the eigenvectors of $A$. We order these columns
in a way that $T$ block-diagonalizes $A$ into a stable block $A_{s}$
and unstable block $A_{u}$:
\begin{equation}
T^{-1}AT=\left(\begin{array}{cc}
A_{s} & 0\\
0 & A_{u}
\end{array}\right),\quad\mathrm{spect}\left(A_{s}\right)=\mathrm{spect}\left(A\vert_{E^{s}}\right),\quad\mathrm{spect}\left(A_{s}\right)=\mathrm{spect}\left(A\vert_{E^{s}}\right).\label{eq:T definition}
\end{equation}
For later purposes, we also define the time-dependent matrix $G(t)\in\mathbb{R}^{n\times n}$
as 
\begin{equation}
G(t)=\left\{ \begin{array}{ccc}
T\left(\begin{array}{cc}
e^{A^{s}t} & 0\\
0 & 0
\end{array}\right)T^{-1}, &  & t\geq0,\\
\\
T\left(\begin{array}{cc}
0 & 0\\
0 & -e^{A^{u}t}
\end{array}\right)T^{-1}, &  & t<0.
\end{array}\right.\label{eq:Green's function definition}
\end{equation}
 Notice that if $x=0$ is either asymptotically stable ($E^{u}=\left\{ 0\right\} $
and $A^{s}=T^{-1}AT$) or repelling ($E^{s}=\left\{ 0\right\} $ and
$A^{u}=T^{-1}AT$), then formula (\ref{eq:Green's function definition})
simplifies to
\[
G(t)=\left\{ \begin{array}{ccc}
e^{At}, &  & t\geq0,\\
\\
0, &  & t<0,
\end{array}\right.\quad\mathrm{or}\quad G(t)=\left\{ \begin{array}{ccc}
0, &  & t\geq0,\\
\\
-e^{At}, &  & t<0,
\end{array}\right.
\]
respectively. Using these quantities, we obtain the following approximation
result for the anchor trajectory $x^{*}(t)$.
\begin{thm}
\label{thm: expansion for x^*}\textbf{\emph{{[}Computation of anchor
trajectory for non-autonomous SSMs{]}}} Assume that a unique, uniformly
bounded trajectory $x^{*}(t)\subset U$ of (\ref{eq: nonlinear system})
exists in a compact neighborhood $U\subset\mathbb{R}^{n}$ of $x=0$,
with $x^{*}(t)$ perturbing from $x=0$ under the forcing term $f_{1}(x,t)$
that satisfies (\ref{eq:f_1 uniformly bounded}). Assume further that
$f_{1}$ has $r\geq1$ continuous derivatives with respect to $x$
at $x=0$ that are uniformly bounded in $t$ within $U$. Then, for
any positive integer $N\leq r$, a formal expansion for $x^{*}(t)$
exists in the form
\begin{equation}
x^{*}(t)=\sum_{\nu=1}^{N}x_{\nu}(t)+o\left(\left\Vert f_{1}\right\Vert _{U}^{N}\right),\label{eq:anchor trajectory expansion}
\end{equation}
 where $x_{\nu}(t)=\left(x_{\nu}^{1}(t),\ldots,x_{\nu}^{n}(t)\right)=\mathcal{O}\left(\left\Vert f_{1}\right\Vert _{U}^{\nu}\right)$
and 
\end{thm}
\begin{align}
x_{\nu}(t)= & \sum_{1\leq\left|\mathbf{\mathbf{\boldsymbol{\gamma}}}\right|\leq\nu}\,\sum_{s=1}^{\nu}\sum_{p_{s}\left(\nu,\mathbf{\mathbf{\boldsymbol{\gamma}}}\right)}\int_{-\infty}^{\infty}G(t-\tau)\left[\frac{\partial^{\left|\mathbf{\mathbf{\boldsymbol{\gamma}}}\right|}f_{0}\left(0\right)}{\partial x_{1}^{\gamma_{1}}\cdots\partial x_{n}^{\gamma_{n}}}\prod_{j=1}^{s}\frac{\prod_{i=1}^{n}\left[x_{\ell_{j}}^{i}(\tau)\right]^{k_{ji}}}{\prod_{i=1}^{n}k_{ji}!}\right]d\tau\label{eq:x_nu hyperbolic case}\\
 & +\sum_{1\leq\left|\mathbf{\mathbf{\boldsymbol{\gamma}}}\right|\leq\nu-1}\,\sum_{s=1}^{\nu-1}\sum_{p_{s}\left(\nu-1,\mathbf{\mathbf{\boldsymbol{\gamma}}}\right)}\int_{-\infty}^{\infty}G(t-\tau)\left[\frac{\partial^{\left|\mathbf{\mathbf{\boldsymbol{\gamma}}}\right|}f_{1}\left(0,\tau\right)}{\partial x_{1}^{\gamma_{1}}\cdots\partial x_{n}^{\gamma_{n}}}\prod_{j=1}^{s}\frac{\prod_{i=1}^{n}\left[x_{\ell_{j}}^{i}(\tau)\right]^{k_{ji}}}{\prod_{i=1}^{n}k_{ji}!}\right]d\tau,\quad\nu\geq1.\nonumber 
\end{align}
\emph{Here $k_{ji}$ is the $i^{th}$ component of the integer vector
}\textbf{\emph{$\mathbf{k}_{j}\in\mathbb{N}^{n}-\left\{ \mathbf{0}\right\} $
}}\emph{appearing}\textbf{\emph{ }}\emph{in the index set}
\[
p_{s}\left(\nu,\mathbf{\mathbf{\boldsymbol{\gamma}}}\right)=\left\{ \left(\mathbf{k}_{1},\ldots,\mathbf{k}_{s},\ell_{1},\ldots,\ell_{s}\right):\mathbf{k}_{i}\in\mathbb{N}^{n}-\left\{ \mathbf{0}\right\} ,\ell_{i}\in\mathbb{N},0<\ell_{1}<\cdots<\ell_{s},\sum_{i=1}^{s}\mathbf{k}_{i}=\mathbf{\mathbf{\boldsymbol{\gamma}}},\sum_{i=1}^{s}\left|\mathbf{k}_{i}\right|\ell_{i}=\nu\right\} .
\]

\begin{proof}
See Appendix \ref{subsec:proof of approximation of x^*}.
\end{proof}
As an example of the approximation provided by Theorem \ref{thm: expansion for x^*}
for $x^{*}(t)$, we evaluate formula (\ref{eq:x_nu hyperbolic case})
up to second order ($N=2)$ for the case wherein the $x=0$ fixed
point is attracting ($E^{u}=\emptyset$) for $f_{1}\left(x,\tau\right)\equiv0$.
In that case, we obtain
\begin{align}
x^{*}(t) & =x_{1}(t)+x_{2}(t)+o\left(\left\Vert f_{1}\right\Vert _{U}^{2}\right),\label{eq:second-order approximation for x^*}\\
x_{1}(t) & =\int_{-\infty}^{t}e^{A\left(t-\tau\right)}f_{1}(0,\tau)\,d\tau,\nonumber \\
x_{2}(t) & =\int_{-\infty}^{t}e^{A\left(t-\tau\right)}\left[\frac{1}{2}\partial_{x}^{2}f_{0}(0)\otimes x_{1}(\tau)\otimes x_{1}(\tau)+\partial_{x}f_{1}(0,\tau)x_{1}(\tau)\right]d\tau,\nonumber 
\end{align}
where $\partial_{x}^{2}f_{0}(0)$ is a three-tensor and $\otimes$
refers to the tensor product.
\begin{rem}
\label{rem: temporally  dicontinuous forcing OK}{[}\textbf{Applicability
to temporally discontinuous forcing}{]} The asymptotic approximation
(\ref{eq:x_nu hyperbolic case}) only requires $f_{1}$ and its $x$-derivatives
at $x=0$ to be uniformly bounded in $t$, as we noted earlier. No
derivatives of $f_{1}(x,t)$ are required to exist with respect to
$t$, and hence temporally discontinuous forcing is also covered by
these formal expansions, as we will also see on a specific example
with discontinuous chaotic forcing in Section \ref{subsec:weakly forced cart}. 
\end{rem}
~~~
\begin{rem}
\label{rem: special case of external weak forcing}{[}\textbf{Simplification
for state-independent forcing}{]} In applications to structural vibrations,
the non-autonomous forcing term in system (\ref{eq: nonlinear system})
often arises from external forcing that does not depend on $x$, i.e.,
$f_{1}(x,t)\equiv f_{1}(t)$. In that case, the second term in formula
(\ref{eq:x_nu hyperbolic case}) is identically zero and the summands
in the first term are well defined as long as $f_{0}(x)$ an its derivatives
are bounded at $x=0$. 
\end{rem}

\subsection{Existence and computation of non-autonomous SSMs }

We first recall available SSM results for the autonomous system (\ref{eq:unperturbed nonlinear system}).
Let $E$ be a spectral subspace, as defined in eq. (\ref{eq:spectral subspace}).
Following the terminology of \citet{haller23}, we call $E$ a \emph{like-mode
spectral subspace }if the sign of $\mathrm{Re}\lambda_{j}$ is the
same for all $\lambda_{j}\in\mathrm{spect}\left(A\vert_{E}\right)$.
Otherwise, we call $E$ a \emph{mixed-mode spectral subspace. }

We call a like-mode spectral subspace $E$ \emph{externally nonresonant}
if 
\begin{equation}
\lambda_{j}\neq\sum_{\lambda_{k}\in\mathrm{spect}\left(A\vert_{E}\right)}m_{k}\lambda_{k},\quad\lambda_{j}\in\mathrm{spect}\left(A\right)-\mathrm{spect}\left(A\vert_{E}\right),\label{eq:external non-resonance condition}
\end{equation}
for any choice of $m_{k}\in\mathbb{N}$ with $\sum_{k=1}^{n}m_{k}\geq2$
and for any choice of $\lambda_{j}$. It turns out that if the condition
(\ref{eq:external non-resonance condition}) is satisfied for $m_{k}$
coefficients up to order $\sum_{k=1}^{n}m_{k}=\sigma\left(E\right)$,
where $\sigma\left(E\right)$ is the spectral quotient defined by
\citet{haller16}, and then the same condition will also be satisfied
for \emph{all} choices of $m_{k}$ with $\sum_{k=1}^{n}m_{k}\geq2$
(\citet{cabre03}). 

If $E$ is an externally nonresonant, like-mode subspace, then $E$
has a unique, smoothest nonlinear continuation in the form of a primary
spectral submanifold (or \emph{primary SSM}), denoted $\mathcal{W}^{\infty}\left(E\right)\in C^{\infty}$.
This was deduced by \citet{haller16} in this context from the more
abstract, general results of \citet{cabre03}. The primary SSM $\mathcal{W}^{\infty}\left(E\right)$
is an invariant manifold of system (\ref{eq:unperturbed nonlinear system})
that is tangent to $E$ at the origin and has the same dimension as
$E$, as shown in Fig. \ref{fig:Geometry of non-autonomous SSM} for
$f_{1}(x,t)\equiv0$ in the extended phase space.
\begin{figure}[H]
\centering{}\includegraphics[width=0.8\textwidth]{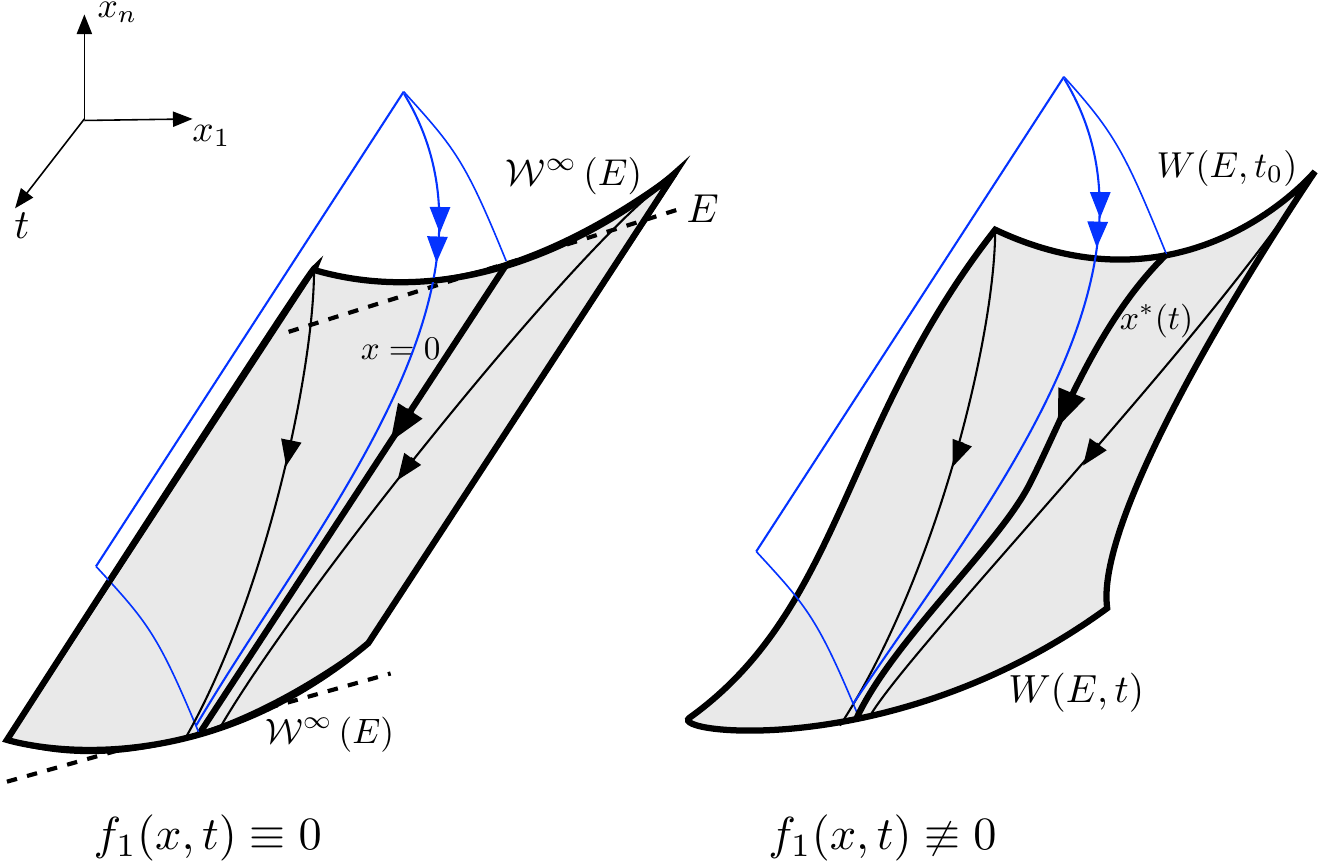}\caption{Left: The geometry of the autonomous SSM $\mathcal{W}^{\infty}\left(E\right)$
in the extended phase space of the $(x,t)$ variables. Shown is the
primary (smoothest SSM) tangent to the spectral subspace $E$ at the
fixed point at $x=0$. Also shown is a fast SSM (blue), defined as
the spectral submanifold that is tangent to the direct sum of the
eigenspaces outside $E$. Right: The anchor trajectory $x^{*}(t)$
and the time-dependent SSM $W\left(E,t\right)$ in the full, non-autonomous
system. \label{fig:Geometry of non-autonomous SSM}}
\end{figure}
In general, an infinite family, $\mathcal{W}\left(E\right)$, of SSMs
exists for system (\ref{eq:unperturbed nonlinear system}) with the
same properties, but $\mathcal{W}^{\infty}\left(E\right)$ is the
unique smoothest among them: the other manifolds in $\mathcal{W}\left(E\right)$
are no more than $\sigma\left(E\right)$-times differentiable (see
\citet{haller23}). A computational algorithm for $\mathcal{W}^{\infty}\left(E\right)$
was developed by \citet{ponsioen2018}, then extended to general,
finite-element-grade problems in second-order ODE form by \citet{jain2022}.
The latest version of the latter algorithm with further extensions
is available in the open-source MATLAB live script package \texttt{SSMTool}
(see \citet{jain23}).

If the spectrum of $A$ is \emph{fully nonresonant}, i.e., 

\begin{equation}
\lambda_{j}\neq\sum_{k=1}^{n}m_{k}\lambda_{k},\quad m_{k}\in\mathbb{N},\quad\sum_{k=1}^{n}m_{k}\geq2,\quad j=1,\ldots,n,\label{eq:full nonresonance conditions}
\end{equation}
then an invariant manifold family $\mathcal{W}\left(E\right)$ tangent
to $E$ at $x=0$ exists for any choice of $E$, whether or not $E$
is of like-mode or mixed mode type. While $\mathcal{W}\left(E\right)$
may contain just a single manifold (e.g., when $E$ is the stable
subspace, $E^{s}$ or the unstable subspace, $E^{u}$, defined in
(\ref{eq:E^s and E^u})), there will be infinitely members in $\mathcal{W}\left(E\right)$
for more general choices of $E$. In the latter case, generic members
of $\mathcal{W}\left(E\right)$ are \emph{secondary }(or\emph{ fractiona}l)
\emph{SSMs}, which have a finite order of differentiability consistent
with their fractional-powered polynomial representations derived explicitly
by \citet{haller23}.

Note that the nonresonance conditions (\ref{eq:external non-resonance condition})-(\ref{eq:full nonresonance conditions})
are less restrictive than what is customary in the nonlinear vibrations
literature. Indeed, conditions (\ref{eq:external non-resonance condition})-(\ref{eq:full nonresonance conditions})
are only violated if both the real and the complex parts of the eigenvalues
satisfy simultaneously exactly the same resonance relationship. Also
note that $1:1$ resonances involving eigenvalues with nonzero real
parts do \emph{not} violate (\ref{eq:external non-resonance condition})-(\ref{eq:full nonresonance conditions}).
We will, however, only allow a $1\colon1$ resonance within $E$ to
guarantee the normal hyperbolicity of $E$ for the existence of $\mathcal{W}\left(E\right)$.
(More specifically, a $1\colon1$ resonance between the spectrum of
$A$ within $E$ and outside $E$ would render our upcoming assumption
(\ref{eq:spectral subspace-1}) to fail for any $\rho>1$). If a $1\colon1$
resonance arises in a given application, one can simply enlarge $E$
to include all resonant modes. This, in turn, removes any issue with
the $1\colon1$ resonance.

We now turn to the existence of non-autonomous SSMs emanating from
the uniformly bounded hyperbolic trajectory $x^{*}(t)$, whose existence
and approximation were discussed in Theorems \ref{thm:x^* under conditions}
and \ref{thm: expansion for x^*}. We will focus on \emph{slow SSMs}
(also called \emph{pseudo-unstable manifolds}), which are continuations
of $d$-dimensional, $\rho$-\emph{normally attracting} (like-mode
or mixed-mode) spectral subspaces $E$ of the linearized system (\ref{eq:linearized ODE}),
i.e., can be written as
\begin{equation}
E=E_{1}\oplus\ldots\oplus E_{k},\qquad\dim E=d\geq1,\qquad\mathrm{\frac{\mathrm{Re}\text{\ensuremath{\lambda_{k}}}}{\mathrm{Re}\text{\ensuremath{\lambda_{k+1}}}}\leq\frac{1}{\rho}},\label{eq:spectral subspace-1}
\end{equation}
for some integer $\rho>1$. Therefore, $E$ is spanned by the $k$
modal subspaces carrying the slowest decaying solution families of
system (\ref{eq:linearized ODE}), including possibly some families
that do not even decay but grow. If such unstable modal subspaces
are present in the direct sum (\ref{eq:spectral subspace-1}), then
$E$ is a mixed-mode spectral subspace which we seek to continue into
a mixed mode non-autonomous SSM in system \eqref{eq: nonlinear system}.
If, in contrast, only stable modal subspaces are present in the direct
sum (\ref{eq:spectral subspace-1}), then $E$ is a like-mode spectral
subspace which we seek to continue into a like-mode non-autonomous
SSM in system \eqref{eq: nonlinear system}. 

Note that the third condition in (\ref{eq:spectral subspace-1}) always
holds for arbitrary $\rho>1$ if $\mathrm{Re}\lambda_{k}$ and $\mathrm{Re}\lambda_{k+1}$
have different signs. If $\mathrm{Re}\lambda_{k}$ and $\mathrm{Re}\lambda_{k+1}$
have the same sign, then the third condition in (\ref{eq:spectral subspace-1})
holds only if the decay exponents of solutions outside $E$ are at
least $\rho$-times stronger than those of solutions inside $E$.
This $\rho$ will then determine the maximal smoothness that we can
\emph{a priori} guarantee for the SSM emanating from $E$ without
further assumptions. Formal expansions for such SSMs will, however,
indicate higher degrees of smoothness under appropriate nonresonance
conditions.

The following theorem gives our main result on the existence of non-autonomous
SSMs associated with a spectral subspace $E$ satisfying (\ref{eq:spectral subspace-1}).
We will use the term \emph{locally invariant manifold }when referring
to a manifold carrying trajectories that can only leave the manifold
through its boundary (see, e.g., \citet{fenichel71}).
\begin{thm}
\label{thm:non-automous SSM}\textbf{\emph{{[}Existence of non-autonomous
SSMs{]}}} Assume that conditions (\ref{eq:conditions for x^*}) of
Theorem \ref{thm:x^* under conditions} are satisfied in a ball $B_{\delta}\subset\mathbb{R}^{n}$
of radius $\delta>0$ around $x=0$. Assume further that $f_{0},f_{1}\in C^{r}(B_{\delta})$
for some $r>1$ and the uniform bound
\begin{equation}
\left|\partial_{x}^{2}f_{1}(x,t)\right|\leq K_{3},\quad x\in B_{\delta},\quad t\in\mathbb{R},\label{eq:Hessian of f_1}
\end{equation}
holds for some finite constant $K_{3}>0$. Assume finally that $E$
is a $\rho$-normally hyperbolic spectral subspace with $1<\rho\leq r$
that satisfies the nonresonance conditions
\end{thm}
\begin{equation}
\lambda_{j}\neq\sum_{k=1}^{n}m_{k}\lambda_{k},\quad m_{k}\in\mathbb{N},\quad\sum_{k=1}^{n}m_{k}\geq2,\quad j=1,\ldots,n.\label{eq:full nonresonance conditions-1}
\end{equation}
\emph{Then:}
\begin{description}
\item [{\emph{(i)}}] \emph{There exists a non-autonomous spectral submanifold
$W(E,t)\subset B_{\delta}$ of class $C^{\rho}$ that has the same
dimension as $E$, contains $x^{*}(t)$ and acts as a locally invariant
manifold for system \eqref{eq: nonlinear system}. }
\item [{\emph{(ii)}}] \emph{The SSM $W(E,t)$ is as smooth in any parameter
as system (\eqref{eq: nonlinear system}).}
\end{description}
\begin{proof}
See Appendix \ref{subsec:Proof of existence of SSM}.
\end{proof}
We sketch the geometry of the non-autonomous SSM, $W(E,t)$, for non-vanishing
$f_{1}(x,t)$ in Fig. \ref{fig:Geometry of non-autonomous SSM}.
\begin{rem}
\textbf{{[}Uniqueness and smoothness of non-autonomous SSMs}{]} The
manifolds discussed in Theorem \ref{thm:non-automous SSM} are generally
non-unique. In principle, an argument involving the smoothness of
the Lyapunov and dichotomy spectra under generic perturbation (see
\citet{son17}) could be invoked to deduce sharper smoothness results
for these manifolds from linearization (see \citet{yomdin88}). These
results, however, require the identification of the full Lyapunov
spectrum of the linear part of the non-autonomous linear differential
equation (\ref{eq:perturbed, transformed}) in our Appendix \ref{subsec:Proof of existence of SSM}
in order to exclude resonances. This is generally challenging for
equations and unrealistic for data sets. As an alternative, Theorem
\ref{thm:computation of non-automous SSM} below will provide unique
Taylor expansions for $W(E,t)$ under explicitly verifiable nonresonance
conditions. These expansions are varied on each persisting non-autonomous
SSM up to the degree of smoothness of the SSM. We conjecture that,
just as in the time-periodic and time-quasiperiodic case treated by
\citet{cabre03} and \citet{haro06}, a unique member of the $W(E,t)$
family of manifolds will be smoother than all the others. Our Taylor
expansion will then approximate those unique, smoothest manifolds
at orders higher than the spectral quotient $\Sigma(E)$ defined in
\citet{haller16}. The remaining manifolds can be constructed via
time-dependent versions of the fractional expansions identified in
\citet{haller23}.
\end{rem}
\begin{rem}
\textbf{{[}Non-autonomous SSMs without anchor trajectories{]}} For
stable hyperbolic fixed points in the $f_{1}(x,t)\equiv0$ limit,
an alternative to the proof in Appendix \ref{subsec:Proof of existence of SSM}
for Theorem \ref{thm:non-automous SSM} can also be given. This alternative
uses the theory of non-compact normally hyperbolic invariant manifolds
(see \citet{eldering13}) coupled with the ``wormhole'' construct
of \citet{eldering18} that enables the handling of inflowing-invariant
normally attracting invariant manifolds. Following the steps of our
proof in Appendix \ref{subsec:Existence-of-the-adiabatic-SSM} for
the adiabatic case, this alternative proof yields results similar
to those in Theorem \ref{thm:non-automous SSM} but does not rely
on a persisting anchor trajectory $x^{*}(t)$ near the origin. As
a consequence, it can also capture non-autonomous SSMs in weakly damped
physical systems for higher forcing levels at which $x^{*}(t)$ is
already destroyed. This is because the strength of hyperbolicity of
the unforced fixed point at $x=0$ (measured by $\left|\mathrm{Re}\text{\ensuremath{\lambda_{1}}}\right|\ll1$)
is generally much weaker than the strength of hyperbolicity of the
unforced SSM, $\mathcal{W}^{\infty}\left(E\right)$, (measured by
$\frac{\left|\mathrm{Re}\text{\ensuremath{\lambda_{k+1}}}\right|}{\left|\mathrm{Re}\text{\ensuremath{\lambda_{k}}}\right|}>1$)
in weakly damped systems. Such persisting SSMs without a stable hyperbolic
anchor trajectory have been well-documented in equation- and data-driven
studies of time-periodically forced systems. Those SSMs are signaled
by overhangs near resonances in the forced response curves at higher
forcing levels (see, e.g., \citet{jain2022,cenedese22a}).
\end{rem}
As our examples will show, $W(E,t)$ will generally persist even under
$f_{1}(x,t)$ perturbations that are significantly larger than those
allowed by the rather conservative assumptions of Theorem \ref{thm:x^* under conditions}.
In addition, $W(E,t)$ will also be smoother than $C^{\rho}$ under
additional nonresonance conditions. To this end, we will next derive
numerically implementable, recursive approximation formulas that are
valid for $W(E,t)$ as long as it persists.

To state these approximations, we first introduce a small perturbation
parameter $\epsilon\geq0$ with which we rescale the non-autonomous
term in (\eqref{eq: nonlinear system} as 
\begin{equation}
f_{1}(x,t)=\epsilon\tilde{f}_{1}(x,t),\label{eq:rescaling of f_1}
\end{equation}
in order to focus on moderate values of $f_{1}(x,t)$. As a consequence
of this scaling, the expansion (\ref{eq:anchor trajectory expansion})
for the anchor trajectory is also rescaled to
\begin{equation}
x_{\epsilon}^{*}(t)=\sum_{\nu=1}^{N}\epsilon^{\nu}\tilde{x}_{\nu}(t)+o\left(\epsilon^{N}\right),\label{eq:e dependent expansion for anchor}
\end{equation}
given that the form of the coefficients $x_{\nu}(t)$ in the formulas
(\ref{eq:x_nu hyperbolic case}) yields 
\begin{equation}
x_{\nu}(t)=\epsilon\tilde{x}_{\nu}(t).\label{eq:scaled x_nu}
\end{equation}

Second, we let $P=[e_{1},\ldots,e_{n}]\in\mathbb{C}^{n}$ contain
the complex unit eigenvectors corresponding to the ordered eigenvalues
(\ref{eq:eigenvalue ordering}) of $A$. Then, under a coordinate
change $x\mapsto(u,v)\in\mathbb{C}^{d}\times\mathbb{C}^{n-d}$ defined
as
\begin{equation}
\left(\begin{array}{c}
u\\
v
\end{array}\right)=P^{-1}\left(x-x_{\epsilon}^{*}(t)\right),\label{eq:coordinate change to (u,v)}
\end{equation}
we obtain a complex system of ODEs 
\begin{align}
\left(\begin{array}{c}
\dot{u}\\
\dot{v}
\end{array}\right) & =\left(\begin{array}{cc}
A^{u} & 0\\
0 & A^{v}
\end{array}\right)\left(\begin{array}{c}
u\\
v
\end{array}\right)+\hat{f}(u,v,\epsilon;t),\label{eq:(u,v) system-1}
\end{align}
where
\begin{align}
 & A^{u}=\left(\begin{array}{ccc}
\lambda_{1} & 0 & 0\\
0 & \ddots & 0\\
0 & 0 & \lambda_{d}
\end{array}\right),\quad A^{v}=\left(\begin{array}{ccc}
\lambda_{d+1} & 0 & 0\\
0 & \ddots & 0\\
0 & 0 & \lambda_{n}
\end{array}\right),\label{eq:A^u and A^v definition}
\end{align}
and
\begin{align}
\hat{f}(u,v,\epsilon;t) & =P^{-1}\left[f_{0}\left(x_{\epsilon}^{*}(t)+P\left(\begin{array}{c}
u\\
v
\end{array}\right)\right)+Ax_{\epsilon}^{*}(t)-\dot{x}_{\epsilon}^{*}(t)+\epsilon\tilde{f}_{1}\left(x_{\epsilon}^{*}(t)+P\left(\begin{array}{c}
u\\
v
\end{array}\right),t\right)\right].\label{eq:fhat definition-2}
\end{align}
In system (\ref{eq:(u,v) system-1}), the fixed point $(u,v)=(0,0)$
corresponds to the anchor trajectory $x_{\epsilon}^{*}(t)$ and the
$u$ coordinate space is aligned with the spectral subspace $E$.

Third, using the integer multi-index $\left(\mathbf{k},p\right)\in\mathbb{N}^{d}\times\mathbb{N}$
with $\left|\left(\mathbf{k},p\right)\right|\geq0$, we define a (now
$\epsilon$-dependent) order-$\left|\left(\mathbf{k},p\right)\right|$
approximation to $x_{\epsilon}^{*}(t)$ and the evaluation of $\hat{f}(u,v,\epsilon;t)$
on this approximation as 
\begin{equation}
x_{\epsilon}^{*}(t;\mathbf{k},p)=\sum_{\nu=0}^{\left|\mathbf{k}\right|+p-1}\epsilon^{\nu}x_{\nu}(t),\qquad\hat{f}(u,v,\epsilon;t;\mathbf{k},p)=\left(\begin{array}{c}
\hat{f}^{u}(u,v,\epsilon;t;\mathbf{k},p)\\
\hat{f}^{v}(u,v,\epsilon;t;\mathbf{k},p)
\end{array}\right)\in\mathbb{C}^{d}\times\mathbb{C}^{n-d},\label{eq:x_epsilon kp}
\end{equation}
 respectively, where 
\begin{align}
\hat{f}(u,v,\epsilon;t;\mathbf{k},p) & =P^{-1}\left[f_{0}\left(x_{\epsilon}^{*}(t;\mathbf{k},p)+P\left(\begin{array}{c}
u\\
v
\end{array}\right)\right)+Ax_{\epsilon}^{*}(t;\mathbf{k},p)\right.\nonumber \\
 & \,\,\,\,\,\,\,\,\,\,\,\,\,\,\,\,\,\,-\left.\dot{x}_{\epsilon}^{*}(t;\mathbf{k},p)+\epsilon\tilde{f}_{1}\left(x_{\epsilon}^{*}(t;\mathbf{k},p)+P\left(\begin{array}{c}
u\\
v
\end{array}\right),t\right)\right].\label{eq:f hat definition}
\end{align}
In short, $x_{\epsilon}^{*}(t;\mathbf{k},p)$ is an approximation
of $x^{*}(t)$ in the scaled variable $\epsilon x$ up to order $\left|\mathbf{k}\right|+p-1$.
Accordingly, $\hat{f}(u,v,\epsilon;t;\mathbf{k},p)$ is an approximation
of $\hat{f}(u,v,\epsilon;t)$ that uses $x_{\epsilon}^{*}(t;\mathbf{k},p)$
instead of $x^{*}(t)$. 

Finally, for any multi-index $\mathbf{k}\in\mathbb{N}^{d}$, we define
the time-dependent diagonal matrix $G_{\mathbf{k}}(t)\in\mathbb{C}^{\left(n-d\right)\times\left(n-d\right)}$
as
\begin{align}
\left[G_{\mathbf{k}}(t)\right]_{\ell\ell} & :=\left\{ \begin{array}{ccc}
e^{\left[\lambda_{\ell}-\sum_{j=1}^{d}k_{j}\lambda_{j}\right]t}, &  & \,\,\,\,\mathrm{Re}\left[\lambda_{\ell}-\sum_{j=1}^{d}k_{j}\lambda_{j}\right]<0,\\
\\
0, &  & \mathrm{otherwise},
\end{array}\right.\quad t\geq0,\nonumber \\
 & \,\,\,\,\label{eq:G_k definition}\\
\left[G_{\mathbf{k}}(t)\right]_{\ell\ell} & :=\left\{ \begin{array}{ccc}
-e^{\left[\lambda_{\ell}-\sum_{j=1}^{d}k_{j}\lambda_{j}\right]t}, &  & \mathrm{Re}\left[\lambda_{\ell}-\sum_{j=1}^{d}k_{j}\lambda_{j}\right]>0,\\
\\
0, &  & \mathrm{otherwise},
\end{array}\right.\quad t<0,\nonumber 
\end{align}
for $\ell=d+1,\ldots,n.$ Using these quantities, we can state the
following results.
\begin{thm}
\label{thm:computation of non-automous SSM}\textbf{\emph{{[}Computation
of non-autonomous SSMs and their reduced dynamics{]}}} Assume that
after the rescaling (\ref{eq:rescaling of f_1}), an $\epsilon$-dependent
SSM, $W_{\epsilon}(E,t)$, of the type described in Theorem \ref{thm:non-automous SSM}
exists in the coordinates $(u,v)$ for system (\ref{eq:(u,v) system-1})
for all $\epsilon\in\left[0,\epsilon^{*}\right]$. Assume further
that $W_{\epsilon}(E,t)$ is $N$-times continuously differentiable
for any fixed $t$ and the nonresonance conditions 
\begin{equation}
\mathrm{Re}\,\lambda_{j}\neq\sum_{\lambda_{k}\in\mathrm{spect}\left(A\vert_{E}\right)}m_{k}\mathrm{Re}\,\lambda_{k},\qquad\lambda_{j}\in\mathrm{spect}\left(A\right)-\mathrm{spect}\left(A\vert_{E}\right),\label{eq:strict nonresonance}
\end{equation}
hold for all $j=k+1,\ldots,n$ and for all $m_{k}\in\mathbb{N}$ with
\begin{equation}
1\leq\sum_{k=1}^{n}m_{k}\leq N.\label{eq:strenghtened nonresonance under epsilon perturbation}
\end{equation}
. 

Then, for all $\epsilon\in\left[0,\epsilon^{*}\right]$\emph{:}
\end{thm}
\begin{description}
\item [{\emph{(i)}}] \emph{The SSM $W_{\epsilon}(E,t)$ admits a formal
asymptotic expansion 
\begin{equation}
W_{\epsilon}\left(E,t\right)=\left\{ \left(u,v\right)\in U\subset\mathbb{R}^{n}:\,v=h_{\epsilon}(u,t)=\sum_{\left|\left(\mathbf{k},p\right)\right|\geq1}^{N}h^{\mathbf{k}p}(t)u^{\mathbf{k}}\epsilon^{p}+o\left(\left|u\right|^{q}\epsilon^{N-q}\right)\right\} .\label{eq:h double expansion-2}
\end{equation}
The uniformly bounded $h^{\mathbf{k}p}(t)$ coefficients in this expansion
can be computed recursively from their initial conditions 
\begin{equation}
h^{\mathbf{0}p}(t)\equiv0,\quad t\in\mathbb{R},\quad p\in\mathbb{N};\qquad h^{\mathbf{k}0}(t)\equiv-A_{\mathbf{k}}^{-1}M^{\mathbf{k}0}(h^{\mathbf{j}0}),\quad\left|\mathbf{j}\right|<\left|\mathbf{k}\right|;\quad h^{\mathbf{0}0}=0,\label{eq:oelementary identities about indices-1-1}
\end{equation}
via the formula
\begin{equation}
h^{\mathbf{k}p}(t)=\int_{-\infty}^{\infty}G_{\mathbf{k}}(t-s)M^{\mathbf{k}p}(s,h^{\mathbf{j}m}(s))\,ds,\quad\left|(\mathbf{j},m)\right|<\left|(\mathbf{k},p)\right|,\quad t\in\mathbb{R},\label{eq:h_k formula-1}
\end{equation}
where the functions $M^{\mathbf{k}p}$ are defined as
\[
M^{\mathbf{k}p}(t,h^{\mathbf{j}m})=\frac{\partial^{\left|\left(\mathbf{k},p\right)\right|}}{\partial u_{1}^{k_{1}}\ldots\partial u_{d}^{k_{d}}\partial\epsilon^{p}}\left[\hat{f}^{v}\left(u,\sum_{\left|\left(\mathbf{j},m\right)\right|\geq1}^{\left|\left(\mathbf{k},p\right)\right|-1}h^{\mathbf{j}m}(t)u^{\mathbf{j}}\epsilon^{m},\epsilon;t;\mathbf{k},p\right)\right.
\]
\begin{equation}
-\left.\left.\sum_{\left|\left(\mathbf{j},m\right)\right|\geq1}^{\left|\left(\mathbf{k},p\right)\right|-1}\epsilon^{p}\left(\begin{array}{ccc}
h_{1}^{\mathbf{j}m}(t)\frac{j_{1}u^{\mathbf{j}}}{u_{1}} & \cdots & h_{1}^{\mathbf{j}m}(t)\frac{j_{d}u^{\mathbf{j}}}{u_{d}}\\
\vdots & \ddots & \vdots\\
h_{n-d}^{\mathbf{j}m}(t)\frac{j_{1}u^{\mathbf{j}}}{u_{1}} & \cdots & h_{n-d}^{\mathbf{j}m}(t)\frac{j_{d}u^{\mathbf{j}}}{u_{d}}
\end{array}\right)\hat{f}^{u}\left(u,\sum_{\left|\left(\mathbf{j},m\right)\right|\geq1}^{\left|\left(\mathbf{k},p\right)\right|-1}h^{\mathbf{j}m}(t)u^{\mathbf{j}}\epsilon^{m},\epsilon;t;\mathbf{k},p\right)\right]\right|_{u=0,\,\epsilon=0}.\label{eq:M^k-1-2-1}
\end{equation}
}
\item [{\emph{(ii)}}] \emph{The reduced dynamics on $W_{\epsilon}\left(E,t\right)$
is obtained by restricting the $u$-component of system }(\ref{eq:(u,v) system-1})\emph{
to $W_{\epsilon}\left(E,t\right)$, which yields 
\begin{align}
\dot{u} & =A^{u}u\label{eq:final reduced dynamics-2}\\
 & +Q_{u}\left[f_{0}\left(x_{\epsilon}^{*}(t)+P\left(\begin{array}{c}
u\\
h_{\epsilon}(u,t)
\end{array}\right)\right)+\epsilon\tilde{f}_{1}\left(x_{\epsilon}^{*}(t)+P\left(\begin{array}{c}
u\\
h_{\epsilon}(u,t)
\end{array}\right),t\right)+Ax_{\epsilon}^{*}(t)-\dot{x}_{\epsilon}^{*}(t)\right].\nonumber 
\end{align}
Here $Q_{u}\in\mathbb{C}^{d\times n}$ is a matrix whose $j^{th}$
row is $\hat{e}_{j}/\left(\hat{e}_{j}\cdot e_{j}\right)$ for $j=1,\ldots d$,
where $\hat{e}_{j}$ is the $j^{th}$ unit left eigenvector of $P$
corresponding to its unit right eigenvector $e_{j}$. Equivalently,
$Q_{u}$ is composed of the first $d$ rows of $P^{-1}$. }
\item [{(iii\emph{)}}] \emph{Let $\left(\xi,\eta\right)^{\mathrm{T}}=P^{-1}x$
denote coordinates aligned with the subspace $E$ and the direct sum
of the remaining eigenspaces, respectively, emanating from the original
$x=0$ fixed point of system (\ref{eq: nonlinear system}) for $\epsilon=0.$
In these coordinates, the reduced dynamics (\ref{eq:final reduced dynamics-2})
becomes}
\begin{align}
\dot{\xi} & =A^{u}\xi\label{eq:dfinal reduced dynamics in xi coordinates}\\
 & +Q_{u}\left[f_{0}\left(x_{\epsilon}^{*}(t)+P\left(\begin{array}{c}
\xi-Q_{u}x_{\epsilon}^{*}(t)\\
h_{\epsilon}(\xi-Q_{u}x_{\epsilon}^{*}(t),t)
\end{array}\right)\right)+\epsilon\tilde{f}_{1}\left(x_{\epsilon}^{*}(t)+P\left(\begin{array}{c}
\xi-Q_{u}x_{\epsilon}^{*}(t)\\
h_{\epsilon}(\xi-Q_{u}x_{\epsilon}^{*}(t),t)
\end{array}\right),t\right)\right].\nonumber 
\end{align}
\end{description}
\begin{proof}
See Appendix \ref{subsec:  Proof of approximation of SSM}.
\end{proof}
\begin{rem}
{[}\textbf{Related results}{]} \citet{potzsche06} derive Taylor approximations
for a broad class of invariant manifolds emanating from fixed points
of non-autonomous ODE. Due to the generality of that setting, the
resulting formulas are less explicit than ours, assume global boundedness
on the nonlinear terms, and also assume that the origin remains a
fixed point for all times (and hence disallow additive external forcing).
Nevertheless, under further assumptions and appropriate reformulations,
those formulas should ultimately yield results equivalent to ours
when applied to SSMs. Although it does not directly overlap with our
results here, we also mention related work by \citet{potzsche09}
on the computation of classic and strong stable/unstable manifolds,
as well as center-stable/unstable manifolds, for non-autonomous difference
equations. 
\end{rem}
\begin{rem}
{[}\textbf{Extension to discontinuous forcing}{]} \label{rem:discontonuous forcing}
The derivatives in the definition of (\ref{eq:M^k-1-2-1}) can be
evaluated using the same higher-order chain-rule formula of \citet{constantine96}
that we applied in the proof of Theorem \ref{thm: expansion for x^*}.
Using these formulas reveal that $h^{\mathbf{k}p}(t)$ can be formally
computed as long as $f_{1}$ and its $x$-derivatives at $x=0$ remain
uniformly bounded in $t$. Therefore, the formal expansion we have
derived for the SSM $W_{\epsilon}\left(E,t\right)$ is valid for small
enough $\epsilon$ even if $f_{1}$ is discontinuous in the time $t$.
We illustrate the continued accuracy of these expansions for discontinuous
forcing on an example in Section \ref{subsec:weakly forced cart}.
\end{rem}
\begin{rem}
{[}\textbf{Relation to the case of periodic or quasi-periodic forcing}{]}
\label{rem:periodic/quasiperiodic forcing} The recursively defined
coefficient vectors $M^{\mathbf{k}p}(s,h^{\mathbf{j}m}(s))$ defined
in eq. (\ref{eq:M^k-1-2-1}) are explicitly time dependent but otherwise
satisfy the same formulas as the terms on the right-hand side of the
invariance equations arising from autonomous SSM calculations. Consequently,
the recursive formulas originally implemented in \texttt{SSMTool by}
\citet{jain23} for autonomous SSM calculations apply here without
modification. The difference is that those autonomous $M^{\mathbf{k}}$
coefficients are used in solving linear algebraic systems of equations
for $h^{\mathbf{k}}$, as opposed to the non-autonomous $M^{\mathbf{k}p}$
coefficients are used in linear ODEs for $h^{\mathbf{k}p}(t)$. 
\end{rem}
\begin{rem}
{[}\textbf{Accuracy of asymptotic SSM formulas}{]} \label{rem:accuracy of nonautonomous SSM formulas}
The accuracy of the reduced-order dynamics (\ref{eq:final reduced dynamics-2})
increases with increasing $\left|(\mathbf{k},p)\right|$. There is
no general convergence result under $N\to\infty$ in the reduced dynamics
(\ref{eq:final reduced dynamics-2}), but such convergence is known
if $f(x,t)$ is analytic and has periodic or quasi-periodic time dependence
\citet{cabre03,haro06}). For those types of time dependencies, the
$\mathcal{O}\left(\epsilon\right)$ term in the reduced dynamics (\ref{eq:final reduced dynamics-2})
is the same as that used for time-periodic and time-quasiperiodic
SSMs in earlier publications (see, e.g., \citet{breunung2018,cenedese22b}). 
\end{rem}
By formula (\ref{eq:dfinal reduced dynamics in xi coordinates}),
in the $\xi$ coordinate aligned with $E$ and emanating from $x=0$,
the reduced dynamics up to first order in $\epsilon$ is of the form
\begin{align}
\dot{\xi} & =A^{u}\xi+Q_{u}f_{0}\left(P\left(\begin{array}{c}
\xi\\
h_{0}(\xi)
\end{array}\right)\right)\nonumber \\
 & +\epsilon Q_{u}\left[Df_{0}\left(P\left(\begin{array}{c}
\xi\\
h_{0}(\xi)
\end{array}\right)\right)\left(\int_{-\infty}^{t}e^{A\left(t-\tau\right)}\tilde{f}_{1}(0,\tau)\,d\tau-P\left(\begin{array}{c}
Q_{u}\int_{-\infty}^{t}e^{A\left(t-\tau\right)}\tilde{f}_{1}(0,\tau)\,d\tau\\
\partial_{\epsilon}h_{\epsilon}(\xi,t)\vert_{\epsilon=0}
\end{array}\right)\right)\right]\nonumber \\
 & +\epsilon Q_{u}\tilde{f}_{1}\left(P\left(\begin{array}{c}
\xi\\
h_{0}(\xi)
\end{array}\right),t\right)\nonumber \\
 & +o\left(\epsilon\right),\label{eq:leading-order xi equation}
\end{align}
where we used formulas (\ref{eq:second-order approximation for x^*})
and (\ref{eq:scaled x_nu}) in evaluating $\tilde{x}_{1}(t)$, the
first coefficient in the expansion of $x_{\epsilon}^{*}(t)$ in eq.
(\ref{eq:e dependent expansion for anchor}). 

In prior treatments of time-periodically and time-quasi-periodically
forced SSMs, the second line on the right-hand side of eq. (\ref{eq:leading-order xi equation})
has been justifiably dropped in leading-order approximations, because
the factor $Df_{0}\left(P\left(\begin{array}{c}
\xi\\
h_{0}(\xi)
\end{array}\right)\right)$ is small close to the origin (see, e.g., \citet{breunung2018,ponsioen2020,jain2022,cenedese22a}).
In that case, the leading-order correction to the unforced reduced
dynamics is simply the third line in (\ref{eq:leading-order xi equation}),
which is just the projection of the restriction of the forcing term
to the autonomous SSM onto the spectral subspace $E$. If, in addition,
the forcing term only depends on time (i.e., $\partial_{x}\tilde{f}_{1}\equiv0$),
as is often the case in structural vibrations, then a leading-order
approximation for small-amplitude motions on the SSM \emph{$W_{\epsilon}\left(E,t\right)$}
becomes
\begin{align}
\dot{\xi} & =A^{u}\xi+Q_{u}f_{0}\left(P\left(\begin{array}{c}
\xi\\
h_{0}(\xi)
\end{array}\right)\right)+\epsilon Q_{u}\tilde{f}_{1}\left(t\right)+\mathcal{O}\left(\epsilon\left|\xi\right|,\epsilon^{2}\right).\label{eq:leading-order approximation}
\end{align}

This level of approximation of the reduced dynamics was shown to be
accurate for small $\epsilon$ and $\left|\xi\right|$ for periodic
and quasi-periodic forcing by the references cited above. Equation
(\ref{eq:leading-order approximation}) now establishes the same approximation
formula for general, $x$-independent forcing $f_{1}(t)=\epsilon\tilde{f}_{1}\left(t\right)$
for trajectories that stay close to the origin. Under general forcing,
however, trajectories may well stray away from the origin, given that
$W_{\epsilon}\left(E,t\right)$ itself will generally move substantially
away from the origin. In that case, the next level of approximation
obtained from eq. (\ref{eq:leading-order xi equation}) is
\begin{align}
\dot{\xi} & =A^{u}\xi+Q_{u}f_{0}\left(P\left(\begin{array}{c}
\xi\\
h_{0}(\xi)
\end{array}\right)\right)\nonumber \\
 & +\epsilon Q_{u}\left[Df_{0}\left(P\left(\begin{array}{c}
\xi\\
0
\end{array}\right)\right)\left(\int_{-\infty}^{t}e^{A\left(t-\tau\right)}\tilde{f}_{1}(0,\tau)\,d\tau-P\left(\begin{array}{c}
Q_{u}\int_{-\infty}^{t}e^{A\left(t-\tau\right)}\tilde{f}_{1}(0,\tau)\,d\tau\\
\partial_{\epsilon}h_{\epsilon}(\xi,t)\vert_{\epsilon=0}
\end{array}\right)\right)\right]\nonumber \\
 & +\epsilon Q_{u}\tilde{f}_{1}\left(P\left(\begin{array}{c}
\xi\\
0
\end{array}\right),t\right)\nonumber \\
 & +\mathcal{O}\left(\epsilon\left|\xi\right|^{2},\epsilon^{2}\right).\label{eq:leading-order xi equation-1}
\end{align}
Higher-order approximation can be systematically developed by Taylor-expanding
the reduced dynamics (\ref{eq:dfinal reduced dynamics in xi coordinates})
in $\epsilon$ and utilizing the terms from the expansions for $x_{\epsilon}^{*}(t)$
and $W_{\epsilon}\left(E,t\right)$ using Theorems \ref{thm: expansion for x^*}
and \ref{thm:computation of non-automous SSM}.

Alternative parametrizations of $W_{\epsilon}\left(E,t\right)$ beyond
the graph-style parametrization worked out here in detail are also
possible (see \citet{haller16,haro16,vizzaccaro22}). One option is
the normal-form style parameterization which simultaneously brings
the reduced dynamics on $W_{\epsilon}\left(E,t\right)$ to normal
form. Normal form transformations can also be applied separately once
$W_{\epsilon}\left(E,t\right)$ is located (see, e.g., \citet{cenedese22a}).

\section{Adiabatic SSMs }

\subsection{Setup}

We now consider slowly varying non-autonomous dynamical systems of
the form
\begin{equation}
\dot{x}=f(x,\epsilon t),\quad x\in\mathbb{R}^{n},\quad f\in C^{r},\quad0\leq\epsilon\ll1,\label{eq: nonlinear system-2}
\end{equation}
for some integer $r\geq1$. We obtain such a system, for instance,
when we replace the forcing function $f_{1}(x,t)$ in system (\ref{eq: nonlinear system})
with its slowly varying counterpart, $f_{1}(x,\epsilon t)$, in which
case we have $f(x,\epsilon t)=Ax+f_{0}(x)+f_{1}(x,\epsilon t).$ The
form in (\ref{eq: nonlinear system-2}) is, however, more general
than this specific example.

Introducing the slow variable $\alpha=\epsilon t$, we rewrite system
(\ref{eq: nonlinear system-2}) as
\begin{align}
\dot{x} & =f(x,\alpha),\label{eq:slow-fast system}\\
\dot{\alpha} & =\epsilon.\nonumber 
\end{align}
For $\epsilon=0$, we assume that the $x$-component of system (\ref{eq:slow-fast system})
has an $\alpha$-dependent, uniformly bounded and uniformly hyperbolic
fixed point $x_{0}(\alpha)$, i.e., with the notation
\begin{equation}
A(\alpha)=D_{x}f(x_{0}(\alpha),\alpha)\in\mathbb{R}^{n\times n},\label{eq:A(alpha) definition}
\end{equation}
we have
\begin{equation}
f(x_{0}(\alpha),\alpha)=0,\quad\left|\mathrm{Re}\left[\mathrm{spect}\left(A(\alpha)\right)\right]\right|\geq K>0,\quad\alpha\in\mathbb{R},\label{eq:hyperbolicity assumption for slo-fast}
\end{equation}
for some constant $K>0$. We further assume that $x_{0}(\alpha)$
is a class $C^{r}$ function of $\alpha$, in which case, by the first
equation in (\ref{eq:hyperbolicity assumption for slo-fast}), we
have
\begin{equation}
\frac{d^{p}}{d\alpha^{p}}f(x_{0}(\alpha),\alpha)\equiv0,\quad0\leq p\leq r,\quad\alpha\in\mathbb{R}.\label{eq:alpha derivatives vanish}
\end{equation}

By the hyperbolicity assumption (\ref{eq:hyperbolicity assumption for slo-fast}),
we can order the spectrum of $A(\alpha)$ as 
\[
\mathrm{spect}A(\alpha)=\left\{ \lambda_{1}(\alpha),\ldots,\lambda_{n}(\alpha)\right\} ,
\]
so that it satisfies 

\begin{equation}
\mathrm{Re}\,\lambda_{n}(\alpha)\leq\ldots\leq\mathrm{Re}\,\lambda_{j}(\alpha)\leq-K<0<K\leq\mathrm{Re}\,\lambda_{j-1}(\alpha)\leq\ldots\leq\mathrm{Re}\,\lambda_{1}(\alpha),\qquad\alpha\in\mathbb{R},\label{eq:eigenvalue(alpha) list}
\end{equation}
for some $j\geq1$. Conditions (\ref{eq:hyperbolicity assumption for slo-fast})
imply that for $\epsilon=0$, the slow-fast system (\ref{eq:slow-fast system})
has a one-dimensional, normally attracting, non-compact (but uniformly
bounded) invariant manifold of the form
\begin{equation}
\mathcal{L}_{0}=\left\{ (x,\alpha)\in\mathbb{R}^{n}\times\mathbb{R}\colon\,\,\,\,\,x=x_{0}(\alpha)\right\} .\label{eq:critical manifold}
\end{equation}
The manifold $\mathcal{L}_{0}$ is composed of fixed points of system
(\ref{eq:slow-fast system}) for $\epsilon=0$ and hence is called
a \emph{critical manifold} in the language of geometric singular perturbation
theory (\citet{fenichel73}).

\subsection{Existence and computation of an adiabatic anchor trajectory }

As in the non-autonomous case treated in Section \ref{sec:Non-Autonomous-SSM},
we first discuss the continued existence of an anchor trajectory that
acts as a continuation of the critical manifold $\mathcal{L}_{0}$.
\begin{thm}
\label{thm:adiabatic anchor}\textbf{\emph{{[}Existence and computation
of anchor trajectory for adiabatic SSMs{]}}} For $\epsilon>0$ small
enough, there exists a unique, attracting, one-dimensional slow manifold
$\mathcal{L}_{\epsilon}$, composed of a slow trajectory $x_{\epsilon}(\alpha)$
that is uniformly bounded in $\alpha\in\mathbb{R}$. The trajectory
$x_{\epsilon}(\alpha)$ is $\mathcal{O}\left(\epsilon\right)$ $C^{1}$-close
and $C^{r}$ diffeomorphic to the critical manifold $\mathcal{L}_{0}$
defined in (\ref{eq:critical manifold}). Finally, for any nonnegative
integer $N\leq r$, $x_{\epsilon}(\alpha)$ can be approximated as
\begin{equation}
\mathcal{L}_{\epsilon}=\left\{ (x,\alpha)\in\mathbb{R}^{n}\times\mathbb{R}\colon\,\,\,\,\,x=x_{\epsilon}(\alpha)=\sum_{\nu=0}^{N}\epsilon^{\nu}x_{\nu}(\alpha)+o\left(\epsilon^{N}\right)\right\} ,\label{eq:slow manifold}
\end{equation}
 with the recursively defined coefficients

\begin{equation}
x_{\nu}(\alpha)=A^{-1}(\alpha)\left[x_{\nu-1}^{\prime}(\alpha)-\sum_{1<\left|\mathbf{\boldsymbol{\gamma}}\right|\leq\nu}\frac{\partial^{\left|\mathbf{\mathbf{\boldsymbol{\gamma}}}\right|}f\left(x_{0}(\alpha),\alpha\right)}{\partial x_{1}^{\gamma_{1}}\ldots\partial x_{n}^{\gamma_{n}}}\sum_{s=1}^{\nu}\sum_{p_{s}\left(\nu,\mathbf{\boldsymbol{\gamma}}\right)}\prod_{j=1}^{s}\frac{\prod_{i=1}^{n}\left[x_{\ell_{j}}^{i}(\alpha)\right]^{k_{ji}}}{\prod_{i=1}^{n}k_{ji}!}\right],\quad\nu\geq1,\label{eq:recursive formula for adiabatic anchor}
\end{equation}
where the index set $p_{s}\left(\nu,\mathbf{\mathbf{\boldsymbol{\gamma}}}\right)$
is defined as
\[
p_{s}\left(\nu,\mathbf{\mathbf{\boldsymbol{\gamma}}}\right)=\left\{ \left(\mathbf{k}_{1},\ldots,\mathbf{k}_{s},\ell_{1},\ldots,\ell_{s}\right):\mathbf{k}_{i}\in\mathbb{N}^{n}-\left\{ \mathbf{0}\right\} ,\ell_{i}\in\mathbb{N},0<\ell_{1}<\cdots<\ell_{s},\sum_{i=1}^{s}\mathbf{k}_{i}=\mathbf{\boldsymbol{\gamma}},\sum_{i=1}^{s}\left|\mathbf{k}_{i}\right|\ell_{i}=\nu\right\} .
\]
\end{thm}
\begin{proof}
See Appendix \ref{subsec:Proof of approximation of slow manifold}.
\end{proof}
Specifically, for $N=2$, formula (\ref{eq:recursive formula for adiabatic anchor})
gives an approximation for the slow anchor trajectory $x_{\epsilon}(\alpha)$
in the form
\begin{align}
x_{\epsilon}(\alpha) & =x_{0}(\alpha)+\epsilon x_{1}(\alpha)+\epsilon^{2}x_{2}(\alpha)+o\left(\epsilon^{2}\right),\nonumber \\
x_{1}(\alpha) & =\left[D_{x}f(x_{0}(\alpha),\alpha)\right]^{-1}x_{0}^{\prime}(\alpha),\nonumber \\
x_{2}(\alpha) & =\left[D_{x}f(x_{0}(\alpha),\alpha)\right]^{-1}\left[x_{1}^{\prime}(\alpha)-\frac{1}{2}D_{x}^{2}f\left(x_{0}(\alpha),\alpha\right)\otimes x_{1}(\alpha)\otimes x_{1}(\alpha)\right].\label{eq:adiabatic anchor up to second order}
\end{align}

\subsection{Existence and computation of an adiabatic SSM}

We seek to construct an \emph{adiabatic SSM} anchored along the slow
invariant manifold $\mathcal{L}_{\epsilon}$ that is spanned by $x_{\epsilon}(\alpha)$
for small $\epsilon>0$. To this end, we assume that the first $k$
eigenvalues $\lambda_{1}(\alpha),\ldots,\lambda_{k}(\alpha)$ in the
list (\ref{eq:eigenvalue(alpha) list}) have negative real parts and
there is a nonzero spectral gap between $\mathrm{Re}\,\lambda_{k}(\alpha)$
and $\mathrm{Re}\,\lambda_{k+1}(\alpha)$:
\begin{equation}
\mathrm{Re}\,\lambda_{n}(\alpha)\leq\ldots\leq\mathrm{Re}\,\lambda_{k+1}(\alpha)<\mathrm{Re}\,\lambda_{k}(\alpha)\leq\ldots\leq\mathrm{Re}\,\lambda_{1}(\alpha)<0,\qquad\alpha\in\mathbb{R}.\label{eq:adiabatic spectrum assumption}
\end{equation}
We also assume that the real spectral subspace $E(\alpha)$ formed
by the corresponding first $k$ eigenvalues varies smoothly in $\alpha$
and hence has a constant dimension 
\[
d=\dim E(\alpha),\quad\alpha\in\mathbb{R}.
\]
We then define the integer part $\rho$ of the spectral gap associated
with $E(\alpha)$ as
\begin{equation}
\rho=\min_{\alpha\in\mathbb{R}}\mathrm{Int}\left[\frac{\mathrm{Re}\,\lambda_{k+1}(\alpha)}{\mathrm{Re}\,\lambda_{k}(\alpha)}\right]\in\mathbb{N}^{+}.\label{eq:rho-normal hyperbolicity-1}
\end{equation}
Note that $\rho\geq1$ holds by assumption (\ref{eq:adiabatic spectrum assumption}),
with $\rho$ measuring the minimum of how many times the attraction
rates toward $E(\alpha)$ overpower the contraction rates within $E(\alpha)$. 

For $\epsilon=0$ and for each $\alpha\in\mathbb{R}$, system (\ref{eq:slow-fast system})
has an SSM $W_{0}\left(E(\alpha)\right)$ tangent to $E(\alpha)$
at $x_{0}(\alpha)$ (see \citet{haller16}). As a consequence, 
\[
\mathcal{M}_{0}=\underset{\alpha\in\mathbb{R}}{\cup}W_{0}\left(E(\alpha)\right)
\]
is a $(d+1)$-dimensional, $\rho$-normally attracting invariant manifold
for system (\ref{eq:slow-fast system}) for $\epsilon=0$ in the sense
of \citet{fenichel71} and \citet{eldering18}. More specifically,
$\mathcal{M}_{0}$ is a $\rho$-normally attracting, non-compact,
inflowing-invariant manifold with a boundary, as shown in Fig. \ref{fig:Geometry of adiabatic SSM}
for $\epsilon=0$ . 
\begin{figure}[H]
\centering{}\includegraphics[width=0.8\textwidth]{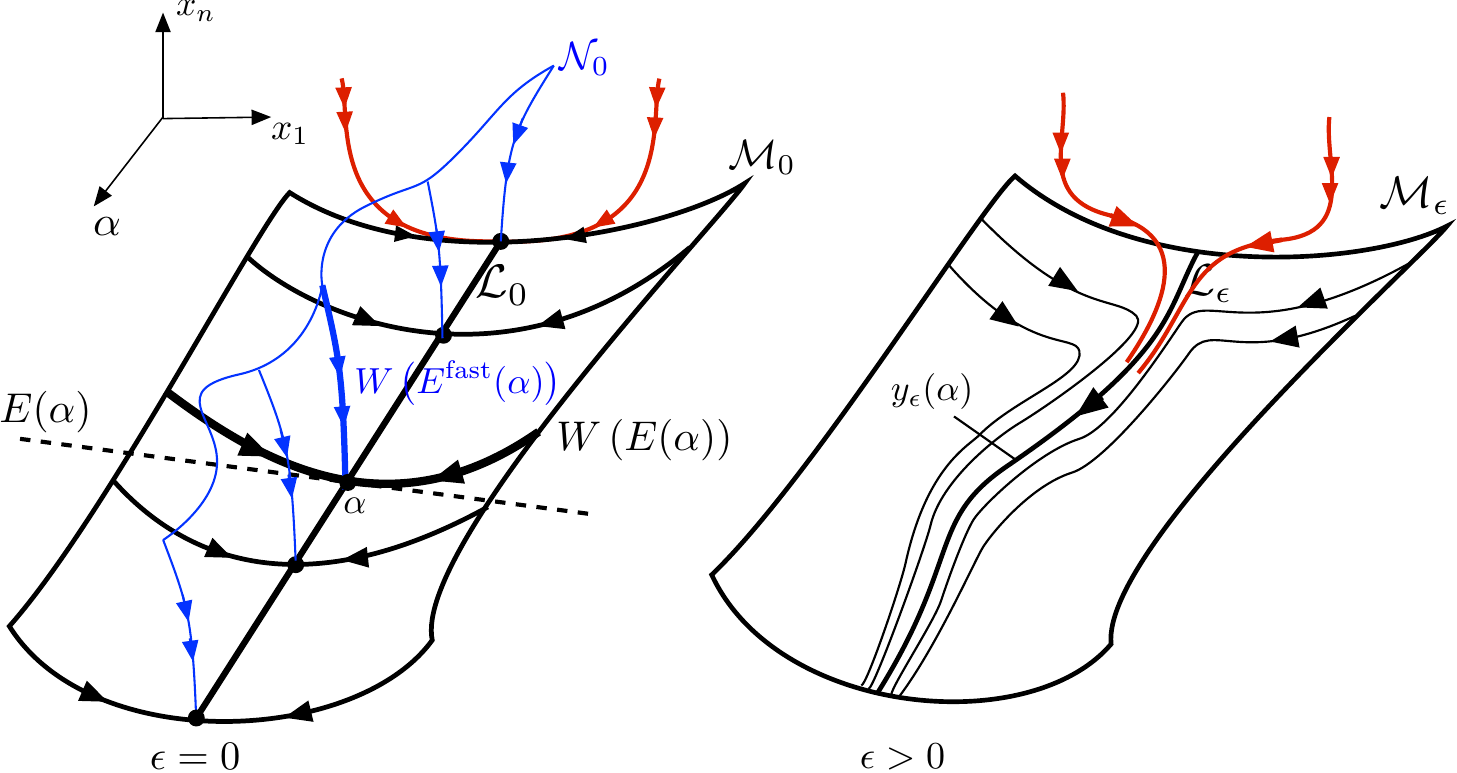}\caption{Left: The geometry of the critical manifold $\mathcal{L}_{0}$ (spanned
by the anchor trajectory $x_{0}(\alpha)$), the attracting invariant
manifold $\mathcal{M}_{0}$, and the fast stable manifold $\mathcal{N}_{0}$
(blue) of $\mathcal{L}_{0}$. Two trajectories off these manifolds
are shown in red. Right: the persistent slow manifold $\mathcal{L}_{\epsilon}$
(spanned by $x_{\epsilon}(\alpha)$) and and the adiabatic SSM $\mathcal{M}_{\epsilon}$
for small enough $\epsilon>0$. \label{fig:Geometry of adiabatic SSM}}
\end{figure}

We have the following main results on the persistence of the manifold
$\mathcal{M}_{0}$ in the form of an adiabatic SSM for small enough
$\epsilon$. 
\begin{thm}
\label{thm:adiabatic SSM}\textbf{\emph{{[}Existence of adiabatic
SSMs{]}}} Assume that $x_{0}(\alpha)$ and $f(x,\alpha)$ have $r$
continuous derivatives that are uniformly bounded in $\alpha\in\mathbb{R}$
in a small, closed neighborhood of $\mathcal{M}_{0}.$ Let $m=\min\left(r,\rho\right).$
Then, for $\epsilon>0$ small enough, there exists a persistent invariant
manifold\emph{ (adiabatic SSM anchored along $\mathcal{L}_{\epsilon}$),
denoted} $\mathcal{M}_{\epsilon}$, that is $C^{m}$ diffeomorphic
to $\mathcal{M}_{0}$, has $m$ uniformly bounded derivatives and
is $\mathcal{O}\left(\varepsilon\right)$ $C^{1}$ -close to $\mathcal{M}_{0}$.
Furthermore, $\mathcal{L}_{\epsilon}\subset\mathcal{M}_{\epsilon}$
holds.
\end{thm}
\begin{proof}
See Appendix \ref{subsec:  Proof of approximation of SSM}.
\end{proof}
We show the geometry of the adiabatic SSM, $\mathcal{M}_{\epsilon}$,
in Fig. \ref{fig:Geometry of adiabatic SSM} for $\epsilon>0$ .
\begin{rem}
{[}\textbf{Applicability to mixed-mode adiabatic SSMs}{]} \label{rem: Applicability-to-slow-mixed-mode-SSM}For
the purposes of constructing $\mathcal{M}_{\epsilon}$, we had to
assume in eq. (\ref{eq:adiabatic spectrum assumption}) that the spectrum
of $A(\alpha)$ lies in the negative complex half plane, i.e., $E(\alpha)$
contains only stable directions. In the terminology of \citet{haller23},
this means that $W\left(E(\alpha)\right)$ has to be a like-mode SSM
for all values of $\alpha$ for Theorem \ref{thm:adiabatic SSM} to
apply. This enables us to select an inflowing-invariant, normally
attracting manifold $\mathcal{\tilde{M}}_{0}$ in our proof to which
the wormhole construct in Proposition B1 of \citet{eldering18} is
applicable. There is every indication that $\mathcal{M}_{\epsilon}$
also exists under the weaker assumption (\ref{eq:hyperbolicity assumption for slo-fast}),
which only requires $A(\alpha)$ to have no eigenvalues on the imaginary
axes, and hence allows $W\left(E(\alpha)\right)$ to be a mixed-mode
SSM that contains both stable and unstable directions. In this case
however, an $\mathcal{\tilde{M}}_{0}$ with an inflowing or overflowing
boundary cannot be chosen and hence Proposition B1 of \citet{eldering18}
would need a technical extension that we will not pursue here, although
it appears doable. We simply note that the asymptotic formulas we
will derive for $\mathcal{M}_{\epsilon}$ are formally valid under
the weaker assumption (\ref{eq:hyperbolicity assumption for slo-fast})
as well (i.e., for mixed-mode $W\left(E(\alpha)\right)$), as long
as $E(\alpha)$ is normally hyperbolic, i.e., attracts solutions at
a rate that is uniformly stronger than the contraction rates inside
$E(\alpha)$. We will evaluate and confirm these formulas in an example
with a mixed mode adiabatic SSMs in Section \ref{subsec:Slow-forced bumpy rail}.
\end{rem}
As in the general non-autonomous case treated in Section \ref{sec:Non-Autonomous-SSM},
the adiabatic $\mathcal{M}_{\epsilon}$ will generally persist for
larger values of $\epsilon$ and will also be smoother than $C^{m}$
under additional nonresonance conditions. In the following, we will
provide a numerically implementable recursive scheme for computing
$\mathcal{M}_{\epsilon}$, assuming that it exists and is smooth enough.

To state these recursive approximation results, we first introduce
the new coordinates 
\begin{equation}
\left(\begin{array}{c}
u\\
v
\end{array}\right)=P^{-1}\left(\alpha\right)\left(x-x_{\epsilon}(\alpha)\right),\label{eq:u-v transformation,  adiabatic case}
\end{equation}
where $P(\alpha)\in\mathbb{C}^{n}$ diagonalizes the matrix $A(\alpha)$.
The coordinates $\left(u,v\right)\in\mathbb{C}^{d}\times\mathbb{C}^{n-d}$
align with the $d$-dimensional stable spectral subspace family $E(\alpha)$
and with the direct sum of the remaining eigenspaces of $A(\alpha)$,
respectively. In these new coordinates, system (\ref{eq:slow-fast system})
becomes
\begin{align}
\left(\begin{array}{c}
\dot{u}\\
\dot{v}
\end{array}\right) & =\left(\begin{array}{cc}
A^{u}\left(\alpha\right) & 0\\
0 & A^{v}\left(\alpha\right)
\end{array}\right)\left(\begin{array}{c}
u\\
v
\end{array}\right)+\hat{f}(u,v,\epsilon;\alpha),\label{eq:(u,v) system-2}\\
\dot{\alpha} & =\epsilon,\nonumber 
\end{align}
where
\begin{align}
 & A^{u}\left(\alpha\right)=\left(\begin{array}{ccc}
\lambda_{1}\left(\alpha\right) & 0 & 0\\
0 & \ddots & 0\\
0 & 0 & \lambda_{d}\left(\alpha\right)
\end{array}\right),\quad A^{v}\left(\alpha\right)=\left(\begin{array}{ccc}
\lambda_{d+1}\left(\alpha\right) & 0 & 0\\
0 & \ddots & 0\\
0 & 0 & \lambda_{n}\left(\alpha\right)
\end{array}\right),\label{eq:alpa dependent matrices}
\end{align}
and 
\begin{equation}
\hat{f}(u,v,\epsilon;\alpha)=P^{-1}\left(\alpha\right)\left[f\left(x_{\epsilon}(\alpha)+P\left(\alpha\right)\left(\begin{array}{c}
u\\
v
\end{array}\right),\alpha\right)-A\left(\alpha\right)P\left(\alpha\right)\left(\begin{array}{c}
u\\
v
\end{array}\right)-\epsilon x_{\epsilon}^{\prime}(\alpha)-\epsilon P^{\prime}\left(\alpha\right)\left(\begin{array}{c}
u\\
v
\end{array}\right)\right].\label{eq:fhat definition-1}
\end{equation}

By existing SSM theory, any $\alpha=const.$ slice of $\mathcal{M}_{0}$
can be written in the form of a Taylor expansion in $u$ with coefficients
depending smoothly on $\alpha$. Therefore, the whole of $\mathcal{M}_{0}$
can be written as
\begin{equation}
\mathcal{M}_{0}=\cup_{\alpha\in\mathbb{R}}W_{0}\left(E(\alpha)\right)=\left\{ \left(u,v,\alpha\right)\in U\subset\mathbb{R}^{n}:\,v=h_{0}(u,\alpha)=\sum_{\left|\mathbf{k}\right|\geq2}^{r}h^{\mathbf{k}}(\alpha)u^{\mathbf{k}}+o\left(\left|u\right|^{r}\right)\right\} .\label{eq:M_0 expansion}
\end{equation}
We can now state the following results on the computation of the adiabatic
SSM, $\mathcal{M}_{\epsilon}$, for $\epsilon\geq0$. 
\begin{thm}
\label{thm:computation of adiabatic SSM}\textbf{\emph{{[}Computation
of adiabatic SSMs{]}}} Assume that an $\epsilon$-dependent adiabatic
SSM, $\mathcal{M}_{\epsilon}$, of the type described in Theorem \ref{thm:adiabatic SSM}
exists in the transformed system (\ref{eq:(u,v) system-2}) for all
$\epsilon\in\left[0,\epsilon^{*}\right]$ . Assume further that $\mathcal{M}_{\epsilon}$
is $N$-times continuously differentiable and the nonresonance conditions
\begin{equation}
\lambda_{j}(\alpha)\neq\sum_{\lambda_{k}(\alpha)\in\mathrm{spect}\left(A(\alpha)\vert_{E(\alpha)}\right)}m_{k}\lambda_{k}(\alpha),\qquad\lambda_{j}(\alpha)\in\mathrm{spect}\left(A(\alpha)\right)-\mathrm{spect}\left(A(\alpha)\vert_{E}(\alpha)\right),\label{eq:strict nonresonance-1}
\end{equation}
hold for all $j=k+1,\ldots,n$ and for all $m_{k}\in\mathbb{N}$ with
\begin{equation}
1\leq\sum_{k=1}^{n}m_{k}\leq N.\label{eq:lowered lower boudn on summation in nonresonance for adiabatic case}
\end{equation}

Then, for all $\epsilon\in\left[0,\epsilon^{*}\right]$ :
\end{thm}
\begin{description}
\item [{\emph{(i)}}] \emph{The adiabatic SSM, $\mathcal{M}_{\epsilon}$,
admits a formal asymptotic expansion 
\begin{equation}
\mathcal{M}_{\epsilon}=\left\{ \left(u,v,\alpha\right)\in U\subset\mathbb{R}^{n}:\,v=h_{\epsilon}(u,\alpha)=\sum_{\left|\left(\mathbf{k},p\right)\right|\geq1}^{N}h^{\mathbf{k}p}(\alpha)u^{\mathbf{k}}\epsilon^{p}+o\left(\left|u\right|^{q}\epsilon^{N-q}\right)\right\} .\label{eq:h double expansion-1}
\end{equation}
The functions $h^{\mathbf{k}p}(\alpha)$ are uniformly bounded in
$\alpha\in\mathbb{R}$ for all $\left(\mathbf{k},p\right)$ and can
be computed recursively from their initial conditions 
\begin{align}
h^{\mathbf{0}p}(\alpha) & \equiv0,\quad p\geq0;\quad h^{\mathbf{k}0}(\alpha)\equiv h^{\mathbf{k}}(\alpha),\quad\mathbf{k}\in\mathbb{N}^{d};\label{eq:known k^=00007Bkp=00007D coeffs in adiabatic expansion-2}\\
h^{\mathbf{k}0}(\alpha) & \equiv h^{\mathbf{k}0}(\alpha)\equiv h^{\mathbf{k}}(\alpha)=0,\quad\left|\mathbf{k}\right|=1;\quad h^{\mathbf{k}(-1)}(\alpha):=0,\nonumber 
\end{align}
via the formula
\begin{equation}
h^{\mathbf{k}p}(\alpha)=A_{\mathbf{k}}^{-1}(\alpha)\left[\left[h^{\mathbf{k}(p-1)}\right]^{\prime}(\alpha)-M^{\mathbf{k}p}(\alpha,h^{\mathbf{j}m})\right],\quad\left|(\mathbf{j},m)\right|<\left|(\mathbf{k},p)\right|,\label{eq:h_=00007Bkp=00007D  adiabatic case}
\end{equation}
where
\[
A_{\mathbf{k}}(\alpha)=\mathrm{diag}\left[\lambda_{\ell}(\alpha)-\sum_{j=1}^{d}k_{j}\lambda_{j}(\alpha)\right]_{\ell=d+1}^{n}\in\mathbb{C}^{\left(n-d\right)\times\left(n-d\right)},
\]
\[
M^{\mathbf{k}p}(\alpha,h^{\mathbf{j}m})=\frac{\partial^{\left|\left(\mathbf{k},p\right)\right|}}{\partial u_{1}^{k_{1}}\ldots\partial u_{d}^{k_{d}}\partial\epsilon^{p}}\left[\hat{f}^{v}\left(u,\sum_{\left|\left(\mathbf{j},m\right)\right|\geq1}^{\left|\left(\mathbf{k},p\right)\right|-1}h^{\mathbf{j}m}(\alpha)u^{\mathbf{j}}\epsilon^{m},\epsilon;\alpha;\mathbf{k},p\right)\right.
\]
\begin{equation}
-\left.\left.\sum_{\left|\left(\mathbf{j},m\right)\right|\geq1}^{\left|\left(\mathbf{k},p\right)\right|-1}\epsilon^{p}\left(\begin{array}{ccc}
h_{1}^{\mathbf{j}m}(\alpha)\frac{j_{1}u^{\mathbf{j}}}{u_{1}} & \cdots & h_{1}^{\mathbf{j}m}(\alpha)\frac{j_{d}u^{\mathbf{j}}}{u_{d}}\\
\vdots & \ddots & \vdots\\
h_{n-d}^{\mathbf{j}m}(\alpha)\frac{j_{1}u^{\mathbf{j}}}{u_{1}} & \cdots & h_{n-d}^{\mathbf{j}m}(\alpha)\frac{j_{d}u^{\mathbf{j}}}{u_{d}}
\end{array}\right)\hat{f}^{u}\left(u,\sum_{\left|\left(\mathbf{j},m\right)\right|\geq1}^{\left|\left(\mathbf{k},p\right)\right|-1}h^{\mathbf{j}m}(\alpha)u^{\mathbf{j}}\epsilon^{m},\epsilon;\alpha;\mathbf{k},p\right)\right]\right|_{u=0,\,\epsilon=0}.\label{eq:M^k-1-2-1-1-1}
\end{equation}
Specifically, for $\epsilon=0$, the coefficients in the expansion
(\ref{eq:M_0 expansion}) for $\mathcal{M}_{0}$ can be computed as
\begin{equation}
h^{\mathbf{k}}(\alpha)\equiv h^{\mathbf{k}0}(\alpha)=-A_{\mathbf{k}}^{-1}(\alpha)M^{\mathbf{k}}(\alpha,h^{\mathbf{j}0}),\quad\left|\mathbf{j}\right|<\left|\mathbf{k}\right|.\label{eq:coeffs in autonomus limit-1}
\end{equation}
}
\item [{\emph{(ii)}}] \emph{The reduced dynamics on $\mathcal{M}_{\epsilon}$
is given by
\begin{align}
\dot{u} & =Q_{u}\left(\alpha\right)\left[f\left(x_{\epsilon}(\alpha)+P(\alpha)\left(\begin{array}{c}
u\\
h_{\epsilon}(u,\alpha)
\end{array}\right),\alpha\right)-\epsilon x_{\epsilon}^{\prime}(\alpha)\right]+\epsilon Q_{u}^{\prime}\left(\alpha\right)P(\alpha)\left(\begin{array}{c}
u\\
h_{\epsilon}(u,\alpha)
\end{array}\right),\nonumber \\
\dot{\alpha} & =\epsilon.\label{eq:adiabatic restriced dynamics}
\end{align}
 Here $Q_{u}\left(\alpha\right)\in\mathbb{C}^{d\times n}$ is a matrix
whose $j^{th}$ row is $\hat{e}_{j}(\alpha)/\left(\hat{e}_{j}(\alpha)\cdot e_{j}(\alpha)\right)$
for $j=1,\ldots d$, where $\hat{e}_{j}(\alpha)$ is the $j^{th}$
unit left eigenvector of $P(\alpha)$ corresponding to its unit right
eigenvector $e_{j}(\alpha)$. Equivalently, $Q_{u}(\alpha)$ is composed
of the first $d$ rows of $P^{-1}(\alpha)$. }
\end{description}
\begin{proof}
See Appendix \ref{subsec:Computation-of-the-adiabatic-SSM}.
\end{proof}
A leading-order approximation for the reduced dynamics on the adiabatic
SSM, \emph{$\mathcal{M}_{\epsilon}$}, can be obtained from formula
(\ref{eq:adiabatic restriced dynamics}) as
\begin{align}
\dot{u} & =Q_{u}\left(\alpha\right)f\left(x_{0}(\alpha)+P(\alpha)\left(\begin{array}{c}
u\\
h_{0}(u,\alpha)
\end{array}\right),\alpha\right)\nonumber \\
 & +\epsilon Q_{u}\left(\alpha\right)D_{x}f\left(x_{0}(\alpha)+P(\alpha)\left(\begin{array}{c}
u\\
h_{0}(u,\alpha)
\end{array}\right),\alpha\right)\left(\left[D_{x}f(x_{0}(\alpha),\alpha)\right]^{-1}-I\right)x_{0}^{\prime}(\alpha)\nonumber \\
 & +\epsilon Q_{u}\left(\alpha\right)D_{x}f\left(x_{0}(\alpha)+P(\alpha)\left(\begin{array}{c}
u\\
h_{0}(u,\alpha)
\end{array}\right),\alpha\right)P(\alpha)\left(\begin{array}{c}
u\\
\partial_{\epsilon}h_{\epsilon}(u,\alpha)\vert_{\epsilon=0}
\end{array}\right)\nonumber \\
 & +\epsilon Q_{u}^{\prime}\left(\alpha\right)P(\alpha)\left(\begin{array}{c}
u\\
h_{0}(u,\alpha)
\end{array}\right)\nonumber \\
 & +\mathcal{O}\left(\epsilon^{2}\right),\nonumber \\
\dot{\alpha} & =\epsilon,\label{eq:first adiabatic reduced dynamics approximation}
\end{align}
 where we have used the expression for $x_{1}(\alpha)$ from the expansion
(\ref{eq:adiabatic anchor up to second order}). 

\subsection{Special case: Adiabatic SSM under additive slow forcing }

We now assume that the slow time dependence in system \eqref{eq: nonlinear system-2}
arises from slow (but generally not small) additive forcing. This
is a reasonable assumption, for instance, in the setting of external
control forces acting on a highly damped structure, whose unforced
decay rate to an equilibrium is much faster than the rate at which
the forcing changes in time. Another relevant setting is very slow
(quasi-static) forcing in structural dynamics.

In such cases, system \eqref{eq: nonlinear system-2} can be rewritten
\begin{equation}
\dot{x}=f(x,\epsilon t)=f_{0}(x)+f_{1}(\epsilon t),\quad x\in\mathbb{R}^{n},\quad f_{0},f_{1}\in C^{r},\quad0\leq\epsilon\ll1,\label{eq: nonlinear system-2-1}
\end{equation}
 or, equivalently, 
\begin{align}
\dot{x} & =f_{0}(x)+f_{1}(\alpha),\label{eq:slow-fast system-1}\\
\dot{\alpha} & =\epsilon.\nonumber 
\end{align}
 For $\epsilon=0$, the invariant manifold $\mathcal{L}_{0}$ of fixed
points satisfies
\begin{equation}
f_{0}(x_{0}(\alpha))+f_{1}(\alpha)=0,\label{eq:frozen acnhor trajectory under slow forcing}
\end{equation}
and we have 
\begin{equation}
A(x_{0}(\alpha))=D_{x}f(x_{0}(\alpha),\alpha)=D_{x}f_{0}(x_{0}(\alpha))\in\mathbb{R}^{n\times n}.\label{eq:A(alpha) definition-1}
\end{equation}
 Note that we have changed our notation slightly for the matrix $A$
to point out that it depends solely on $x_{0}(\alpha)$.

These simplifications imply that the matrix $A(x_{0}(\alpha))$, the
column matrix $P(x_{0}(\alpha))$ of the right eigenvectors of $A(x_{0}(\alpha))$
and the row matrix $Q_{u}(x_{0}(\alpha))$ of the first $d$, appropriately
normalized left eigenvectors of $A(x_{0}(\alpha))$ can now be determined
solely from the unforced part $f_{0}(x)$ of the right-hand side of
system (\ref{eq: nonlinear system-2-1}) along any parametrized path
$x_{0}(\alpha)$ defined implicitly by eq. (\ref{eq:frozen acnhor trajectory under slow forcing}).
Similarly, an inspection of the formulas (\ref{eq:coeffs in autonomus limit-1})
defining the slow SSM, $\mathcal{M}_{0}$, reveals that the graph
$h_{0}(u,x_{0}(\alpha))$ of $\mathcal{M}_{0}$ can be computed along
the path $x_{0}(\alpha)$ solely based on the knowledge of $f_{0}(x)$;
only the path itself depends on the specific forcing term $f_{1}(\alpha)$.

From eq. (\ref{eq:adiabatic restriced dynamics}), we obtain that
the leading-order reduced dynamics on $\mathcal{M}_{\epsilon}$ are
now
\begin{align}
\dot{u} & =Q_{u}\left(x_{0}(\alpha)\right)f_{0}\left(x_{0}(\alpha)+P(x_{0}(\alpha))\left(\begin{array}{c}
u\\
h_{0}(u,x_{0}(\alpha))
\end{array}\right)\right)+Q_{u}\left(x_{0}(\alpha)\right)f_{1}(\alpha),\label{eq:full leading-order adiabatic model}\\
\dot{\alpha} & =\epsilon.\nonumber 
\end{align}
Along any envisioned path $x_{0}(\alpha)$, one can a priori compute
the parametric family of frozen-time reduced models
\begin{equation}
\dot{u}=Q_{u}\left(p\right)f_{0}\left(p+P(p)\left(\begin{array}{c}
u\\
h_{0}(u,p)
\end{array}\right)\right),\quad p\in U,\label{eq:model family}
\end{equation}
at all points $p$ within  an open set $U$ in the phase space that
is expected to contain $x_{0}(\alpha)$ for possible forcing functions
$f_{1}(\alpha)$ of interest. Then, for any specific forcing $f_{1}(\alpha)$,
we can determine the path $x_{0}(\alpha)$ from eq. (\ref{eq:frozen acnhor trajectory under slow forcing}),
and use the appropriate elements of the model family (\ref{eq:model family})
with $p=x_{0}(\alpha)$ in eq. (\ref{eq:full leading-order adiabatic model})
to obtain the non-autonomous leading order model
\begin{equation}
\dot{u}=Q_{u}\left(x_{0}(\epsilon t)\right)f_{0}\left(x_{0}(\epsilon t)+P(x_{0}(\epsilon t))\left(\begin{array}{c}
u\\
h_{0}(u,x_{0}(\epsilon t))
\end{array}\right)\right)+Q_{u}\left(x_{0}(\epsilon t)\right)f_{1}(\epsilon t).\label{eq:non-autonomous leading-order model}
\end{equation}
In practice, one would only precompute (\ref{eq:model family}) at
a discrete set of points $\left\{ p_{k}\right\} _{k=1}^{K}\subset$
$U$ to obtain a finite model family from (\ref{eq:model family}),
and then interpolate from these points to approximate the full leading-order
model (\ref{eq:full leading-order adiabatic model}) along $x_{0}(\alpha)$.

A specific feature arising in applications to forced mechanical systems
is that the phase space variable $x=(q,v)$ is composed of positions
$q\in\mathbb{R}^{n/2}$ and their corresponding velocities $v=\dot{q}\in\mathbb{R}^{n/2}$,
where $n$ is an even number denoting twice the number of degree of
freedom of the mechanical system. In that case, forcing only appears
in the $v$ component of the ODE (\ref{eq: nonlinear system-2-1}),
implying 
\[
f_{0}(x)=\left(v,f_{0}^{v}(q,v)\right)\in\mathbb{R}^{n/2}\times\mathbb{R}^{n/2},\quad f_{1}(\alpha)=\left(0,f_{1}^{v}(\alpha)\right)\in\mathbb{R}^{n/2}\times\mathbb{R}^{n/2}.
\]
 As a consequence, eq. (\ref{eq:frozen acnhor trajectory under slow forcing})
takes the more specific form 
\begin{align}
v(\alpha) & =0,\label{eq:frozen acnhor trajectory under slow forcing-1}\\
f_{0}^{v}(q(\alpha),v(\alpha))+f_{1}^{v}(\alpha) & =0,\nonumber 
\end{align}
 leading to the single equation 
\[
f_{0}^{v}(q(\alpha),0)+f_{1}^{v}(\alpha)=0,
\]
for the path $x_{0}(\alpha)=(q(\alpha),0).$ This means that the reduced-order
model family (\ref{eq:model family}) can be constructed a priori
along different spatial locations $q$ of the configuration space
under the application of static forces $f_{1}^{v}(\alpha)$ that keep
the mechanical system at equilibrium ($v=0$) at those $q(\alpha)$
locations. This is a significant simplification over the general setting
of eq. (\ref{eq:model family}), because it only requires the construction
of the model family (\ref{eq:model family}) over a codimension-$n$
subset $U_{q}=\left\{ (q,v)\in\mathbb{R}^{n}:\,\,v=0\right\} $ of
the full phase space $\mathbb{R}^{n}$ of system (\ref{eq: nonlinear system-2-1}). 

This simplified reduced model construction is expected to provide
further improvement over current autonomous-SSM-based control strategies
that use only a single member of the model family (\ref{eq:model family}),
computed at fixed point $x_{0}^{*}$ of the unforced system (see \citet{alora23,alora23b}).
This single model is then used along the full path $x_{0}(\alpha)$
to approximate the leading-order model (\ref{eq:full leading-order adiabatic model}).
Clearly, this approximation will generally only be valid when the
instantaneous equilibrium path $x_{0}(\alpha)$ remains close to the
unperturbed fixed point $x_{0}^{*}$. We will explore the application
of the ideas described in this section to robot control in other upcoming
publications.

\section{Examples}

Here we consider two main mechanical examples to illustrate the construction
of anchor trajectories and the corresponding non-autonomous SSMs under
weak or adiabatic forcing. The governing equations of these mechanical
systems are of the general form
\begin{equation}
\mathbf{{M}\ddot{{x}}+}\mathbf{\mathbf{{C}\dot{{x}}}+{K}\mathbf{{x}}+}\mathbf{f}(\mathbf{{x}})=\mathbf{{F}}(t),\label{eq:general mechanical system}
\end{equation}
where $\mathbf{x}$ is the vector of generalized coordinates; $\mathbf{M}=\mathbf{M}^{\mathrm{T}}$
is the positive definite mass matrix; $\mathbf{C}=\mathbf{C}^{\mathrm{T}}$
is the positive definite damping matrix. The stiffness matrix $\mathbf{K}=\mathbf{K}^{\mathrm{T}}$
is positive definite in our first example and indefinite in our second
example; $\mathbf{f}(\mathbf{{x}})$ is the vector of geometric nonlinearities;
 $\mathbf{{F}}(t)$ is the vector of external forcing. We can rewrite
the second-order system (\ref{eq:general mechanical system}) in the
first-order form we have used in this paper by letting
\[
x={\scriptstyle \left(\begin{array}{c}
\mathbf{x}\\
\dot{\mathbf{x}}
\end{array}\right)},\quad A={\scriptstyle \left(\begin{array}{cc}
\mathbf{0} & \mathbf{I}\\
\mathbf{-{M}^{-1}K} & \mathbf{-{M}^{-1}C}
\end{array}\right)},\quad f_{0}(x)={\scriptstyle \left(\begin{array}{c}
\mathbf{0}\\
-\mathbf{{M}}^{-1}\mathbf{f}(\mathbf{x})
\end{array}\right)},\quad f_{1}(t)={\scriptstyle \left(\begin{array}{c}
\mathbf{0}\\
-\mathbf{{M}}^{-1}\mathbf{{F}}(t)
\end{array}\right)}.
\]
With this notation, system (\ref{eq:general mechanical system}) will
take the form of the first-order system (\ref{eq: nonlinear system})
or that of (\ref{eq: nonlinear system-2-1}), depending on whether
$\mathbf{{F}}(t)$ can be considered small or slow.

Both of our examples treated below admit a hyperbolic fixed point
at $x=0$ in the absence of forcing. To add general aperiodic forcing
to these systems, we will construct the vector $\mathbf{{F}}(t)$
from a trajectory on the chaotic attractor of the classic Lorenz system,
\begin{align}
\dot{x} & =\sigma(y-x),\nonumber \\
\dot{y} & =x(\rho-z)-y,\nonumber \\
\dot{z} & =xy-\beta z,\label{eq:lorenz}
\end{align}
with parameter values $\ensuremath{\rho=28},\ensuremath{\sigma=10},$
and $\beta=\frac{8}{3}$ . We solve this system over the time interval
$[0,500]$, starting from the initial condition $(x_{0},y_{0,}z_{0})=\left(0,0.3,0.5\right)$
to generate weak forcing and over the time interval $[-15,20]$ from
$(x_{0},y_{0,}z_{0})=\left(0.8,0.3,0.2\right)$ to generate slow forcing.
Specifically, we use the $x(t)$ component of this solution as forcing
profile after appropriate scaling in both cases. Outside the time
interval $[0,500]$, we select the forcing to be identically zero
for the weak forcing case. To emulate slow forcing, we use a slowed-down
version of the chaotic signal by rescaling time as $t\mapsto\alpha=\epsilon t$.
These choices of the forcing ensure its uniform boundedness, which
was one of our fundamental assumptions in deriving our results for
weak forcing. In the slow forcing case, the signal is smooth enough
to ensure accurate numerical differentiability in the $\alpha$-interval
$[0,6]$ we consider.

We formulated our results in the previous sections for fully non-dimensionalized
systems for which smallness or slowness can simply be imposed by selecting
a single perturbation parameter $\epsilon$ small enough. In specific
physical examples, such a non-dimensionalization can be done in multiple
ways and is often cumbersome to carry out for large systems. One nevertheless
needs to assess whether the external forcing is \emph{de facto} smaller
or slower than the magnitude or the speed of variation of internal
forces along trajectories. While such a consideration has been largely
ignored in applications of perturbation results to physical problems,
here we will propose and apply heuristic physical measures that help
assessing the magnitude and the slowness of the perturbation. 

To assess smallness or slowness in a systematic and non-dimensional
fashion without non-dimensionalization of the full mechanical system,
we first express the mechanical system in the form 
\begin{equation}
\mathbf{M}\ddot{\mathbf{x}}=\mathbf{F}_{int}\left(\mathbf{x},\dot{\mathbf{x}}\right)+\mathbf{F}_{ext}\left(\mathbf{x},\dot{\mathbf{x}},t\right),\label{eq:general forced equation of motion}
\end{equation}
with the subscripts referring to the autonomous internal forces and
non-autonomous external forces, respectively. Using these forces,
we define the non-dimensional \emph{forcing weakness} $r_{w}$ and
\emph{forcing speed} $r_{s}$ as

\begin{equation}
r_{w}=\frac{\overline{\int_{t_{0}}^{t_{f}}\left|\mathbf{F}_{ext}\left(\mathbf{x},\dot{\mathbf{x}},t\right)\right|dt}}{\overline{\int_{t_{0}}^{t_{f}}\left|\mathbf{F}_{int}\left(\mathbf{x},\dot{\mathbf{x}}\right)\right||dt}},\qquad r_{s}=\frac{\overline{\int_{t_{0}}^{t_{f}}\left|\frac{\partial}{\partial t}\mathbf{F}_{ext}\left(\mathbf{x},\dot{\mathbf{x}},t\right)\right|dt}}{\overline{\int_{t_{0}}^{t_{f}}\left|\frac{d}{dt}\mathbf{F}_{int}\left(\mathbf{x},\dot{\mathbf{x}}\right)\right||dt}},\label{eq:weakness and slowness measures}
\end{equation}
where the over-bar represents averaging with respect to the initial
conditions $\left(\mathbf{x}_{0},\dot{\mathbf{x}}_{0}\right)$ of
the unperturbed (i.e. $\epsilon=0)$ trajectories $\left(\mathbf{x}(t),\dot{\mathbf{x}}(t)\right)$
of the system over an open domain of interest in the phase space. 

In a practical setting, we deem a specific external forcing weak if
$r_{w}\ll1$ or slow if $r_{s}\ll1$. This is the range in which the
perturbation-based invariant manifold techniques used in proving our
relevant theorems can reasonably expected to apply. We have noted,
however, that the SSMs obtained from these techniques are normally
hyperbolic and hence persist under further increases in $r_{w}$ and
$r_{s}$. Indeed, we will show that our formal asymptotic SSM expansions
tend to remain valid for $r_{w}>1$ and and $r_{s}>1$. In some cases,
we also find these formulas to have predictive power for $r_{w},r_{s}\gg1$. 

All MATLAB scripts used in producing the results for the examples
below can be publicly accessed at \citet{kaundinya23}.

\subsection{Example 1 : Chaotically forced cart system\label{subsec:Cart example}}

\subsubsection{Weak forcing\label{subsec:weakly forced cart}}

We consider the two-degree-of-freedom (2 DOF) mechanical system from
\citet{haller16}, placed now on a moving cart subject to external
chaotic shaking, as shown in Fig. \ref{Figure:Shaking_cart}. In this
example, the fixed point of the unforced system is asymptotically
stable. As a result, under external forcing, the SSM attached to a
unique nearby anchor trajectory will be of like-mode type.

Using the coordinate vector $\mathbf{x}=\left(q_{1},q_{2},x_{c}\right)$,
we obtain the equations of motion in the form (\ref{eq:general mechanical system})
with 

\begin{gather*}
{\scriptstyle {\displaystyle \mathbf{M}}=\left(\begin{array}{ccc}
M_{f}\frac{(m_{1}+m_{2})}{M_{T}} & \frac{m_{2}M_{f}}{M_{T}} & 0\\
\frac{m_{2}M_{f}}{M_{T}} & m_{2}\frac{(m_{1}+M_{f})}{M_{T}} & 0\\
0 & 0 & M_{T}
\end{array}\right),\quad{\displaystyle \mathbf{K}}=\left(\begin{array}{ccc}
2k+\frac{k_{f}(m_{1}+m_{2})^{2}}{M_{T}^{2}} & k+\frac{k_{f}m_{2}(m_{1}+m_{2})}{M_{T}^{2}} & -k_{f}\frac{m_{1}+m_{2}}{M_{T}}\\
k+\frac{k_{f}m_{2}(m_{1}+m_{2})}{M_{T}^{2}} & 2k+\frac{k_{f}m_{2}^{2}}{M_{T}^{2}} & -k_{f}\frac{m_{2}}{M_{T}}\\
-k_{f}\frac{m_{1}+m_{2}}{M_{T}} & -k_{f}\frac{m_{2}}{M_{T}} & k_{f}
\end{array}\right)},
\end{gather*}

\begin{gather}
{\scriptstyle {\displaystyle \mathbf{C}}=\left(\begin{array}{ccc}
c+\frac{c_{f}(m_{1}+m_{2})^{2}}{M_{T}^{2}} & c+\frac{c_{f}m_{2}(m_{1}+m_{2})}{M_{T}^{2}} & -c_{f}\frac{m_{1}+m_{2}}{M_{T}}\\
c+\frac{c_{f}m_{2}(m_{1}+m_{2})}{M_{T}^{2}} & 2c+\frac{c_{f}m_{2}^{2}}{M_{T}^{2}} & -c_{f}\frac{m_{2}}{M_{T}}\\
-c_{f}\frac{m_{1}+m_{2}}{M_{T}} & -c_{f}\frac{m_{2}}{M_{T}} & c_{f}
\end{array}\right),\quad{\displaystyle \mathbf{f}\left(\mathbf{x},\dot{\mathbf{x}}\right)}=\left(\begin{array}{c}
\gamma q_{1}^{3}\\
0\\
0
\end{array}\right),\quad{\displaystyle \mathbf{F}(t)}=\left(\begin{array}{c}
0\\
0\\
M_{T}g(t)
\end{array}\right)},\nonumber \\
\label{eq:data for weakly forced cart}
\end{gather}
where $M_{T}=M_{f}+m_{1}+m_{2}$. We further set $m_{1}=m_{2}=1\,\left[\mathrm{kg}\right]$,
$M_{f}=4\,\left[\mathrm{kg}\right]$, $k=k_{f}=1\,\mathrm{\left[N/m\right]}$
and $c=c_{f}=0.3\,\mathrm{\left[Nm/s\right]}$ and $\gamma=0.5\,\left[\mathrm{N/m^{3}}\right]$. 

We will consider two different cases of forcing by scaling the magnitude
of the chaotic signal $g(t)$ in two different ways. In the first
case, we will have $\max\left|\mathbf{F}\left(t\right)\right|=0.06\,\mathrm{\left[N\right]}$,
which gives the non-dimensional forcing weakness $r_{w}=0.16$ from
the first formula in eq. (\ref{eq:weakness and slowness measures}).
This forcing scheme can thus be considered weak albeit not very small.
Our subsequent choice of $\max\left|\mathbf{F}\left(t\right)\right|=3\,\mathrm{\left[N\right]}$
gives $r_{w}=7.8$, which is definitely outside the range of small
forcing.
\begin{figure}[H]
\subfloat[]{\includegraphics[width=0.6\columnwidth]{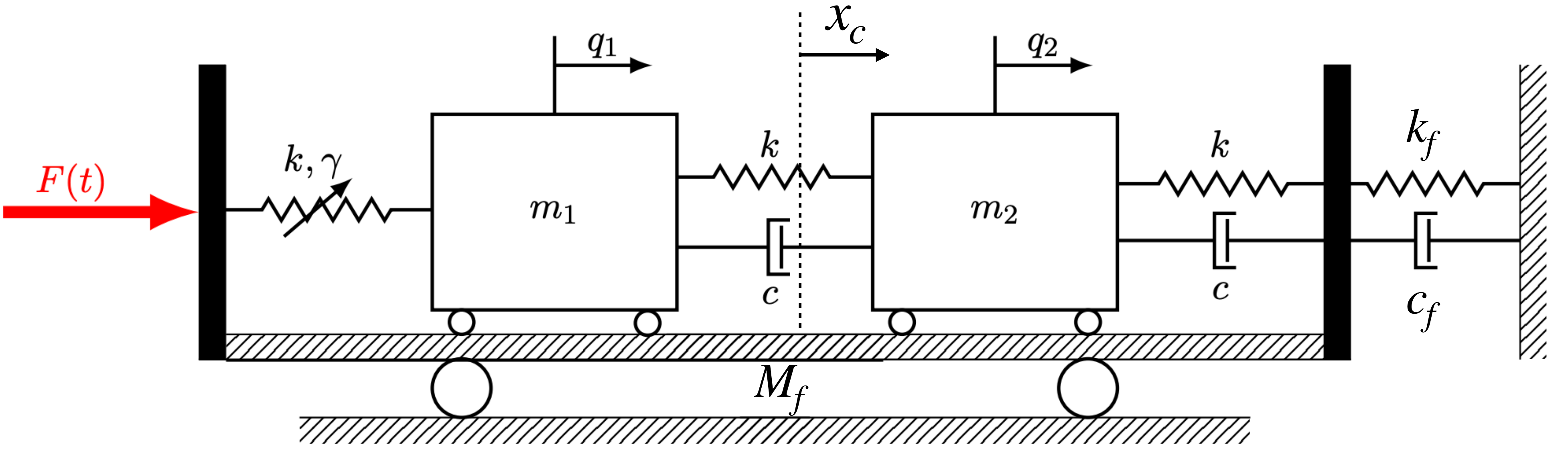}

\label{Figure:Shaking_cart}}\subfloat[]{\includegraphics[width=0.4\columnwidth]{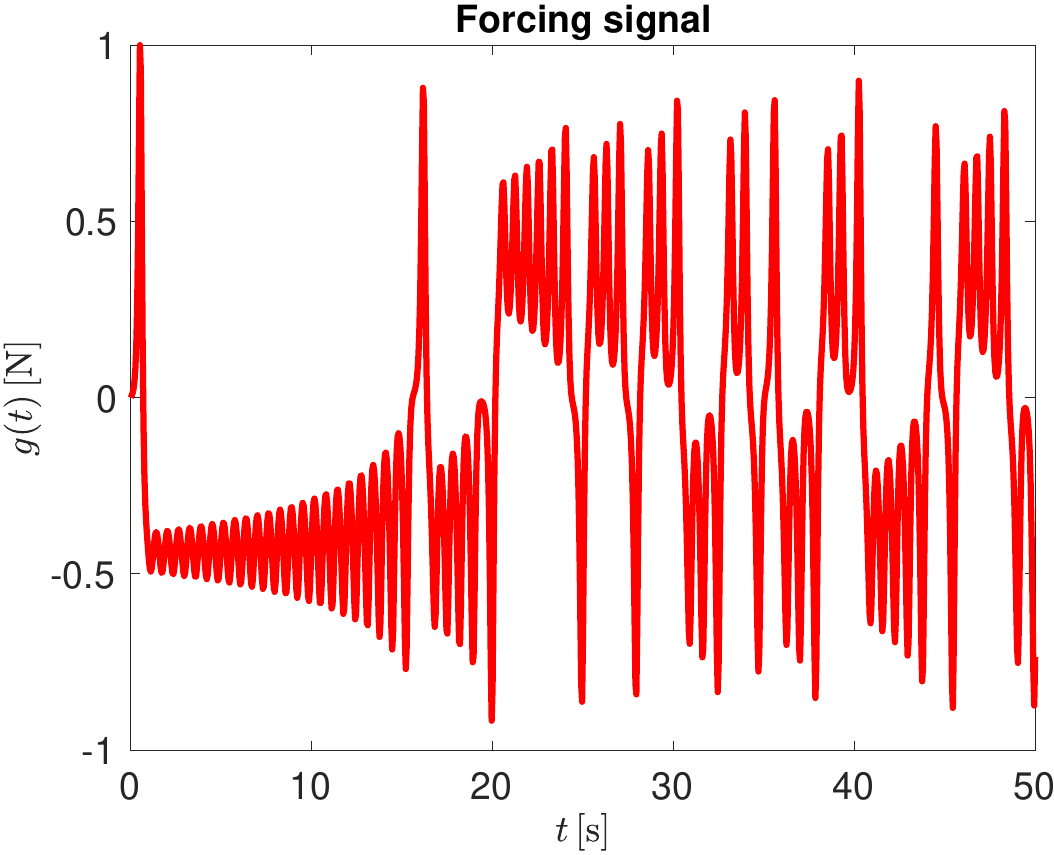}\label{fig:force_weak}}\caption{(a) The physical setup for the forced cart problem. (b) The chaotic
forcing signal $g(t)$, generated as the $x(t)$ component of the
numerically solved Lorenz system \eqref{eq:lorenz} with initial condition
$(0,0.3,0.5)$, computed over the time interval $[0,500]$ measured
in seconds. For times below $t<0$, we set $g(t)\equiv0.$}
\end{figure}

As the unforced \textbf{$\mathbf{x}=\mathbf{0}$} fixed point is asymptotically
stable and the forcing in uniformly bounded, this fixed point perturbs
for small enough $\left|M_{T}g(t)\right|$ into a nearby attracting
anchor trajectory $x^{*}(t)$ by Theorem \eqref{thm:x^* under conditions}.
From the asymptotic formula (\ref{eq:anchor trajectory expansion})
listed in the theorem, a $5^{th}$-order approximation for $x^{*}(t)$
has the terms

\begin{gather*}
x_{1}(t)=\int_{-\infty}^{t}e^{A(t-\tau)}f_{1}(\tau)d\tau,\qquad x_{2}(t)=0,\qquad x_{3}(t)=\int_{-\infty}^{t}e^{A(t-\tau)}\left(\begin{array}{c}
0\\
0\\
0\\
-\frac{\gamma(M_{f}+m_{1})}{m_{1}M_{f}}[x_{1}^{1}(\tau)]^{3}\\
\frac{\gamma}{m_{1}}[x_{1}^{1}(\tau)]^{3}\\
0
\end{array}\right)d\tau,
\end{gather*}
\begin{gather}
x_{4}(t)=0,\hfill x_{5}(t)=\int_{-\infty}^{t}e^{A(t-\tau)}\left(\begin{array}{c}
0\\
0\\
0\\
-\frac{\gamma(M_{f}+m_{1})}{m_{1}M_{f}}3[x_{1}^{1}(\tau)]^{2}x_{3}^{1}(\tau)\\
\frac{\gamma}{m_{1}}3[x_{1}^{1}(\tau)]^{2}x_{3}^{1}(\tau)\\
0
\end{array}\right)d\tau,\label{eq:anchor of shaken cart up to order 5}
\end{gather}
 with the external forcing $f_{1}(t)=\left(0,0,0,0,0,g(t)\right)^{\mathrm{T}}$
appearing in the first-order version of the equations motion.

Utilizing that the forcing signal $g(t)$ vanishes for $t<0$, we
evaluate the integrals in (\ref{eq:anchor of shaken cart up to order 5})
via the trapezoidal integration scheme of MATLAB. For sufficient accuracy,
we use forcing values evaluated at $10^{4}$ equally spaced times
to ensure accurate numerical integration. The terms $x_{\nu}(t)$
in the expansion for $x^{*}(t)$ are then spline-interpolated in time
to obtain an expression for $x^{*}(t)$ for arbitrary $t\in\left[0,500\right]$. 

To compute the non-autonomous SSM coefficients, we change coordinates
such that $x^{*}(t)$ serves as the origin of the transformed system
(\ref{eq:(u,v) system-1}). As a result, the matrices defined in eq.
(\ref{eq:A^u and A^v definition}) are of the form

\[
A^{u}=\left(\begin{array}{cc}
\lambda_{1} & 0\\
0 & \lambda_{2}
\end{array}\right),\quad A^{v}=\left(\begin{array}{cccc}
\lambda_{3} & 0 & 0 & 0\\
0 & \lambda_{4} & 0 & 0\\
0 & 0 & \lambda_{5} & 0\\
0 & 0 & 0 & \lambda_{6}
\end{array}\right),
\]
with $\lambda_{1,2}=-0.0227\pm0.3956i$ and $\lambda_{3,4}=-0.1234\pm1.2700i$,
$\lambda_{5,6}=-0.3788\pm1.6707i$. These eigenvalues satisfy the
non-resonance condition (\ref{eq:strict nonresonance}), and hence
by Theorem \ref{thm:computation of non-automous SSM}, there exists
a 2D autonomous SSM, \emph{$W_{\epsilon}(E,t)$, }that admits a truncated
Taylor expansion

\begin{equation}
v_{i}\approx\sum_{j,\mathbf{{m}=1}}^{j+|\mathbf{m}|=N}\epsilon^{j}u_{1}^{m_{1}}u_{2}^{m_{2}}h_{i}^{j\mathbf{{m}}}(t)\label{eq:v_i for shaken cart}
\end{equation}
for any $N\geq2$ in system (\ref{eq:(u,v) system-1}). Here we choose
the dimensional book-keeping parameter 
\[
\epsilon=\text{\ensuremath{\frac{\text{\ensuremath{\max\left|\mathbf{F}\left(t\right)\right|}}}{M_{f}+m_{1}+m_{2}}}}
\]
in computing the expansion (\ref{eq:v_i for shaken cart}). As already
noted, the actual ratio between external forces to internal forces
along trajectories is better represented by the dimensionless forcing
weakness parameter $r_{w}$.

Up to the order of truncation in eq. (\ref{eq:v_i for shaken cart}),
the two-dimensional SSM-reduced order model (\ref{eq:final reduced dynamics-2})
can be written in the complex coordinate $u_{1}$ as 

\begin{equation}
\dot{u}_{1}=\lambda_{1}u_{1}+f_{u_{1}}\left(u,\sum_{j,\mathbf{{m}=1}}^{j+|m|=5}\epsilon^{j}u_{1}^{m_{1}}u_{2}^{m_{2}}\mathbf{{h}}^{j\mathbf{{m}}}(t),x_{\epsilon}(t))\right).\label{eq:reduced dynamics weakly shaken cart}
\end{equation}
After computing the trajectory from formulas (\ref{eq:anchor of shaken cart up to order 5})
up to fifth order, we use the recursive formulas in statement (i)
of Theorem \ref{thm:computation of non-automous SSM} to compute the
coefficients $\mathbf{{h}}^{j\mathbf{{m}}}(t)$ in eq. (\ref{eq:reduced dynamics weakly shaken cart})
up to the same order recursively. Details for the coefficients in
the SSM-reduced model (\ref{eq:reduced dynamics weakly shaken cart})
can be found in the code available from \citet{kaundinya23}. 

To assess the performance of the truncated SSM-reduced model (\ref{eq:reduced dynamics weakly shaken cart})
for $N=5$, we first plot the normalized mean trajectory error computed
from 10 different initial conditions for $\text{{max}}||F||=0.06\,\left[N\right]$
in Fig. \ref{Figure:NMTE_REL}a. We compute the dependence of this
error on $\text{{max}}||F||$ in Fig. \ref{Figure:NMTE_REL}b. As
expected, the errors grow with $\text{{max}}||F||$ but remains an
order of magnitude less than $\text{{max}}||F||$.

\begin{figure}[H]
\subfloat[]{\includegraphics[width=0.52\textwidth]{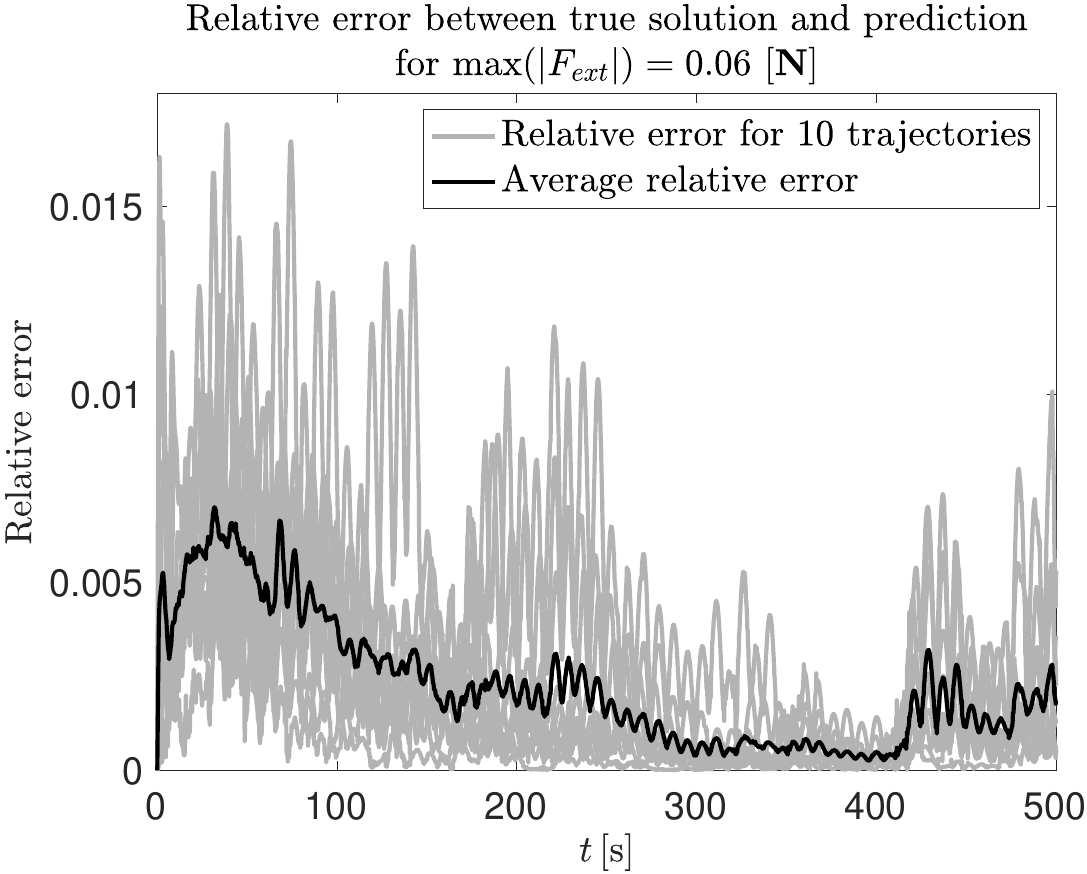}

}\subfloat[]{\includegraphics[width=0.48\textwidth]{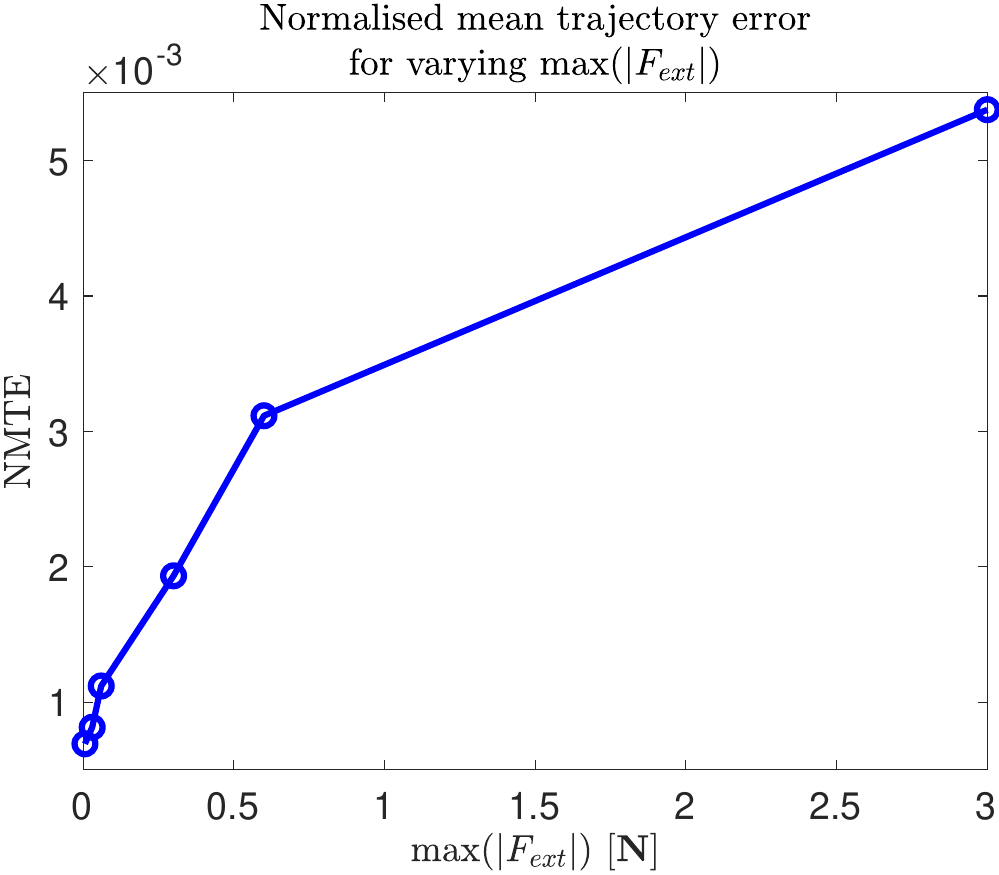}\label{Figure:NMTE_REL}}\caption{Error distribution of the SSM-reduced modeling error for the chaotically
shaken cart example. (a) Ensemble average (black) of the mean trajectory
errors (grey, normalized by the maximum of the maximal trajectory
amplitude) computed from 10 initial conditions under forcing with
$\text{{max}}||F||=0.06$ $\left[N\right]$. (b) The mean trajectory
error under varying $\text{{max}}||F||$ on trajectories starting
from the initial condition $u_{1}=u_{2}=1.2$ in all cases.}
\end{figure}

Probing larger forcing amplitudes, we still obtain accurate reduced-order
models up to $\text{{max}}||F||=3\,\left[N\right]$. As an illustration,
Fig. \ref{fig:xc_weak} uses the center-of-mass coordinate $x_{c}$
to show our $5^{th}$ order asymptotic approximation using Theorem
\ref{thm:x^* under conditions} for the anchor trajectory of the chaotic
SSM (blue), a simulation from the full model for a trajectory in the
2D SSM converging to the anchor point (black), and a prediction from
the 2D SSM-reduced model obtained from Theorem \ref{thm:computation of non-automous SSM}
for the same trajectory (red). 

In Fig. \ref{fig:Snap_Weak} (Multimedia available online), we also
plot snapshots of the chaotic 2D SSM, its anchor trajectory, a trajectory
from the full order model and the prediction for this trajectory from
the SSM-reduced order model at three different times. Note how both
the full trajectory and the predicted model trajectory synchronize
with the anchor trajectory of the SSM (see Figure \ref{fig:Snap_Weak}). 

\begin{figure}[t]
\centering{}\includegraphics[width=0.55\columnwidth]{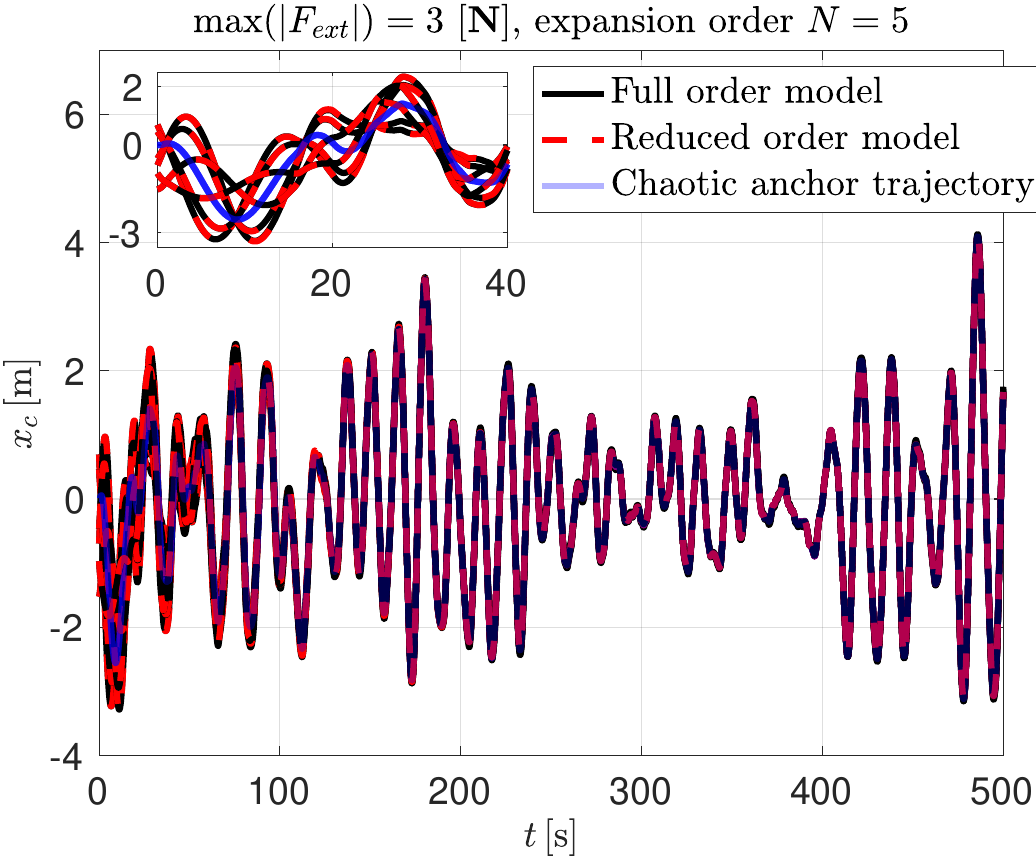}\caption{SSM-reduced model predictions and their verification in the chaotically
forced cart example (see the text for details). The non-dimensional
forcing weakness measure gives $r_{w}=7.8$, which is clearly outside
the small forcing regime. We also zoom in to show the initial transient
phase to highlight how the reduced order model captures the dynamics
of the full order model. \label{fig:xc_weak}}
\end{figure}

\begin{figure}[H]
\begin{centering}
\subfloat[]{\includegraphics[width=0.5\textwidth]{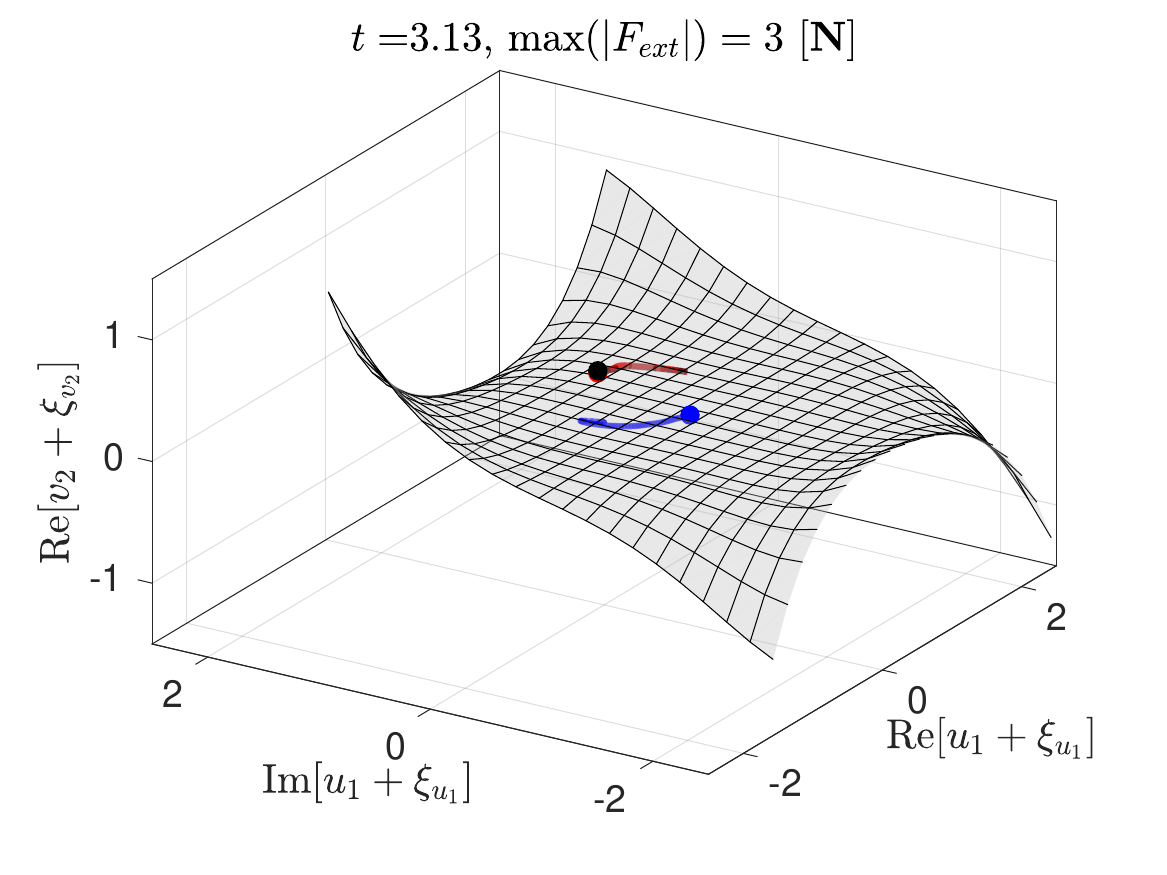}}\subfloat[]{\includegraphics[width=0.5\textwidth]{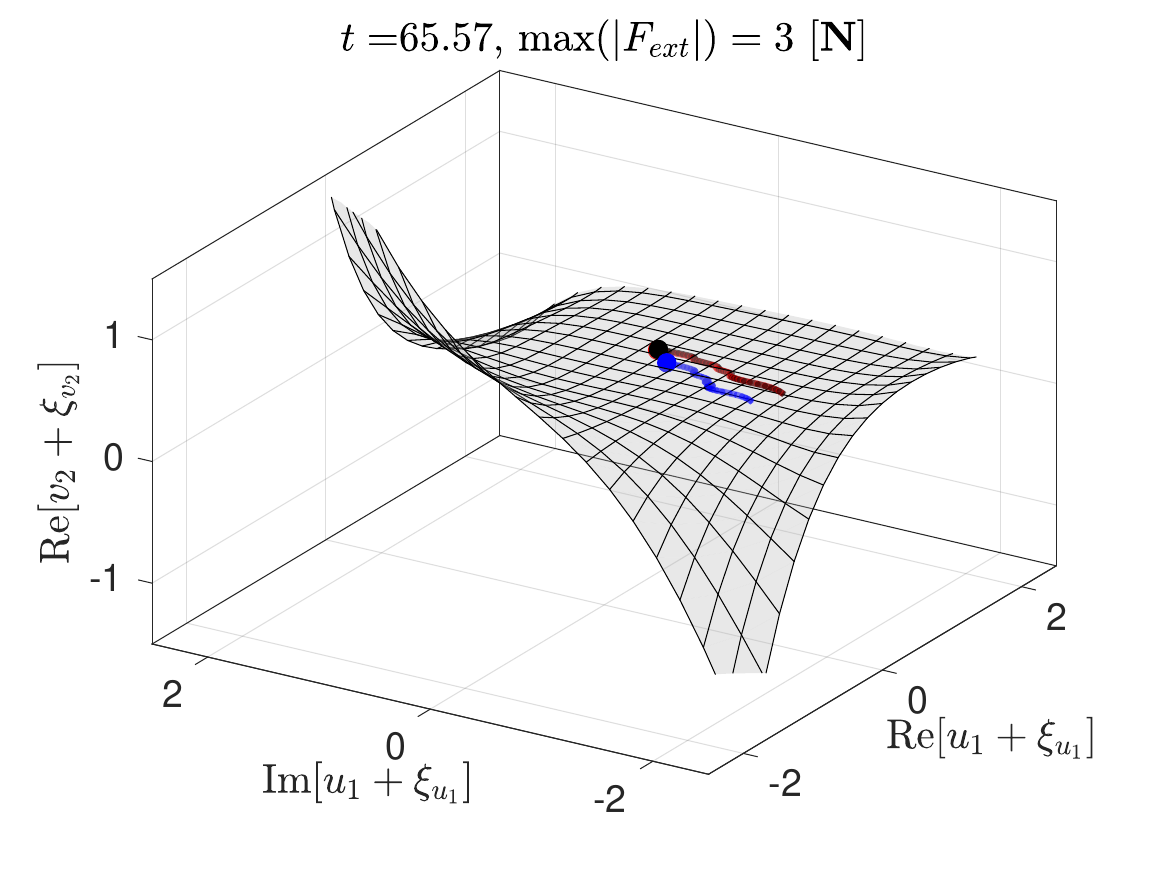}}
\par\end{centering}
\begin{centering}
\vfill{}
\par\end{centering}
\centering{}\subfloat[]{\includegraphics[width=0.5\textwidth]{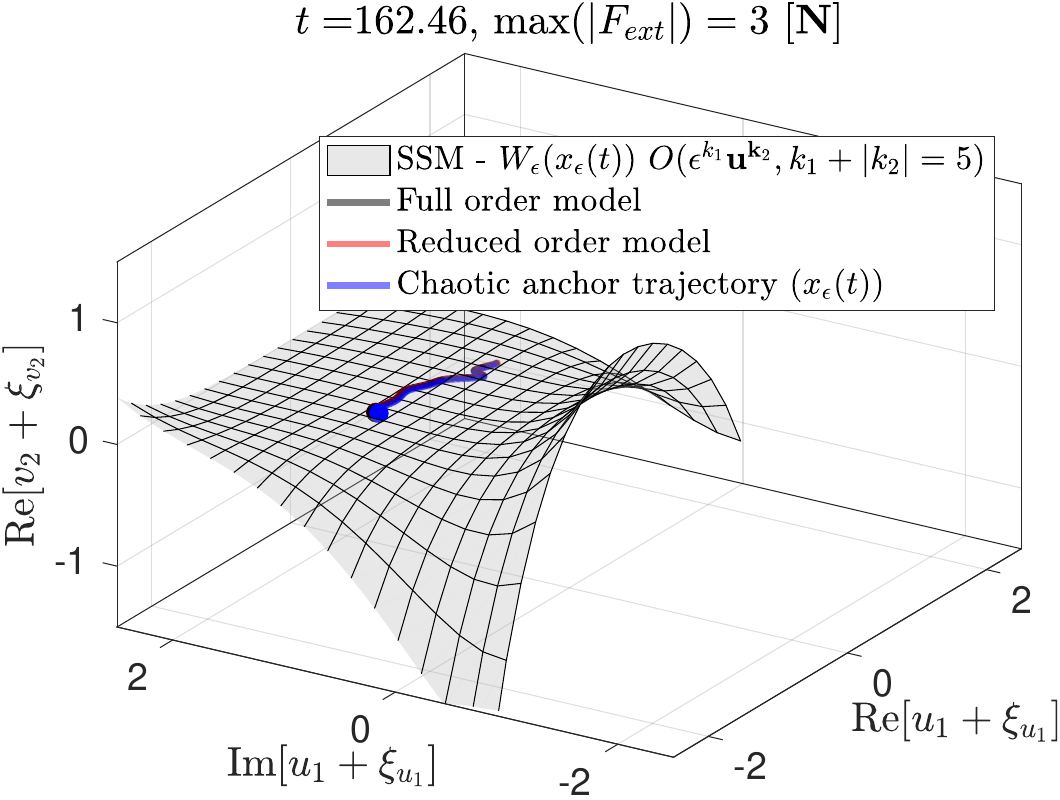}}\caption{(a)\textendash (c) Snapshots of the evolution of the 2D SSM, its anchor
trajectory and a trajectory of the SSM-reduced model shown in the
coordinates used in Theorem \ref{thm:computation of non-automous SSM}.
The non-dimensional forcing weakness measure is again $r_{w}=7.8$,
which is no longer small. See Movie 1 in the Supplementary Material
for the full evolution. (Multimedia available online) \label{fig:Snap_Weak}}
\end{figure}

Finally, we would like to illustrate that the improper integrals in
our asymptotic formulas for anchor trajectories and SSMs indeed converge
and give correct results even for discontinuous forcing, as noted
in Remark \ref{rem:discontonuous forcing}. To this end, we make the
forcing discontinuous at 10 random points in time, as shown in Figure
\ref{SubFig:Disc_Force}. The same formulas for the anchor trajectory,
its attached SSM and the SSM-reduced order model remain formally applicable
and continue to give accurate approximations even for the relatively
large forcing amplitude of $\text{{max}}||F||=3\left[\mathrm{N}\right]$,
as seen in Figure \ref{fig:Disc_Result}. 

\begin{figure}[H]
\subfloat[]{\includegraphics[width=0.5\columnwidth]{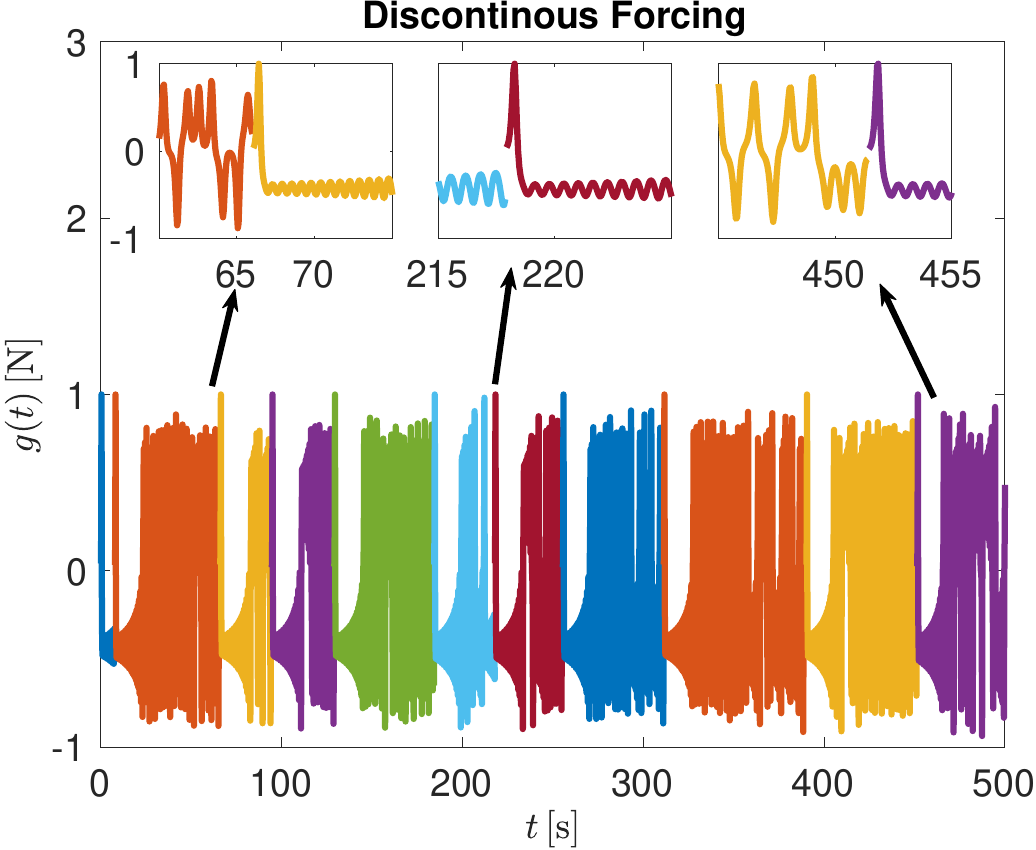}

\label{SubFig:Disc_Force}}\subfloat[]{\includegraphics[width=0.5\columnwidth]{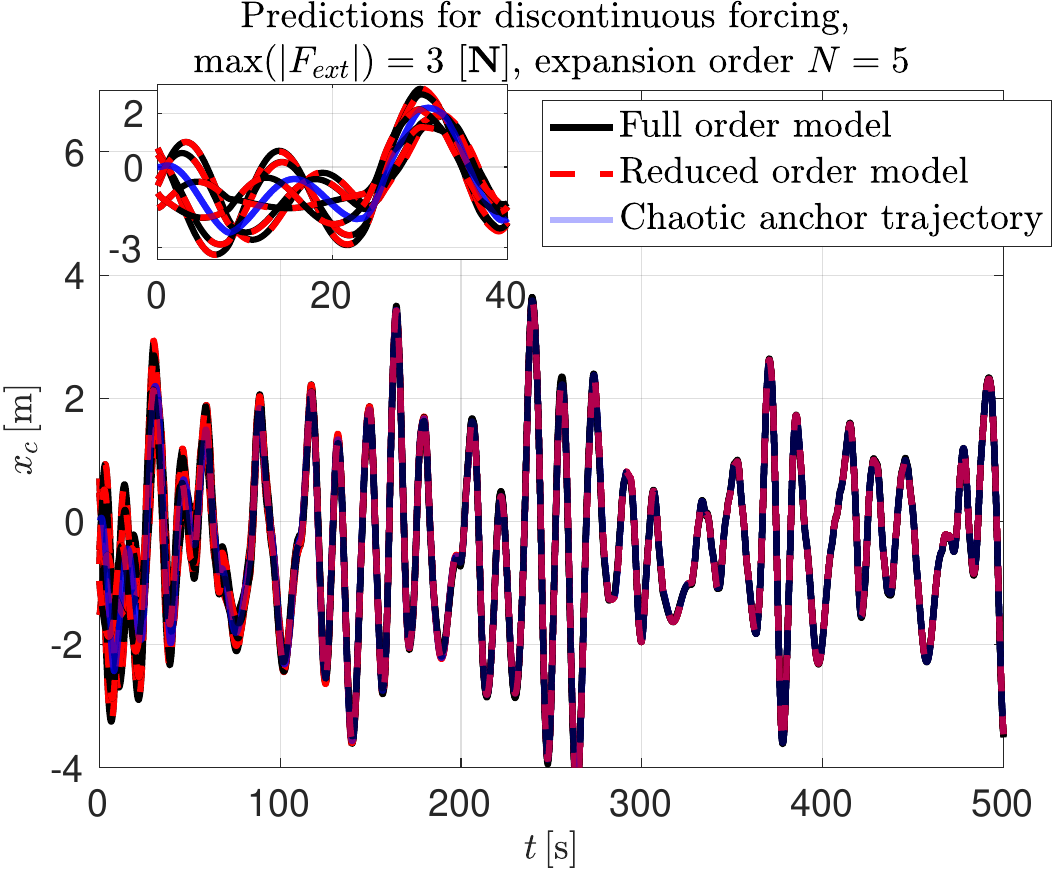}}\caption{(a) Discontinuous chaotic forcing profile for the cart example with
non-dimensional forcing weakness $r_{w}=7.8$. We set $g(t)\equiv0$
for $t<0$. (b) Same as Fig. \ref{fig:xc_weak} but for the discontinuous
chaotic forcing profile shown in the subplot (a).\label{fig:Disc_Result}}
\end{figure}

We close this subsection with a slight modification of our example
that adds more nonlinearity to the problem and underlines the utility
of the higher-order approximations we have developed for the anchor
trajectory (or generalized steady state) and the SSM attached to it.
For brevity, we will only illustrate the need for higher-order approximations
to the anchor trajectory by replacing the localized nonlinearity in
eq. (81) with the more general nonlinearity 
\begin{equation}
\mathbf{f}(\mathbf{{x},\dot{{x}}})=\left(\begin{array}{c}
\gamma q_{1}^{3}-\gamma_{f}\frac{{m_{1}+m_{2}}}{M_{T}}\beta(\mathbf{{x}})\\
-\gamma_{f}\frac{{m_{2}}}{M_{T}}\beta(\mathbf{{x}})\\
\gamma_{f}\beta(\mathbf{{x}})
\end{array}\right),\qquad\beta(\mathbf{{x}})=\Bigg(x_{c}-\frac{{(m_{1}+m_{2})}}{M_{T}}q_{1}-\frac{{m_{2}}}{M_{T}}q_{2}\Bigg)^{3}.\label{eq:new nonlinear spring}
\end{equation}
This makes the originally linear right-most spring between the cart
and the wall in Fig. 3 nonlinear with cubic nonlinear stiffness coefficient
$\gamma_{f}$. 

First, Fig. \ref{fig:xc_compare}a shows for $\gamma_{f}=0\mathrm{\left[N/m^{3}\right]}$
(i.e., for the linear limit of the right-most spring) the $\mathcal{O}\left(1\right)$
and $\mathcal{O}\left(11\right)$ asymptotic approximations to the
anchor trajectory of the chaotic SSM using the results of Theorem
1. Already in this case, the $\mathcal{O}\left(1\right)$ approximation
is noticeably improved by the $\mathcal{O}\left(11\right)$ approximation,
but the $\mathcal{O}\left(1\right)$ approximation is also close to
the actual generalized steady state.

\begin{figure}[H]
\subfloat[]{\includegraphics[width=0.5\columnwidth]{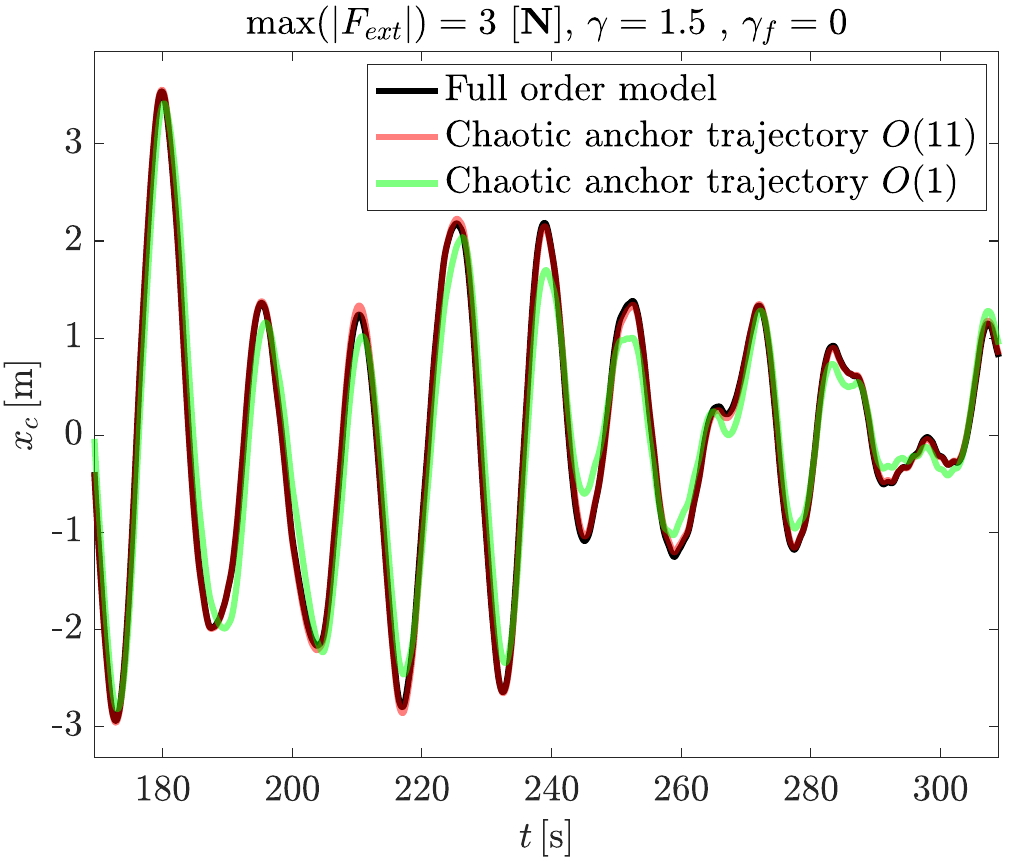}

}\subfloat[]{\includegraphics[width=0.5\columnwidth]{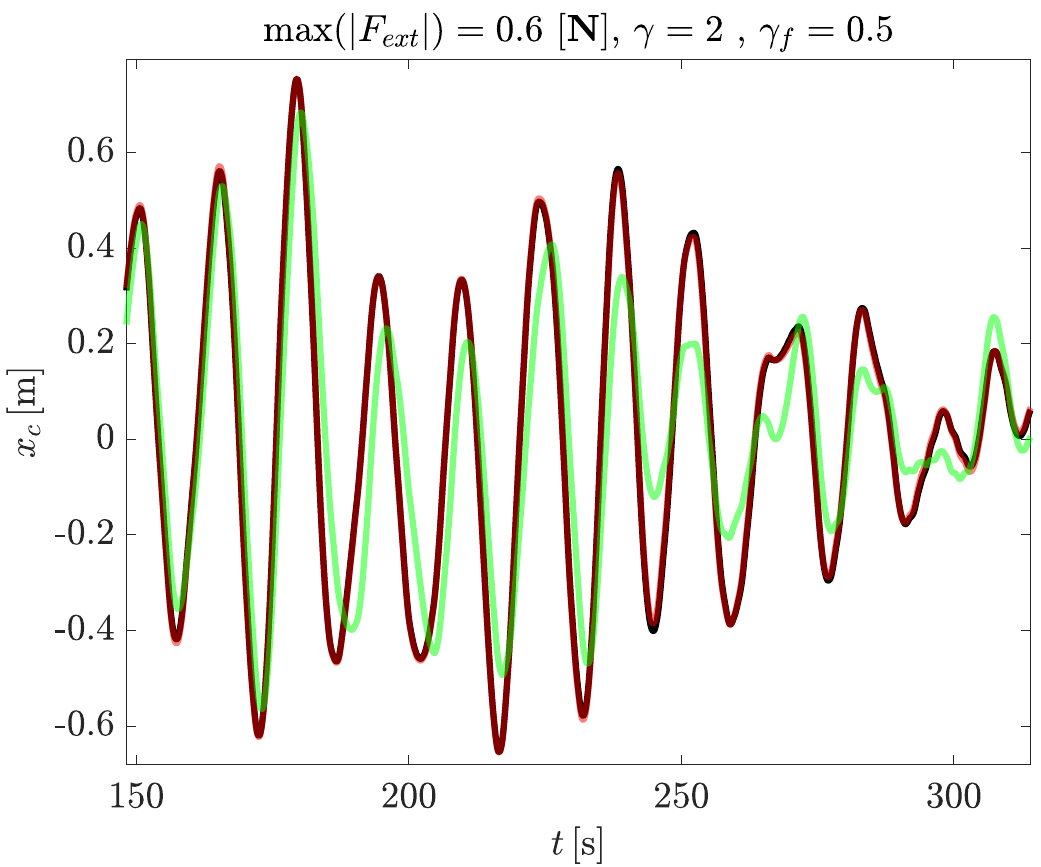}}\caption{Leading-order and higher-order approximation of the generalized steady
state of the cart system under weak chaotic forcing. (a) For $\text{{max}}||F||=3\left[\mathrm{N}\right]$
with non-dimensional forcing weakness $r_{w}=15.82$ and a linear
right-most spring ($\gamma_{f}=0\mathrm{\left[N/m^{3}\right]}$).
(b) For one-fifth of the forcing level in case (a) ($\text{{max}}||F||=0.6\left[\mathrm{N}\right]$,
$r_{w}=6$) but for a nonlinear right-most spring ($\gamma_{f}=0.5\mathrm{\left[N/m^{3}\right]}$)
\label{fig:xc_compare}}
\end{figure}

In contrast, Fig. \ref{fig:xc_compare}b shows a case in which much
smaller forcing is applied but the right-most spring is nonlinear
with $\gamma_{f}=0.5\mathrm{\left[N/m^{3}\right]}$. In this case,
more significant error develops between the actual anchor trajectory
and its $\mathcal{O}\left(1\right)$, while the $\mathcal{O}\left(11\right)$
approximation remains indistinguishably close to the actual anchor
trajectory. This example, therefore, underlines the need for higher
approximations to the anchor point in the presence of stronger nonlinearities.
As the SSM is attached to the anchor trajectory, the accuracy of an
SSM-reduced model also depends critically on this refined approximation.

\subsubsection{Slow forcing }

For the same cart system shown in Fig. \ref{Figure:Shaking_cart},
we select the physical parameters $m_{1}=m_{2}=1\,\mathrm{\left[kg\right]}$,
$M_{f}=2\,\left[\mathrm{kg}\right]$, $k=k_{f}=1\,\,\mathrm{\left[N/m\right]}$
and $c=c_{f}=0.3\,\mathrm{\left[Ns/m\right]}$. We now slow down the
previously applied chaotic forcing by replacing the external forcing
vector in eq. (\ref{eq:data for weakly forced cart}) with $\mathbf{{F}(\alpha)}=\left(0,0,N_{s}g(\alpha)\right)^{\mathrm{T}}.$
We select this forcing to be of large amplitude by letting $N_{s}=10$
and consider the phase range $\alpha\in[0,6]$. In our analysis, we
will start with the minimal forcing speed $\epsilon=\dot{\alpha}=0.001$,
for which obtain the non-dimensional forcing speed $r_{s}=0.72$ from
the second formula in (\ref{eq:weakness and slowness measures}).
This qualifies as moderately slow external forcing relative to the
speed of variation of internal forces along unperturbed trajectories.
In contrast, for the maximal forcing speed $\epsilon=\dot{\alpha}=0.008$
we will consider, we obtain $r_{s}=7.26$, which can no longer be
considered slow. Again, of particular interest will be how our asymptotic
formulas for adiabatic SSMs and their anchor points perform in the
latter case, which is clearly faster than what is normally considered
slowly varying in perturbation theory.

We start by solving for the fixed point $\mathbf{x}_{0}(\alpha)$
of the system under static forcing, as defined by the algebraic equation
in eq. (\ref{eq:hyperbolicity assumption for slo-fast}). In the present
example, this amounts to solving the equation

\begin{equation}
\mathbf{{K}}\mathbf{\mathbf{x}}_{0}+\mathbf{f}(\mathbf{x}_{0},\mathbf{0})=\mathbf{{F}}(\alpha),\label{eq:critical manifold for mechanical systems}
\end{equation}
which is obtained from the general forced equation of motion (\ref{eq:general forced equation of motion}).
We solve this nonlinear algebraic equation for the selected range
of $\alpha$ values by Newton iteration.

Once $x_{0}(\alpha)$ is available numerically, we evaluate the recursive
formula (\ref{eq:recursive formula for adiabatic anchor}) up to order
$N=3$ to obtain

\begin{gather*}
x_{1}(\alpha)=A^{-1}(\alpha)x'_{0}(\alpha),\quad x_{2}(\alpha)=A^{-1}(\alpha)\left(x'_{1}(\alpha)-\left(\begin{array}{c}
0\\
0\\
0\\
-\frac{\gamma(M_{f}+m_{1})}{m_{1}M_{f}}3[x_{1}^{1}(\alpha)]^{2}x_{0}^{1}(\alpha)\\
\frac{\gamma}{m_{1}}3[x_{1}^{1}(\alpha)]^{2}x_{0}^{1}(\alpha)\\
0
\end{array}\right)\right),
\end{gather*}

\begin{gather*}
x_{3}(\alpha)=A^{-1}(\alpha)\left(x'_{2}(\alpha)-\left(\begin{array}{c}
0\\
0\\
0\\
-\frac{\gamma(M_{f}+m_{1})}{m_{1}M_{f}}([x_{1}^{1}(\alpha)]^{3}+6x_{1}^{1}(\alpha)x_{0}^{1}(\alpha)x_{2}^{1}(\alpha))\\
\frac{\gamma}{m_{1}}([x_{1}^{1}(\alpha)]^{3}+6x_{1}^{1}(\alpha)x_{0}^{1}(\alpha)x_{2}^{1}(\alpha))\\
0
\end{array}\right)\right).
\end{gather*}

Having computed this approximation for the anchor trajectory, we change
to the coordinates defined in eq. (\ref{eq:u-v transformation,  adiabatic case}).
In those coordinates, we compute the $\alpha$-dependent matrices
defined in eq. (\ref{eq:alpa dependent matrices}) numerically. Their
general form in the present examples is 

\[
A^{u}=\left(\begin{array}{cc}
\lambda_{1}(\alpha) & 0\\
0 & \lambda_{2}(\alpha)
\end{array}\right),\quad A^{v}=\left(\begin{array}{cccc}
\lambda_{3}(\alpha) & 0 & 0 & 0\\
0 & \lambda_{4}(\alpha) & 0 & 0\\
0 & 0 & \lambda_{5}(\alpha) & 0\\
0 & 0 & 0 & \lambda_{6}(\alpha)
\end{array}\right).
\]
To perturb from this frozen-time limit, we set $\alpha=\epsilon\Omega t$
and $\Omega=1\text{{Hz}}$ to make the small parameter $\epsilon$
non-dimensional.

We set $N$= 3 and solve for all the $\alpha$ dependent coefficients
in the expansion of the 2D slow chaotic SSM by solving a recursive
linear algebraic equations (\ref{eq:h_=00007Bkp=00007D  adiabatic case})
order by order. We find that all 28 coefficients that describe the
SSM up to cubic order are nonzero. An important numerical step used
in this task is to perform differentiation with respect to $\alpha$.
For this purpose, we first need to re-orient the unit eigenvectors
originally returned by MATLAB for specific values of $\alpha$ to
make these eigenvectors smooth functions of $\alpha.$ We then perform
a central finite differencing which includes 4 adjacent points in
the $\alpha$ direction. We finally employ the Savitzky\textendash Golay
filtering function of MATLAB to smoothen the results obtained from
finite differencing an already finite-differenced signal. 

To illustrate the ultimate accuracy of these numerical procedures,
Table \ref{tab:Slow cart} shows the normalized mean trajectory error
for predictions from two different SSM-reduced models. We see that
under increasing $\epsilon$ (i.e., faster forcing), the errors still
remain an order-of-magnitude smaller then $\epsilon$. 
\begin{table}[H]
\begin{centering}
\begin{tabular}{|c|c|c|}
\cline{2-3} \cline{3-3} 
\multicolumn{1}{c|}{} & $O(\epsilon,j=1,|m|=3)$ & $O(\epsilon^{3},j+|m|=3)$\tabularnewline
\hline 
$\epsilon=0.001$~~ ($r_{s}=0.72$) & $9.8\times10^{-5}$ & $7.2\times10^{-6}$\tabularnewline
\hline 
$\epsilon=0.008$~~ ($r_{s}=7.26$) & $6.4\times10^{-3}$ & $7.8\times10^{-4}$\tabularnewline
\hline 
$\epsilon=0.010$~~ ($r_{s}=8.66$) & $1.1\times10^{-2}$ & $2.9\times10^{-3}$\tabularnewline
\hline 
\end{tabular}
\par\end{centering}
\caption{Normalized mean trajectory error for predictions from two different
SSM-reduced models under different $\epsilon$ values. The corresponding
non-dimensional forcing speed $r_{s}$ is also shown for each listed
value of $\epsilon$. \label{tab:Slow cart}}

\end{table}

As we did for weak forcing, we track the center of mass coordinate
$x_{c}$ to test the accuracy of our asymptotic formulas up to order
$N=3$ for $\epsilon=0.008$ in Fig. \ref{Fig:slow_forcing}. Note
how this cubic-order approximation of the SSM-reduced model already
tracks the full solution closely. Finally, as in the weakly forced
case, we also plot in Fig. \ref{fig:Snap_Slow-1} snapshots of the
evolution of the slow anchor trajectory, the attached chaotic SSM,
a simulated full-order trajectory and its prediction from the same
initial condition based on the SSM-reduced dynamics (Multimedia available
online). These snapshots further confirm the accuracy of our analytic
approximation formulas, now in several coordinate directions.

\begin{figure}[t]
\subfloat[]{\includegraphics[width=0.45\columnwidth]{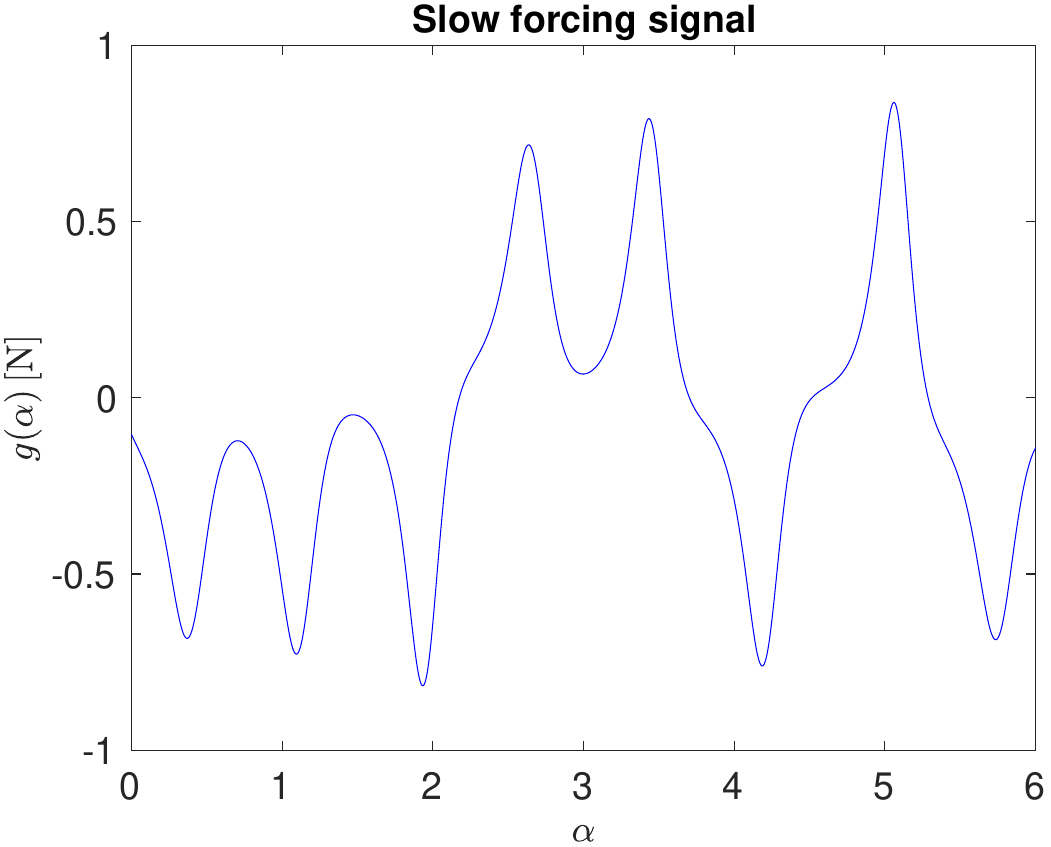}}\subfloat[]{\includegraphics[width=0.5\columnwidth]{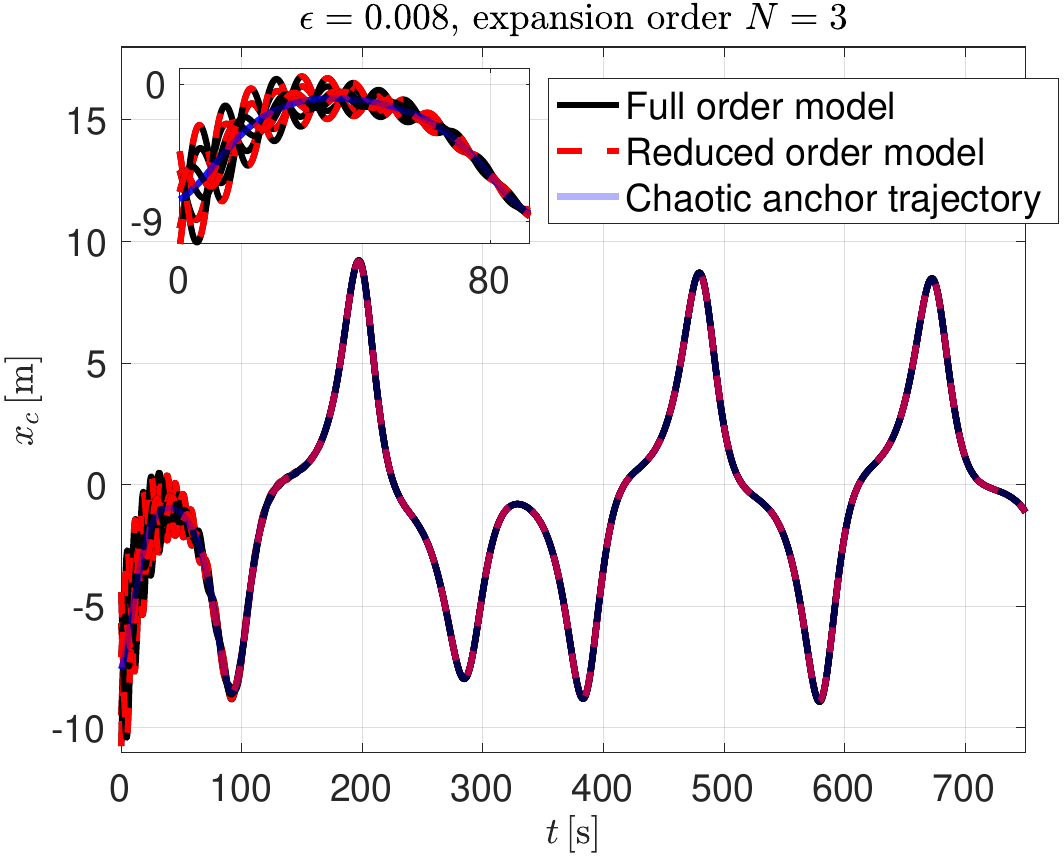}}

\caption{(a) Slow chaotic forcing signal $g(\epsilon t)$ for the cart example
for $\epsilon=0.008$, which amounts to $r_{s}=7.26$. (b) Assessment
of the accuracy of the anchor trajectory and the reduced dynamics
of the slow chaotic SSM in the cart example, based on the history
of the $x_{c}(t)$ coordinate of the center of mass. \label{Fig:slow_forcing}}
\end{figure}

\begin{figure}[H]
\subfloat[]{\includegraphics[width=0.5\textwidth]{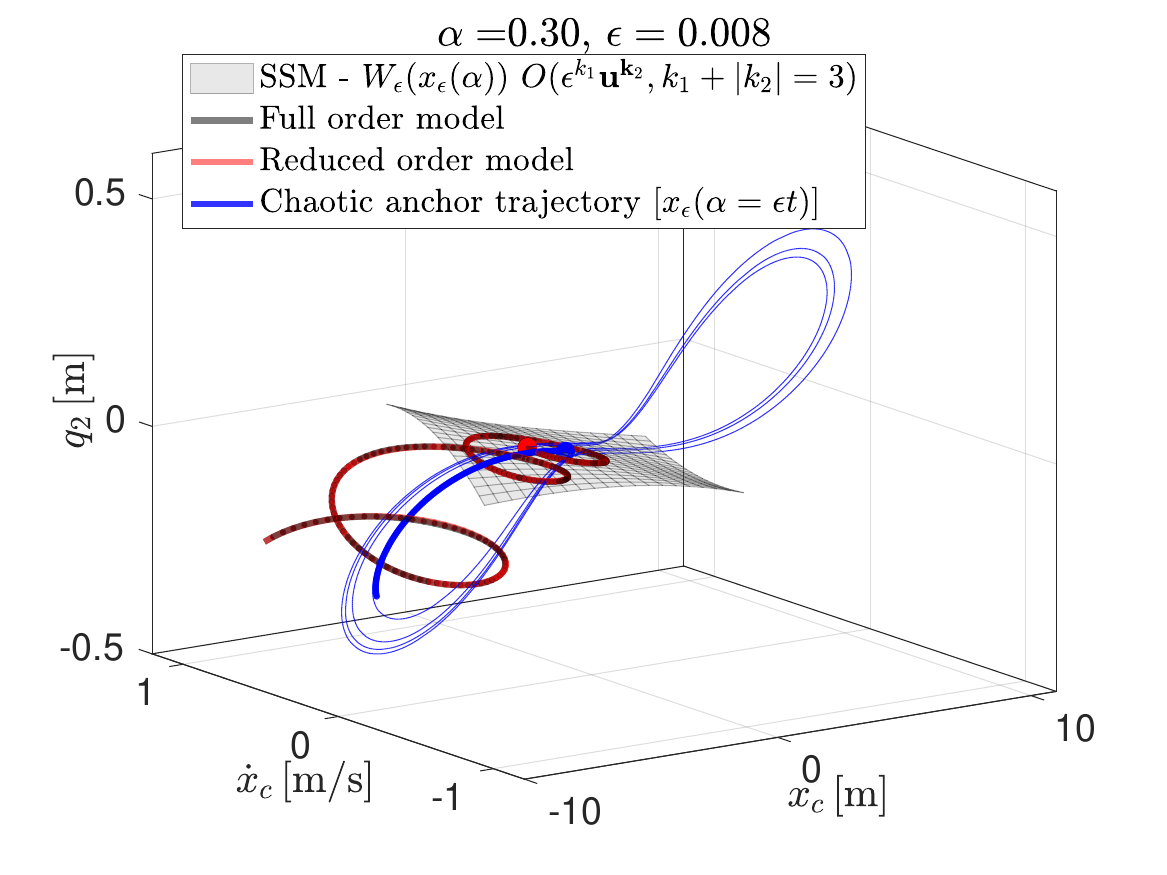}\label{fig:slow_snap1}}\subfloat[]{\includegraphics[width=0.5\textwidth]{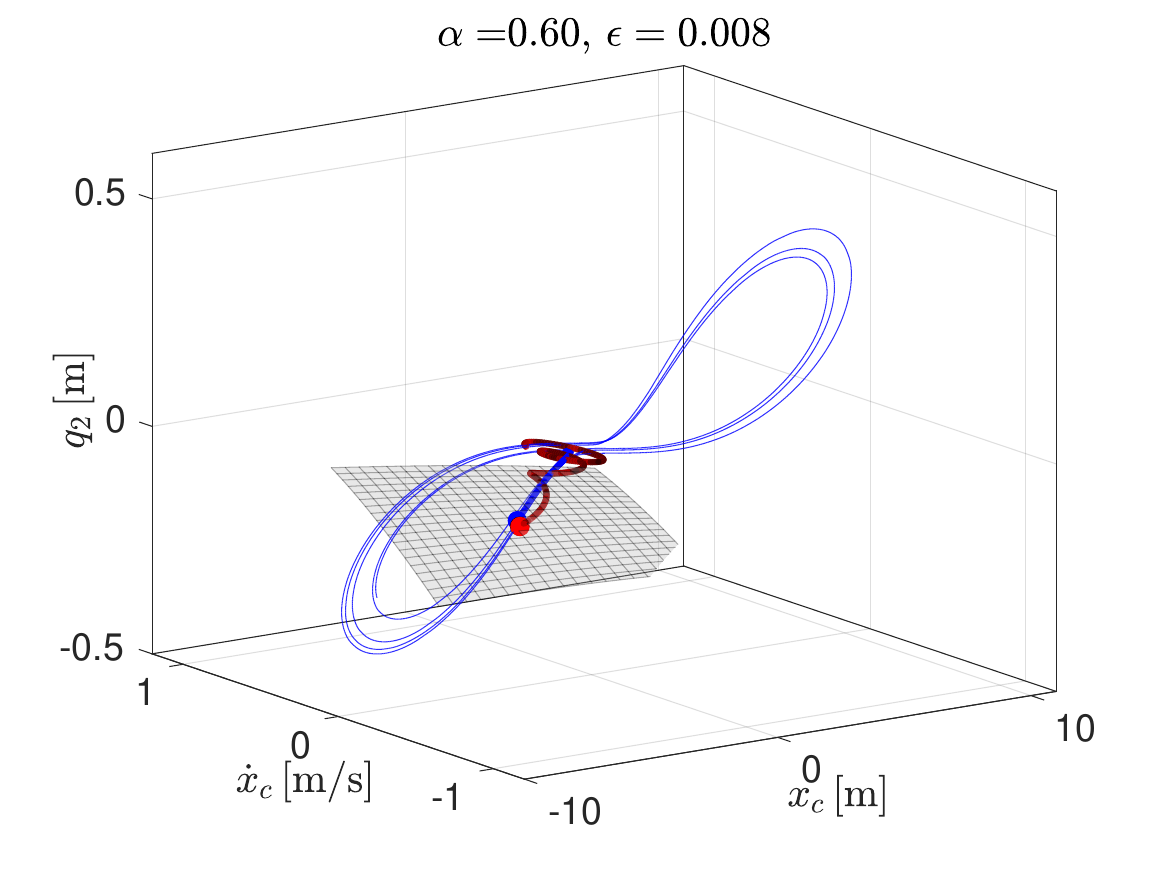}}

\vfill{}

\centering{}\subfloat[]{\includegraphics[width=0.5\textwidth]{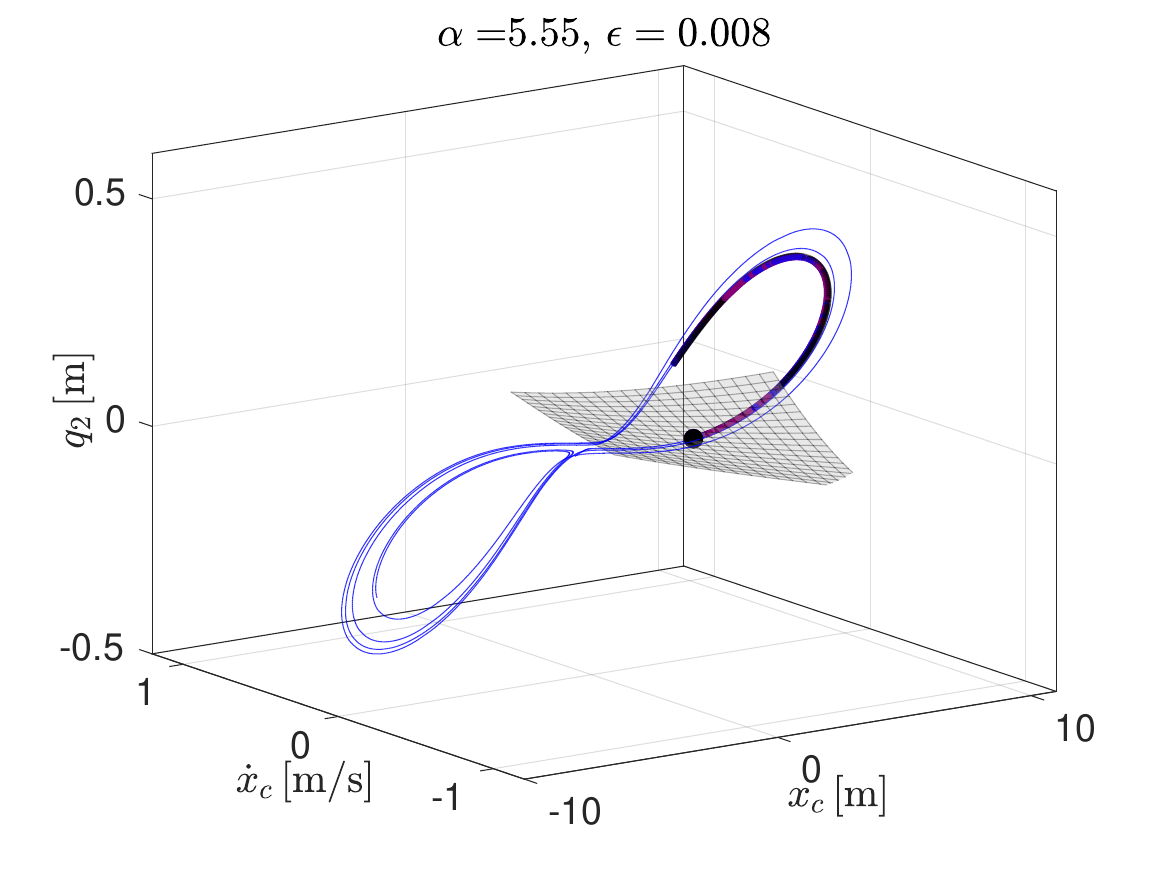}}\caption{(a)\textendash (c) Snapshots of the evolution of the slow chaotic
SSM, its anchor trajectory, the true solution of the full system and
the SSM-reduced model prediction for the true solution of the chaotically
shaken cart, all represented in the physical coordinates $[q_{2},x_{c},\dot{{x}}_{c}]$.
The non-dimensional forcing speed is $r_{s}=7.26$. See Movie 2 in
the Supplementary Material for the full evolution. (Multimedia available
online) \label{fig:Snap_Slow-1}}
\end{figure}

\subsection{Example 2: Chaotically forced bumpy rail\label{subsec:Bumbed real example}}

\subsubsection{Weak forcing}

The physical setup of our second example, a particle moving on a shaken
rail with a bump in the middle, is shown in Fig. \ref{fig:Mixed_mode}.
A major difference from our previous mechanical example is that the
origin of the unforced system is now an unstable, saddle-focus-type
fixed point. Therefore, under weak chaotic forcing, a nearby chaotic
SSM of mixed-mode type is expected to arise attached to a saddle-focus-type
chaotic anchor trajectory.

We use the center of mass position $x_{c}$ and the relative position
$x$ of the smaller mass $m$ as generalized coordinates. In terms
of these coordinates, the quantities featured in the general equation
of motion (\ref{eq:general mechanical system}) take the form

\[
\mathbf{M}={\scriptscriptstyle \left(\begin{array}{cc}
M_{f}+m & 0\\
0 & \frac{mM_{f}}{m+M_{f}}
\end{array}\right)},\quad\mathbf{K}={\scriptstyle \left(\begin{array}{cc}
k_{f} & -k_{f}\frac{m}{M_{f}+m}\\
-k_{f}\frac{m}{M_{f}+m} & -4\beta a^{2}mg+k_{f}\frac{m^{2}}{(M_{f}+m)^{2}}
\end{array}\right)},
\]

\[
\mathbf{C}={\scriptstyle \left(\begin{array}{cc}
c_{f} & -c_{f}\frac{m}{M_{f}+m}\\
-c_{f}\frac{m}{M_{f}+m} & c+c_{f}\frac{m^{2}}{(M_{f}+m)^{2}}
\end{array}\right)},\quad\mathbf{f}(\mathbf{x},\dot{\mathbf{x}})={\scriptstyle \left(\begin{array}{c}
0\\
4g\beta mx^{3}+16m\beta^{2}a^{4}x\dot{{x}}^{2}
\end{array}\right)},
\]
 
\[
\mathbf{F}(t)={\scriptscriptstyle \left(\begin{array}{c}
(m+M_{f})g(t)\\
0
\end{array}\right)},
\]
with for the coordinate vector $\mathbf{x}=\left(x_{c},x\right)^{\mathrm{T}}$.
The parameter $\beta$ controls the height of the bumpy rail and $a$
the distance to the wells. The parameter $c$ models linear damping
of the mass $m$ due to the rail. The velocity-dependent nonlinearity
$\mathbf{f}(\mathbf{x},\dot{\mathbf{x}})$ is the result of expanding
the exact equations of motion with respect to the height of the rail
and keeping only the leading-order nonlinearities.

We set $m=1\,\mathrm{\left[kg\right]}$, $M_{f}=4\mathrm{\,\left[kg\right]}$,
$k_{f}=1\,\mathrm{\left[N/m\right]}$, $c_{f}=c=0.3$~$\left[\mathrm{Ns/m}\right]$,
$a=0.3\,\left[\mathrm{m}\right]$, $\beta=\frac{1}{5a^{3}}$ and $g=9.8$~$\left[\mathrm{m/s^{2}}\right]$
. The forcing signal $g(t)$ will be the same as in the cart example.
We scale the chaotic forcing signal in a way that the non-dimensional
forcing weakness measure defined in formula (\eqref{eq:weakness and slowness measures})
gives $r_{w}=44.6$ when evaluated on the phase space region containing
the trajectories used in our analysis. Therefore, the external force
here is large relative to the internal forces of the system.

This forcing magnitude is outside the small range in which Theorem
\ref{thm:computation of non-automous SSM} strictly guarantees the
existence of a saddle-type anchor trajectory for a mixed-mode chaotic
SSM in this problem. We nevertheless evaluate our asymptotic expansions
for the anchor trajectory, as those expansions hold as long as the
SSM smoothly persist under increasing forcing. Specifically, a cubic-order
approximation $x^{*}(t)\approx\sum_{\nu=1}^{3}x_{\nu}(t)$ for the
anchor trajectory $x^{*}(t)$ from formula (\ref{eq:anchor trajectory expansion})
requires the functions

\begin{gather*}
x_{1}(t)=\int_{-\infty}^{\infty}G(t-\tau)f_{1}(\tau)d\tau,\qquad x_{2}(t)\equiv0,
\end{gather*}

\[
x_{3}(t)=\int_{-\infty}^{\infty}G(t-\tau)\left(\begin{array}{c}
0\\
0\\
0\\
-4g(1+\frac{m}{M})\beta[x_{1}^{2}(\tau)]^{3}-16a^{4}(1+\frac{m}{M})\beta^{2}x_{1}^{2}(\tau)[x_{1}^{4}(\tau)]^{2}
\end{array}\right)d\tau,
\]
with the Green's function $G(t)$ defined in formula (\ref{eq:Green's function definition})
for saddle-type anchor trajectories. The external forcing appears
now as $f_{1}(t)=\left(0,g(t),0,0\right)^{\mathrm{T}}$ in the first-order
formulation of the equation of motions. We calculate these integrals
using the same numerical methodology as in our first cart example.
In this example, the calculation of $G(t-\tau)$ from (\ref{eq:Green's function definition})
involves the matrices 

\[
A^{u}=\left(\begin{array}{cc}
\lambda_{1} & 0\\
0 & \lambda_{2}
\end{array}\right),\quad A^{v}=\left(\begin{array}{cc}
\lambda_{3} & 0\\
0 & \lambda_{4}
\end{array}\right),
\]
where the eigenvalues of the unforced saddle-focus-type fixed point
at the origin are $\lambda_{1}=5.5196$, $\lambda_{2}=-5.9094$, $\lambda_{3,4}=-0.0301\pm0.4465i$.
The 2D non-autonomous SSM can be constructed as a graph over the mixed-mode
tangent space corresponding to the real eigenvalues $\lambda_{1}$
and $\lambda_{2}$. 

Without listing the details here, we also perform a similar calculation
for the two attracting anchor trajectories created by the chaotic
forcing from the two asymptotically stable fixed points of the unperturbed
system that lie at the two bottom points of the bumpy rail. These
two stable anchor trajectories will represent two chaotic attractors
for the forced system. Trajectories initialized near the unstable
anchor trajectory on its attached 2D mixed-mode SSM will converge
to one of these chaotic attractors. This is indeed observable in Fig.
\ref{Fig:mixed-mode-weak}, in which we launch several trajectories
near the unstable anchor trajectory (red) within the 2D mixed-mode
SSM approximated up to cubic order. These trajectories then converge
to one of the two predicted attracting chaotic anchor trajectories
(blue) which we have computed up to cubic order of accuracy from our
formulas.

The instability of the anchor trajectory perturbing from the origin
makes it challenging to verify the accuracy of our local expansion
(\ref{eq:h_=00007Bkp=00007D  adiabatic case}) for the 2D mixed-mode
SSM attached to this trajectory. We can nevertheless verify the local
accuracy for these formulas by tracking nearby trajectories launched
from the predicted 2D SSM and confirm in Fig. \ref{fig:snap_mixed_mode_weak}
(Multimedia available online) that those trajectories evolve together
with the predicted SSM until ejected from a neighborhood of the unstable
anchor trajectory. 

\begin{figure}[H]
\subfloat[]{\includegraphics[width=0.45\columnwidth]{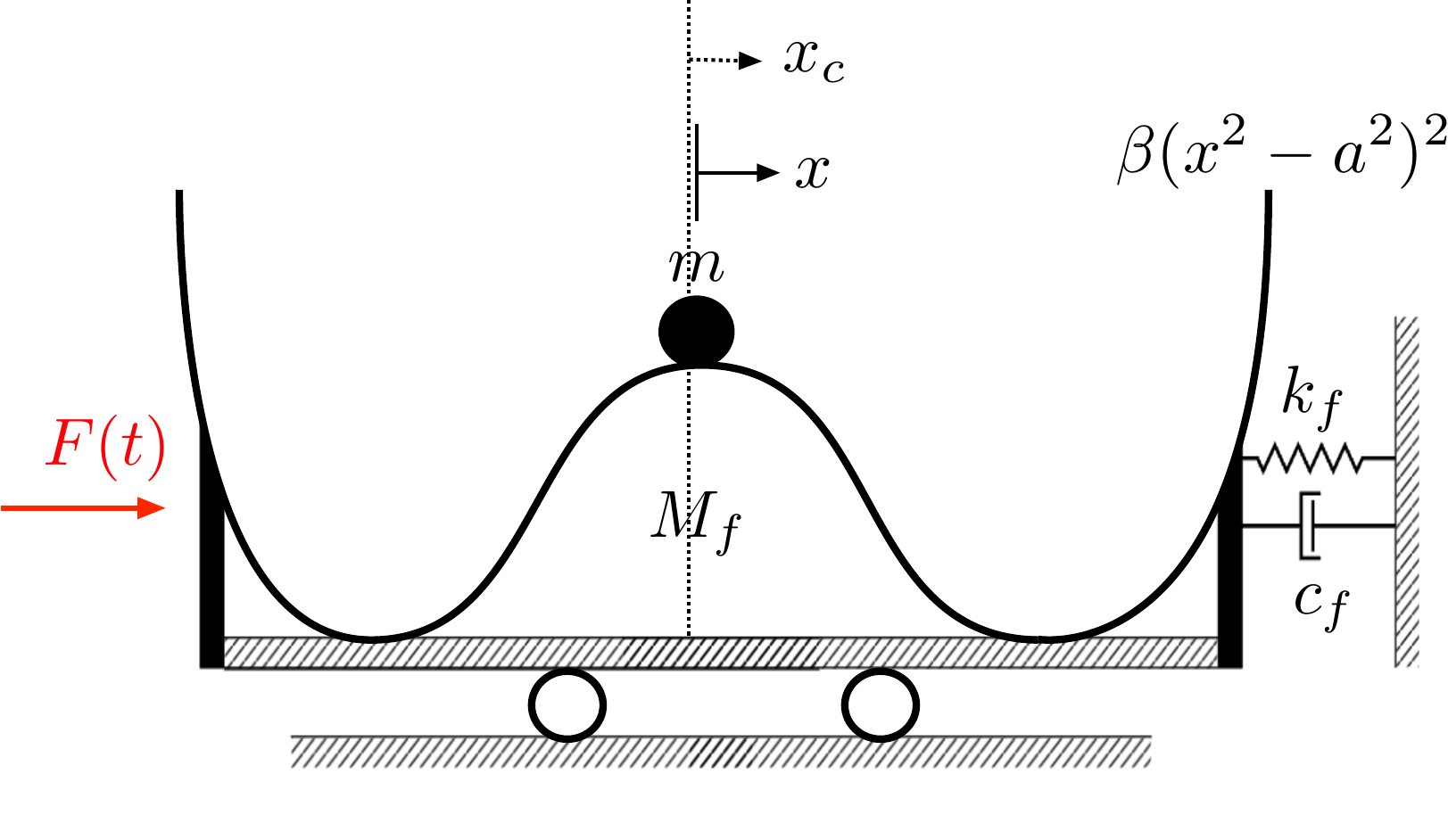}\label{fig:Mixed_mode}}\hfill{}\subfloat[]{\includegraphics[width=0.5\columnwidth]{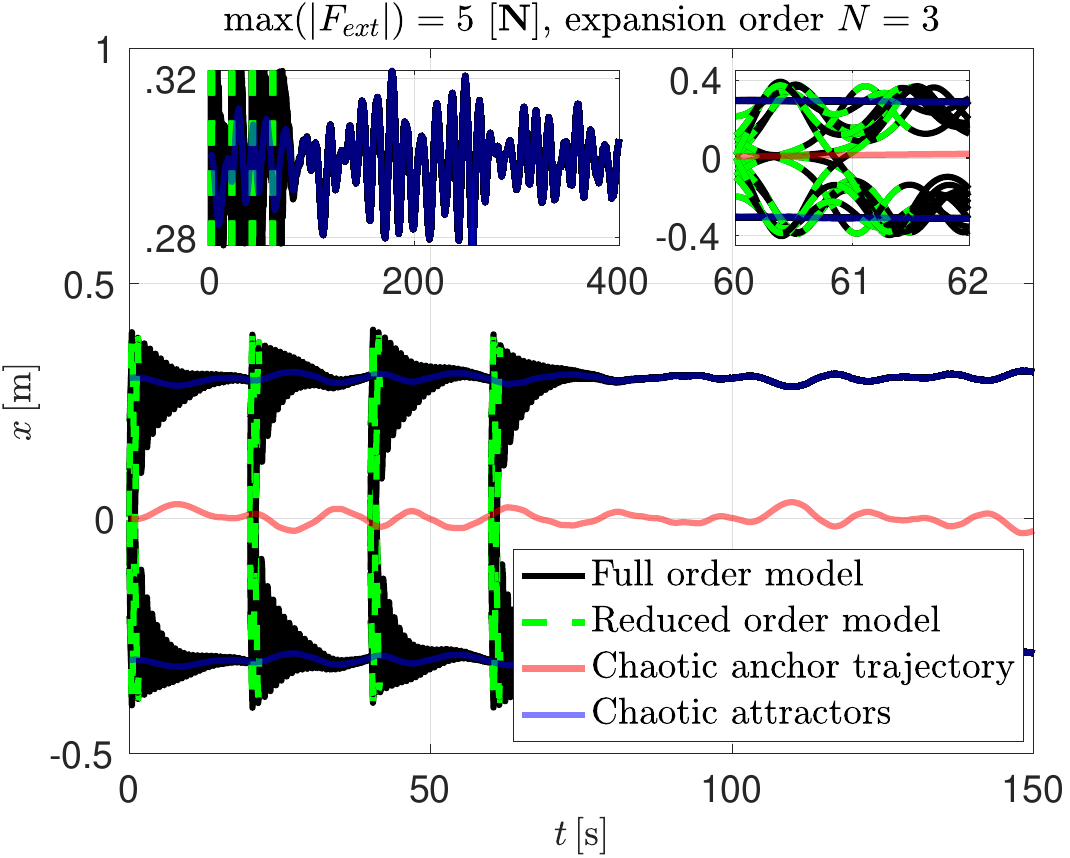}\label{Fig:mixed-mode-weak}}\caption{(a) Setup for the chaotically forced bumpy real example (b) Plots
of 16 initial conditions on local mixed-mode non-autonomous 2D SSM,
released at four different times. The left inset shows how the transient
behavior of the full system matches those of the anchor trajectories
calculated from our theory. The right inset shows fast convergence
of the full solution to the anchor trajectories together with local
predictions of the reduced order model. The non-dimensional forcing
weakness measure in this example is $r_{w}=44.6$.}
\end{figure}

\begin{figure}[H]
\subfloat[]{\includegraphics[width=0.5\textwidth]{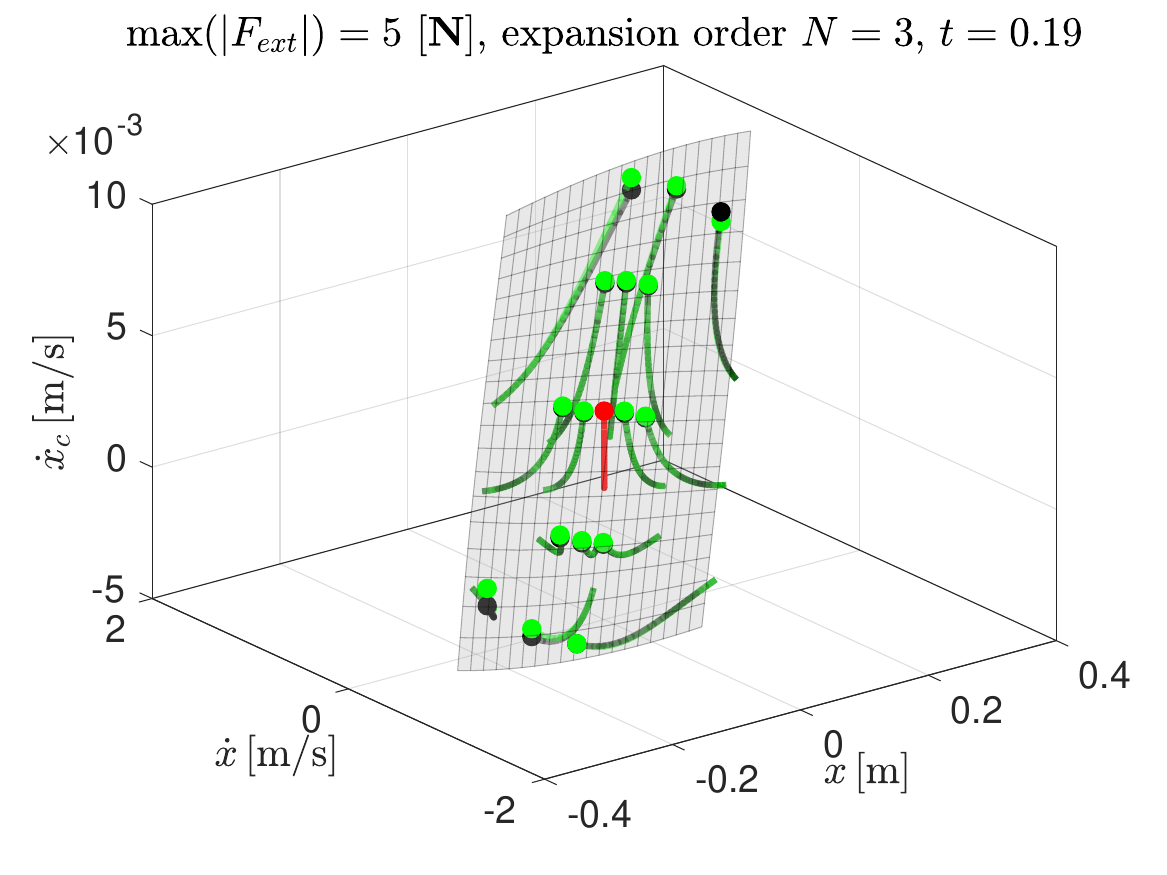}}\subfloat[]{\includegraphics[width=0.5\textwidth]{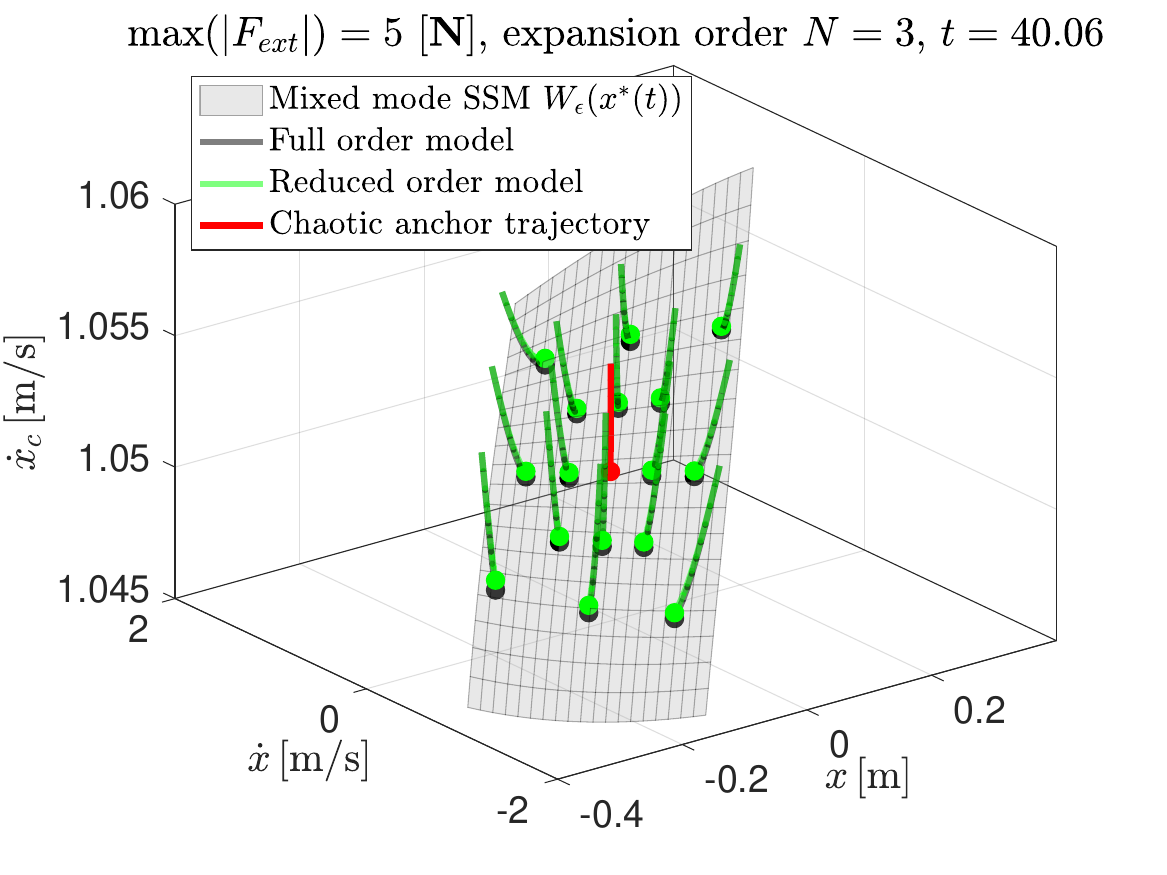}}\caption{(a)\textendash (b) SSM dynamics in the (not so) weakly forced bumpy
rail example. Shown are snapshots of the predicted 2D mixed-mode SSM
(grey), its anchor trajectory (red), trajectories predicted by the
SSM-reduced model (green) and full system simulations of trajectories
launched from the same initial conditions (black). The non-dimensional
forcing weakness measure is again $r_{w}=44.6$. See Movie 3 in the
Supplementary Material for the full evolution. (Multimedia available
online)\label{fig:snap_mixed_mode_weak}}
\end{figure}

\subsubsection{Slow forcing\label{subsec:Slow-forced bumpy rail}}

In the same bumpy rail problem but now with height $\beta=\frac{1}{10^{3}}$,
we rescale the forcing in time to make it slowly varying. The slow
forcing vector is given by $\mathbf{{F}(\alpha)}=\left(\begin{array}{c}
N_{s}g(\alpha)\\
0
\end{array}\right)$ with $N_{s}=3$ and the range of interest is $\alpha\in[0,6].$ We
will perform simulations for the forcing speed $\epsilon=\dot{\alpha}=0.01$,
which yields the non-dimensional forcing speed measure $r_{s}=1.22$,
signaling moderately fast forcing.

As in our first example, we compute the zeroth order term $x_{0}(\alpha)$
in the expansion of the slow anchor trajectory for the SSM by solving
an algebraic equation of the form (\ref{eq:critical manifold for mechanical systems})
by Newton iteration. 

\[
\mathbf{{K}}\mathbf{{x}}_{0}+\mathbf{f}(\mathbf{x}_{0},\mathbf{0})=\mathbf{{F}}(\alpha).
\]
Up to cubic order, the expansion (\ref{eq:recursive formula for adiabatic anchor})
for the slow, unstable anchor trajectory perturbing from the origin
under the slow forcing can be written as 

\begin{gather*}
x_{1}(\alpha)=A^{-1}(\alpha)x'_{0}(\alpha),\\
x_{2}(\alpha)={\scriptstyle A^{-1}(\alpha)\left(x'_{1}(\alpha)-\left(\begin{array}{c}
0\\
0\\
0\\
(1+\frac{m}{M})\left(-12g\beta x_{0}^{2}(\alpha)[x_{1}^{2}(\alpha)]^{2}-16a^{4}\beta^{2}x_{0}^{2}(\alpha)[x_{1}^{4}(\alpha)]^{2}\right)
\end{array}\right)\right)},
\end{gather*}

\begin{gather*}
x_{3}(\alpha)={\scriptscriptstyle {\scriptstyle A^{-1}(\alpha)\left(x'_{2}(\alpha)-\left(\begin{array}{c}
0\\
0\\
0\\
(1+\frac{m}{M})\times\\
{\scriptstyle \left(-4g\beta\left([x_{1}^{2}(\alpha)]^{3}+6x_{0}^{2}(\alpha)x_{1}^{2}(\alpha)x_{2}^{2}(\alpha)\right)-16a^{4}\beta^{2}\left(x_{1}^{2}(\alpha)[x_{1}^{4}(\alpha)]^{2}+2x_{0}^{2}(\alpha)x_{1}^{4}(\alpha)x_{2}^{4}(\alpha)\right)\right)}
\end{array}\right)\right)}}.
\end{gather*}

With this approximation for the anchor trajectory at hand, we can
change to the coordinates defined in eq. (\ref{eq:u-v transformation,  adiabatic case}).
In those coordinates, we compute the $\alpha$-dependent matrices
defined in eq. (\ref{eq:alpa dependent matrices}) now arise as

\[
A^{u}=\left(\begin{array}{cc}
\lambda_{1}(\alpha) & 0\\
0 & \lambda_{2}(\alpha)
\end{array}\right),\quad A^{v}=\left(\begin{array}{cc}
\lambda_{3}(\alpha) & 0\\
0 & \lambda_{4}(\alpha)
\end{array}\right),
\]
where we again set $\alpha=\epsilon\Omega t$ with $\Omega=1\text{ {Hz}}$,
as in our first example.

As in the weakly forced version of this example, we plot the predictions
from the 2D SSM-reduced order model and compare it to simulations
of the full system in Fig. \ref{fig:slow_mixed_x} (Multimedia available
online). We again see close agreement of the reduced dynamics with
the true solution in the vicinity of the slow, unstable anchor trajectory,
prior to the convergence of the trajectories to the two stable, chaotic
anchor trajectories arising from the two stable unforced equilibria
along the rail. Finally, as in our first example, we show snapshots
of the projected 2D slow, mixed-mode SSM along with some trajectories
of its restricted dynamics. These agains stay close locally to trajectories
of the full system that are launched from the same initial conditions,
thereby confirming the expectations put forward in Remark \ref{rem: Applicability-to-slow-mixed-mode-SSM}.

\begin{figure}[H]
\begin{centering}
\subfloat[]{\includegraphics[width=0.5\columnwidth]{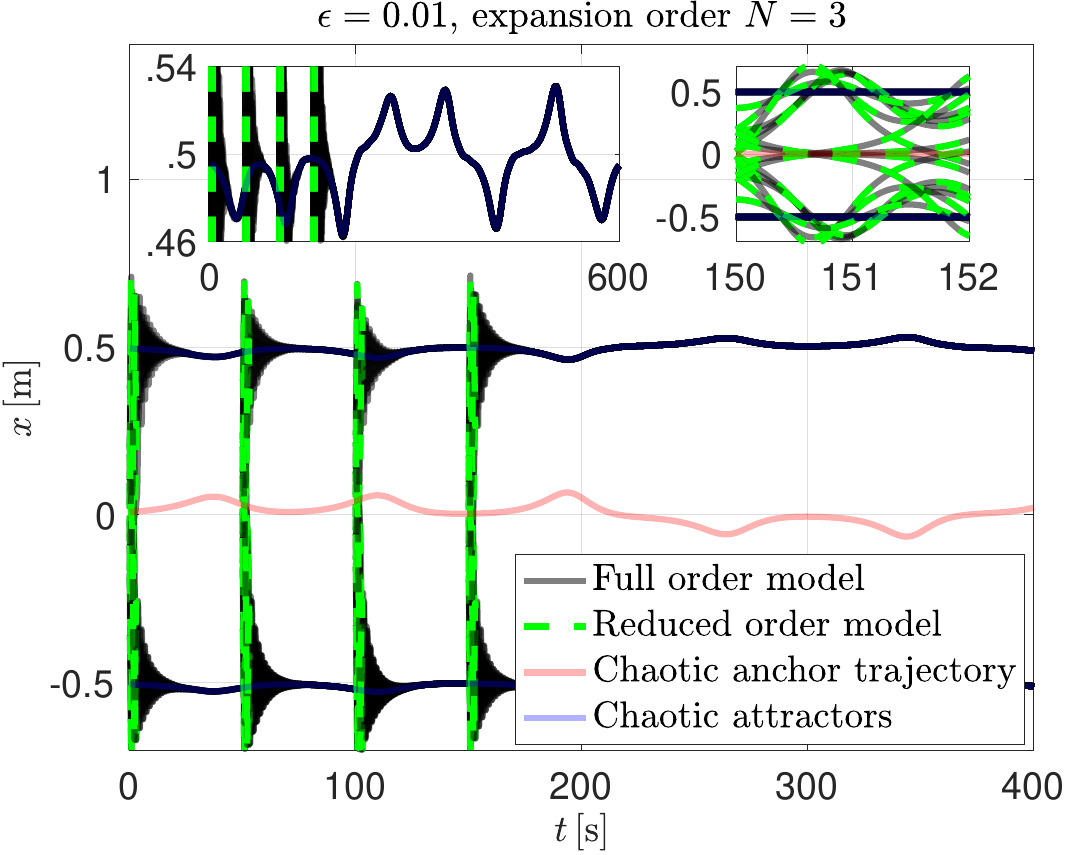}\label{fig:slow_mixed_x}}
\par\end{centering}
\centering{}\subfloat[]{\includegraphics[width=0.5\columnwidth]{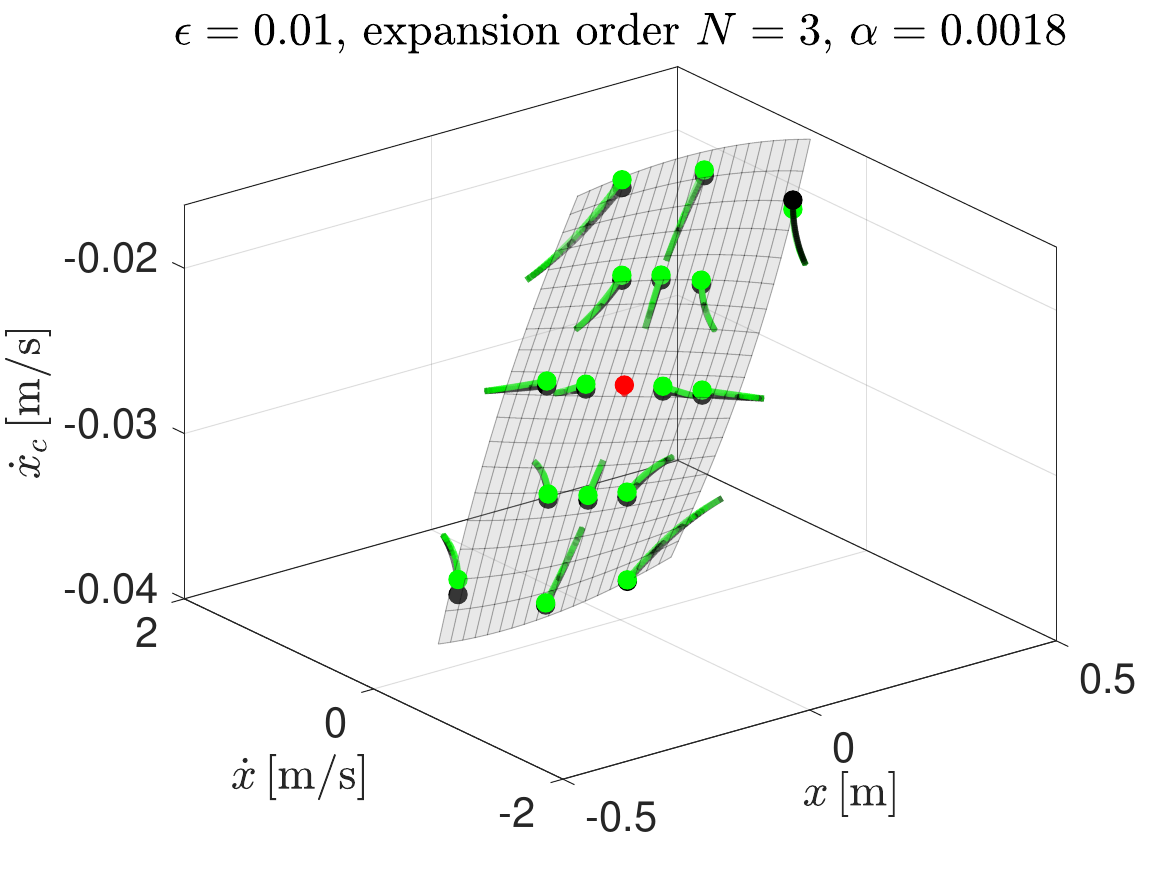}\label{fig:slow_mixed_snap}}\subfloat[]{\includegraphics[width=0.5\columnwidth]{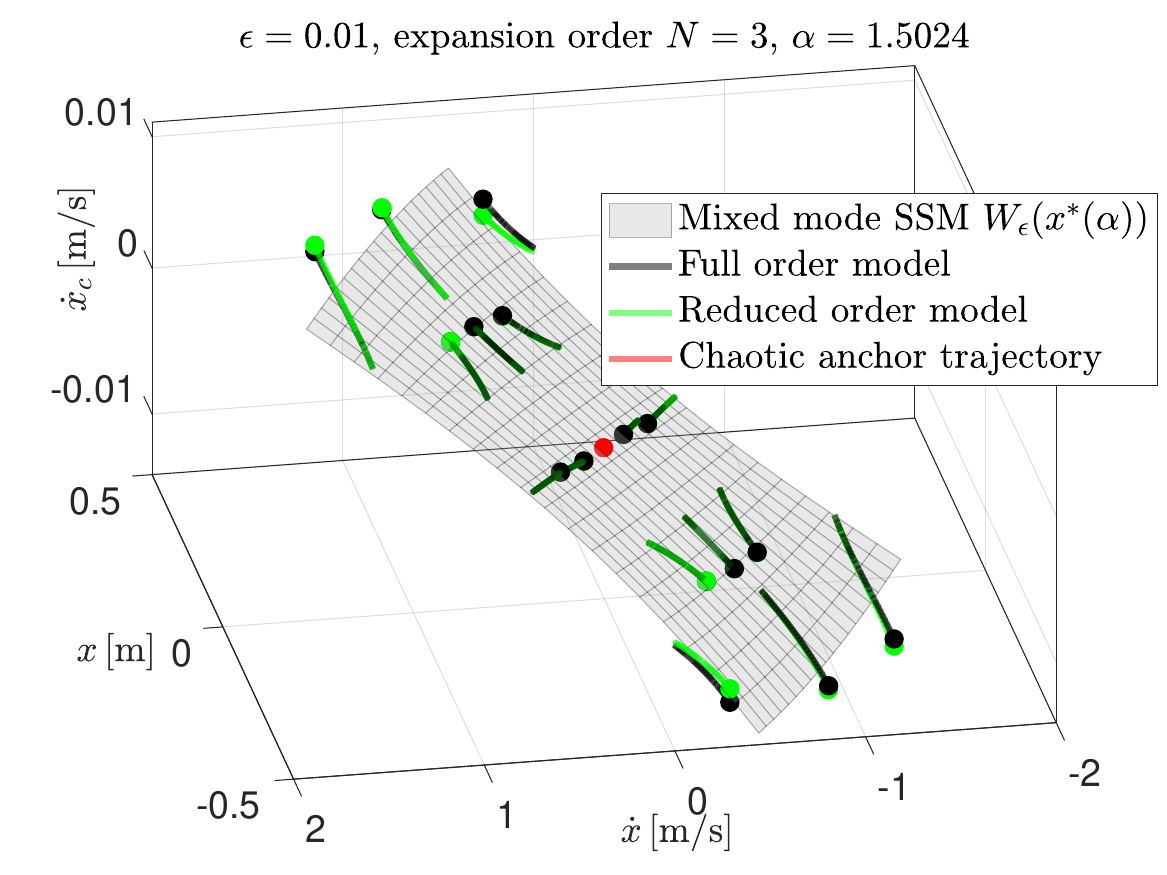}}\caption{(a)\textendash (c) Same as Fig. \ref{fig:snap_mixed_mode_weak}, but
now for moderately slow chaotic forcing with non-dimensional forcing
speed measure $r_{s}=1.22$. See Movie 4 in the Supplementary Material
for the full evolution. (Multimedia available online)}
\end{figure}

\section{Conclusions}

In this paper, we have derived existence results for spectral submanifolds
(SSMs) in non-autonomous dynamical systems that are either weakly
non-autonomous or slowly varying (adiabatic). While our exact proofs
for these SSMs results only cover weak enough or slow enough time
dependence, the invariant manifolds we obtain are structurally stable
and hence are persistent away from these limits. Under further nonresonance
conditions, the manifolds and their reduced dynamics also admit asymptotic
expansions up to any finite order for which we present recursive formulas.
These formulas suggest that one of the non-autonomous SSMs is as smooth
as the original dynamical system, just as primary SSMs are known to
be in the autonomous case reviewed in the Introduction. MATLAB implementations
of our recursive formulas for the examples treated here are available
from \citet{kaundinya23}.

In our results, weakly forced SSMs are guaranteed to exist under uniformly
bounded forcing. This restriction is expected because unbounded perturbations
to an autonomous limit of a dynamical system will generally wipe out
any structure identified in that limit. For adiabatic SSMs, the uniformity
of the forcing magnitude is replaced by the uniformity of the strength
of hyperbolicity along the anchor trajectory of the SSM in the limit
of frozen time. 

In the weakly forced case, the SSMs emanate from unique, uniformly
bounded hyperbolic anchor trajectories. Our expansions for these anchor
trajectories are of independent interest as they provide formal approximations
of the \emph{generalized steady state} response of any nonlinear system
subject to moderate but otherwise arbitrary time-dependent forcing.
In finite-element simulations under aperiodic forcing, such steady
states have been assumed to exist but their computation tends to involve
lengthy simulations of randomly chosen initial conditions until they
have reached a (somewhat vaguely defined) statistical steady state. 

Even with today's computational power, such simulations are still
prohibitively long due to the small damping of typical structural
materials, the high degrees of freedom of the models used, and the
costly evaluations of nonlinearities arising from complex geometries
and multi-physics. With the explicit recursive formulas we obtained
here, one can now directly approximate these generalized steady states
without lengthy simulations.

For these asymptotic formulas to converge, the forcing does not actually
have to be uniformly bounded: it may grow temporally unbounded in
a compact neighborhood of the origin as long as the improper integrals
in the statements of Theorems 2 and 4 converge. To the best of our
knowledge, such explicit, readily computable asymptotic expansions
for time-dependent steady states have been unavailable in the literature
even for time-periodic or time-quasiperiodic forcing, let alone time-aperiodic
forcing. Weak or slow time-periodic and time-quasi-periodic forcing
is also covered by these formulas as special cases. 

Our existence results for SSMs assume temporal smoothness for the
dynamical system and hence do not strictly cover the case of random
forcing with continuous paths. Our approximation formulas, however,
are formally applicable (i.e., the improper integrals in them converge)
for much broader forcing types, including mildly exponentially growing
or temporally discontinuous forcing. Therefore, the results in this
paper provide formal nonlinear model reduction formulas that can be
implemented as approximations for all forcing types arising in practice. 

The simple examples considered here already illustrate the ability
of our asymptotic formulas to work under various forcing types, including
chaotic and discontinuous forcing. Our weakly non-autonomous SSM theory
is expected to be helpful in model reduction for structural vibrations
problems in which so far time-periodic and time-quasiperiodic SSMs
have been constructed. An application of our adiabatic SSM theory
is already underway in the model-predictive control of soft robots,
wherein the target trajectories of the robot generally have a much
slower time scale than the highly damped robot itself.

Going from the simple, illustrative examples treated here to finite-element
problems with temporally aperiodic forcing will require no further
theoretical development. Indeed, the existence and smoothness results
derived in this paper are valid in arbitrarily high (finite) dimensions.
However, a computational reformulation will be required to make these
results directly applicable to second-order mechanical systems without
the need to convert them to first-order systems with diagonalized
linear parts. The development of such a reformulation using the approach
of \citet{jain2022} is currently underway.

\vskip 1 true cm

\textbf{Acknowledgements}

\vskip 0.5 true cm

We are grateful to Christian Pötzsche for very helpful and detailed
explanations and references on the continuity of the dichotomy spectrum
for non-autonomous differential equations. We are also grateful to
Martin Rasmussen and Peter Kloeden for helpful references on non-autonomous
invariant manifolds.

\appendix

\section{Proof of Theorems \ref{thm:x^* under conditions} and \ref{thm: expansion for x^*}}

\setcounter{equation}{0}
\numberwithin{equation}{section}
\renewcommand{\theequation}{\thesection\arabic{equation}} 

\subsection{Proof of Theorem \ref{thm:x^* under conditions}: Existence and regularity
of a hyperbolic anchor trajectory \label{subsec:proof of existence of anchor trajectory}}

We first recall a result of \citet{palmer73}, reformulated from \citet{hartman64},
for a general non-autonomous ODE of the form
\begin{equation}
\dot{x}=A(t)x+f(x,t,p),\label{eq:perturbed system Palmer}
\end{equation}
where $A(t)$ and $f(x,t,p)$ are $C^{0}$ functions of their arguments;
$\left|f(x,t,p)\right|\leq\mu$ is assumed to hold for a constant
$\mu>0$ and for all $x\in\mathbb{R}^{n}$ and $t\in\mathbb{R}$;
$f(\,\cdot\,,t,p)$ admits a global Lipschitz constant $L>0$ for
all times $t$; and $p\in\mathbb{R}^{m}$ is a parameter vector.

Assume that the linearization 
\begin{equation}
\dot{x}=A(t)x\label{eq:homogeneous linear nonautonmous ODE}
\end{equation}
of system (\ref{eq:perturbed system Palmer}) has an \emph{exponential
dichotomy}, i.e., there exists a fundamental matrix solution $\Phi(t)$
of (\ref{eq:homogeneous linear nonautonmous ODE}), constants $K,\kappa>0$
and a projection $P$ (with $P^{2}=P)$ such that 
\begin{align}
\left|\Phi(t)P\Phi^{-1}(s)\right| & \leq Ke^{\kappa\left(t-s\right)},\quad t\geq s,\nonumber \\
\left|\Phi(t)\left(I-P\right)\Phi^{-1}(s)\right| & \leq Ke^{\kappa\left(t-s\right)},\quad s\geq t.\label{eq:dichotomies}
\end{align}
 This condition guarantees that the $x=0$ solution of the linearized
system (\ref{eq:homogeneous linear nonautonmous ODE}) is hyperbolic,
i.e., has well-defined stable and/or unstable subbundles (no neutrally
stable directions). Finally, assume that
\begin{equation}
4LK\leq\kappa.\label{eq:assumption for palmer's result}
\end{equation}

Under these assumptions, there exists a unique, globally bounded solution
$x^{*}(t;p)$ of the nonlinear system (\ref{eq:perturbed system Palmer})
that satisfies\textbf{ 
\begin{equation}
\left|x^{*}(t;p)\right|\leq2K\mu/\kappa,\quad t\in\mathbb{R},\label{eq:bound on x^star}
\end{equation}
}as shown in Lemma 1 of \citet{palmer73}). Furthermore, $x^{*}(t;p)$
is continuous in the parameter $p$. (If $f$ is smooth in its arguments,
then the continuity of $x^{*}(t;p)$ in $p$ can be strengthened to
smooth dependence on $p$ using more general results for the persistence
of non-compact, normally hyperbolic invariant manifolds; see \citet{eldering13}).
Such a strengthened version, however, would not be specific enough
about the norm of $\left|x^{*}(t;p)\right|$ to the extent given in
(\ref{eq:bound on x^star}).) Note that $x^{*}(t;p)$ takes over the
role of $x=0$ as a unique, uniformly bounded solution under nonzero
$f(x,t,p)$. The result of \citet{palmer73} quoted above makes no
assumption on $f(x,t,p)$ containing only nonlinear terms, and hence
$f(x,t,p)$ is allowed to be an arbitrary perturbation to the linear
ODE (\ref{eq:homogeneous linear nonautonmous ODE}), as long as it
is uniformly bounded and uniformly Lipschitz. 

The uniform boundedness conditions (\ref{eq:assumption for palmer's result})-(\ref{eq:bound on x^star})
will not be satisfied in realistic applications. Nevertheless, one
can still use the above results in such applications using smooth
cut-off functions (see, e.g., \citet{fenichel71}) that make $f(x,t,p)$
vanish for all $t\in B_{\delta}$ outside a small ball \textbf{$B_{\delta}\subset\mathbb{R}^{n}$}.
The latter classic tool from invariant manifold theory is, however,
only applicable here if the predicted hyperbolic trajectory $x^{*}(t;p)$
lies inside $B_{\delta}$, where the cut-off right-hand side and the
original right-hand side of (\ref{eq:perturbed system Palmer}) still
coincide. 

To ensure this, we consider a $\delta$-ball $B_{\delta}\subset\mathbb{R}^{n}$
around $x$ and define the uniform bound $\mu(\delta)$ and uniform
Lipschitz constant $L(\delta)$ as
\begin{equation}
\mu(\delta)=\max_{x\in B_{\delta}}\left|f(x,t,p)\right|,\qquad\left|f(x,t,p)-f(\tilde{x},t,p)\right|\leq L(\delta)\left|x-\tilde{x}\right|,\quad x,\tilde{x}\in B_{\delta}.\label{eq:ass1}
\end{equation}
Assume further that these bounds satisfy
\begin{equation}
\frac{2K\mu(\delta)}{\kappa}\leq\delta,\quad L(\delta)\leq\frac{\kappa}{4K},\label{eq:ass 2}
\end{equation}
which assures that assumption (\ref{eq:assumption for palmer's result})
holds and also that the predicted $x^{*}(t;p)$ falls in $B_{\delta}$
by the estimate (\ref{eq:bound on x^star}). Note that the inequalities
in (\ref{eq:ass 2}) will always hold for $\delta>0$ small enough
if $f(x,t,p)$ is only composed of terms that are nonlinear in $x$.
Indeed, in that case, we can select $\mu(\delta)=\mathcal{O}\left(\delta^{2}\right)$
and $L(\delta)=\mathcal{O}\left(\delta\right)$ for all $x\in B_{\delta}$,
and hence the inequalities in (\ref{eq:ass 2}) are satisfied for
small enough $\delta$, given that $K,\kappa>0$ do not depend on
$\delta$. 

We now apply the above refined persistence result in the setting of
the perturbed nonlinear system (\ref{eq: nonlinear system}) by letting
\begin{equation}
A(t)=A,\qquad f(x,t,p)=f_{0}(x)+f_{1}(x,t),\label{eq:assumption on A(t) and f}
\end{equation}
and assume a separate uniform Lipschitz constant $L_{1}(\delta)$
for $f_{1}(x,t)$ satisfying 
\[
\left|f_{1}(x,t)-f_{1}(\tilde{x},t)\right|\leq L_{1}(\delta)\left|x-\tilde{x}\right|,\quad x,\tilde{x}\in B_{\delta}.
\]
By our hyperbolicity assumption on the $x=0$ fixed point, we can
select the constant $\kappa>0$ so that
\begin{equation}
\kappa<\min_{1\leq j\leq n}\left|\mathrm{Re}\lambda_{j}\right|.\label{eq:alpha_choice}
\end{equation}
We can also bound within $B_{\delta}$ the nonlinear term $f_{0}$
and its spatial derivative as
\begin{equation}
\left|f_{0}(x,p)\right|\leq K_{1}\delta^{2},\qquad\left|\partial_{x}f_{0}(x)\right|\leq K_{2}\delta,\qquad x\in B_{\delta},\label{eq:constants}
\end{equation}
where $K_{1},K_{2}>0$ are appropriate constants. Additionally, we
can specifically choose $K_{2}=\max_{x\in B_{\delta}}\left|\partial_{x}f_{0}(x)\right|$
by the mean value theorem. Given these bounds, the constants in formulas
(\ref{eq:ass1}) can be estimated as
\begin{equation}
\mu(\delta)\leq K_{1}\delta^{2}+\left|f_{1}(x,t)\right|,\quad L(\delta)\leq L_{1}(\delta)+K_{2}\delta,\quad x\in B_{\delta}.\label{eq:smallness of w}
\end{equation}

Based on the inequalities (\ref{eq:smallness of w}), we see that
assumptions (\ref{eq:ass 2}) are satisfied if
\[
\frac{2K\left(K_{1}\delta^{2}+\left|f_{1}(x,t)\right|\right)}{\kappa}\leq\delta,\quad L_{1}(\delta)+K_{2}\delta\leq\frac{\kappa}{4K},\qquad x\in B_{\delta},\quad t\in\mathbb{R},
\]
or, equivalently, if
\begin{equation}
\left|f_{1}(x,t)\right|\leq\frac{\kappa\delta}{2K}-K_{1}\delta^{2},\qquad L_{1}(\delta)\leq\frac{\kappa}{4K}-K_{2}\delta,\qquad x\in B_{\delta},\quad t\in\mathbb{R}.\label{eq:bounds on perturbation}
\end{equation}
We stress again that these conditions do not have to hold globally
for all $x$, only for $x\in B_{\delta}$, but they must hold uniformly
within $B_{\delta}$ for all times. This follows because we can apply
the $C^{\infty}$ cutoff procedure mentioned above to the whole right-hand
side of (\ref{eq:perturbed system Palmer}), including $f_{1}(x,t)$.
Based on these considerations and using the inequalities in eq. (\ref{eq:constants}),
we obtain statement (i) of Theorem \ref{thm:x^* under conditions}. 

Statement (ii) of Theorem \ref{thm:x^* under conditions} follows
from the smooth dependence of $x^{*}(t;p)$ on the parameter vector
$p$, as shown by \citet{palmer73}.

\subsection{Proof of Theorem \ref{thm: expansion for x^*}: Approximation of
the anchor trajectory $x^{*}(t)$ \label{subsec:proof of approximation of x^*}}

First, we introduce a perturbation parameter $\epsilon\geq0$ and
rewrite the full non-autonomous system (\ref{eq: nonlinear system})
as

\begin{equation}
\dot{x}=Ax+f(x,t;\epsilon),\quad f(x,t;\epsilon)=f_{0}(x)+\epsilon\tilde{f}_{1}(x,t).\label{eq:original system 1}
\end{equation}
By the smooth dependence of the uniformly bounded hyperbolic solution
$x^{*}(t)$ on parameters, we can seek this solution in the form of
a Taylor expansion
\begin{equation}
x^{*}(t)=\sum_{\nu\geq1}\epsilon^{\nu}\xi_{\nu}(t),\label{eq:formal expansion for anchor point}
\end{equation}
with Taylor coefficients $\xi_{\nu}(t)$ that are uniformly bounded
in time. By statement (ii) of Theorem \eqref{thm:x^* under conditions},
such a formal Taylor expansion in $\epsilon$ is justified up to any
finite order, although may not necessarily converge.

Substitution of the expansion (\ref{eq:formal expansion for anchor point})
into eq. (\ref{eq:original system 1}) gives 
\begin{align}
\sum_{j\geq1}\epsilon^{j}\dot{\xi}_{j}(t) & =\sum_{j\geq1}\epsilon^{j}A\xi_{j}(t)+f\left(x_{\epsilon}(t),t;\epsilon\right).\label{eq:eq:epsilon expansion}
\end{align}
 Equating equal powers of $\epsilon$ on both sides yields the system
of differential equations
\begin{equation}
\dot{\xi}_{\nu}(t)=A\xi_{\nu}(t)+\frac{1}{\nu!}\left.D_{\epsilon}^{\nu}f\left(x_{\epsilon}(t),t;\epsilon\right)\right|_{\epsilon=0}.\label{eq:xi_j dot1}
\end{equation}
Note that the term formally containing $\xi_{\nu}(t)$ in $\left.D_{\epsilon}^{\nu}f\left(x_{\epsilon}(t),t;\epsilon\right)\right|_{\epsilon=0}$
is
\[
\nu!\left[D_{x}f_{0}(x_{\epsilon}(t))+\epsilon\tilde{f}_{1}\left(x_{\epsilon}(t),t\right)\right]_{\epsilon=0}\xi_{\nu}(t)\equiv0,
\]
because $f_{0}(x)=\mathcal{O}\left(\left|x\right|^{2}\right)$ and
$x_{0}(t)\equiv0$. Therefore, the right-hand side of eq. (\ref{eq:xi_j dot1})
only depends on $\xi_{1}(t),\ldots,\xi_{\nu-1}(t)$, which makes the
whole system of equations a recursively solvable sequence of inhomogeneous,
linear, constant-coefficient system of ODEs. The recursive solutions
are of the form
\begin{align}
\xi_{\nu}(t) & =e^{A\left(t-t_{0}\right)}\xi_{\nu}(t_{0})+\frac{1}{\nu!}\int_{t_{0}}^{t}e^{A\left(t-\tau\right)}\left.D_{\epsilon}^{\nu}f\left(\sum_{j=1}^{\nu-1}\epsilon^{j}\xi_{j}(\tau),\tau;\epsilon\right)\right|_{\epsilon=0}d\tau,\quad\nu\geq1.\label{eq:x_12-2}
\end{align}
 Assume first, for simplicity, that the $x=0$ fixed point is asymptotically
stable for $\epsilon=0$, and hence 
\begin{equation}
\mathrm{Re}\left[\mathrm{spect}\left(A\right)\right]\subset\left(-\infty,0\right).\label{eq: stable fixed point assumption}
\end{equation}
 In that case, by the uniform boundedness of $\xi_{1}(t_{0})$ for
all $t_{0}\in\mathbb{R}$, we can take the $t_{0}\to-\infty$ limit
in (\ref{eq:x_12-2}) to obtain
\begin{align}
\xi_{\nu}(t) & =\frac{1}{\nu!}\int_{-\infty}^{t}e^{A\left(t-\tau\right)}\left.D_{\epsilon}^{\nu}f\left(\sum_{j=1}^{\nu-1}\epsilon^{j}\xi_{j}(\tau),\tau;\epsilon\right)\right|_{\epsilon=0}d\tau,\quad\nu\geq1.\label{eq:x_12-2-1}
\end{align}

To obtain a more explicit recursive formula for $\xi_{\nu}(t)$ for
general $\nu$ that is suitable for direct numerical implementation,
we will use the multi-variate Faá di Bruno formula of \citet{constantine96}
for higher-order derivatives of composite functions.\emph{ }To recall
the general form of this formula, we first consider a general composite
function $H\colon\mathbb{R}^{p}\to\mathbb{R}$, defined as 
\begin{equation}
H(x_{1},\ldots,x_{p})=f\left(g^{1}(x_{1},\ldots,x_{p}),\ldots,g^{m}(x_{1},\ldots,x_{p})\right),\label{eq:Hdef}
\end{equation}
and introduce the nonnegative multi-index $\mathbf{\boldsymbol{\nu}}$
and related notation as
\[
\mathbf{\boldsymbol{\nu}}=\left(\nu_{1},\ldots,\nu_{p}\right)\in\mathbb{N}^{p},\quad\left|\mathbf{\boldsymbol{\nu}}\right|=\sum_{i=1}^{p}\nu_{i},\quad\mathbf{\boldsymbol{\nu}}!=\prod_{i=1}^{p}\left(\nu_{i}!\right),\quad D_{\mathbf{x}}^{\mathbf{\boldsymbol{\nu}}}=\frac{\partial^{\left|\mathbf{\boldsymbol{\nu}}\right|}}{\partial x_{1}^{\nu_{1}}\cdots\partial x_{p}^{\nu_{p}}}.
\]
 We also introduce an ordering relation on $\mathbb{N}^{p}$ for arbitrary
$\mathbf{\boldsymbol{\mu}},\mathbf{\boldsymbol{\nu}}\in\mathbb{N}^{p}$
such that $\mathbf{\boldsymbol{\nu}}\prec\mathbf{\boldsymbol{\mu}}$
holds provided one of the following is satisfied: 

(i) $\left|\mathbf{\boldsymbol{\mu}}\right|<\left|\mathbf{\boldsymbol{\nu}}\right|$

(ii) $\left|\mathbf{\boldsymbol{\mu}}\right|=\left|\mathbf{\boldsymbol{\nu}}\right|$
and $\mu_{1}<\nu_{1}$ or

(iii) $\left|\mathbf{\boldsymbol{\mu}}\right|=\left|\mathbf{\boldsymbol{\nu}}\right|$
and $\mu_{1}=\nu_{1},\ldots\mu_{k}=\nu_{k},\mu_{k+1}<\nu_{k+1}$ for
some $1\leq k<p.$ 

Finally, for any $\mathbf{\boldsymbol{\mu}},\boldsymbol{\gamma}\in\mathbb{N}^{p}$
and $s\in\mathbb{N}^{+}$, we define the index set 
\[
p_{s}\left(\mathbf{\boldsymbol{\nu}},\boldsymbol{\gamma}\right)=\left\{ \left(\mathbf{k}_{1},\ldots,\mathbf{k}_{s},\boldsymbol{\ell}_{1},\ldots,\boldsymbol{\ell}_{s}\right):\mathbf{k}_{i}\in\mathbb{N}^{m}-\left\{ \mathbf{0}\right\} ,\boldsymbol{\ell}_{i}\in\mathbb{N}^{p},\mathbf{0}\prec\boldsymbol{\ell}_{1}\prec\cdots\prec\boldsymbol{\ell}_{s},\sum_{i=1}^{s}\mathbf{k}_{i}=\boldsymbol{\gamma},\sum_{i=1}^{s}\left|\mathbf{k}_{i}\right|\boldsymbol{\ell}_{i}=\mathbf{\boldsymbol{\nu}}\right\} .
\]
 With this notation, \citet{constantine96} prove the following multi-variate
version of the Faá di Bruno formula:

\begin{equation}
D_{\mathbf{x}}^{\mathbf{\boldsymbol{\nu}}}H\left(\mathbf{x}^{0}\right)=\sum_{1\leq\left|\boldsymbol{\gamma}\right|\leq\left|\mathbf{\boldsymbol{\nu}}\right|}D_{\mathbf{y}}^{\mathbf{\boldsymbol{\gamma}}}f\left(\mathbf{y}^{0}\right)\sum_{s=1}^{\left|\mathbf{\boldsymbol{\nu}}\right|}\sum_{p_{s}\left(\mathbf{\boldsymbol{\nu}},\mathbf{\boldsymbol{\gamma}}\right)}\mathbf{\boldsymbol{\nu}}!\prod_{j=1}^{s}\frac{\prod_{i=1}^{m}\left[D_{\mathbf{x}}^{\mathbf{\boldsymbol{\ell}}_{j}}g^{i}(\mathbf{x}^{0})\right]^{k_{ji}}}{\left(\mathbf{k}_{j}\right)!\left[\left(\mathbf{\boldsymbol{\ell}}_{j}\right)!\right]^{\left|\mathbf{k}_{j}\right|}},\label{eq:general Faa di Bruno formula}
\end{equation}
where $k_{ji}$ denoted the $i^{th}$ element of the multi-index $\mathbf{k}_{j}\in\mathbb{N}^{m}-\left\{ \mathbf{0}\right\} $. 

Of relevance to us is the case $p=1$, wherein we have $H=f_{q}$,
$m=n,$ $g^{i}=\sum_{j\geq1}\epsilon^{j}\xi_{j}^{i}(\tau)$, $i=1,\ldots,n$;
$\mathbf{x}=\epsilon\in\mathbb{R},$ $\mathbf{x}^{0}=0\in\mathbb{R}$,
$\mathbf{\boldsymbol{\nu}}=\nu\in\mathbb{N}$, $\boldsymbol{\ell}_{i}=\ell_{i}\in\mathbb{N}$,
and $\mathbf{y}=x\in\mathbb{R}^{n}$. In that case, we can write
\begin{equation}
H(\epsilon)=f_{q}\left(\sum_{j=1}^{\nu-1}\epsilon^{j}\xi_{j}^{1}(\tau),\ldots,\sum_{j=1}^{\nu-1}\epsilon^{j}\xi_{j}^{n}(\tau),\tau;\epsilon\right),\quad q=1,\ldots,n.\label{eq:H(epsilon)}
\end{equation}
 Note, however, that $H(\epsilon)$ also has explicit dependence on
$\epsilon$ and hence is not exactly of the form (\ref{eq:Hdef}).
To address this issue, we also observe that
\begin{align*}
\frac{d^{\nu}}{d\epsilon^{\nu}}H\left(0\right) & =\frac{d^{\nu}}{d\epsilon^{\nu}}H_{0}\left(0\right)+\nu\frac{d^{\nu-1}}{d\epsilon^{\nu-1}}H_{1}\left(0\right),\\
H_{0}\left(\epsilon\right) & =f_{0q}\left(\sum_{j=1}^{\nu-1}\epsilon^{j}\xi_{j}^{1}(\tau),\ldots,\sum_{j=1}^{\nu-1}\epsilon^{j}\xi_{j}^{n}(\tau)\right),\\
H_{1}\left(\epsilon\right) & =\tilde{f}_{1q}\left(\sum_{j=1}^{\nu-1}\epsilon^{j}\xi_{j}^{1}(\tau),\ldots,\sum_{j=1}^{\nu-1}\epsilon^{j}\xi_{j}^{n}(\tau),\tau\right),
\end{align*}
and hence $H_{0}\left(\epsilon\right)$ and $H_{1}\left(\epsilon\right)$
are individually of the form (\ref{eq:Hdef}). Applied to these two
functions, the formula (\ref{eq:general Faa di Bruno formula}) simplifies
to
\begin{align*}
\frac{d^{\nu}}{d\epsilon^{\nu}}H_{0}\left(0\right) & =\left.D_{\epsilon}^{\nu}f_{0q}\left(\sum_{j=1}^{\nu-1}\epsilon^{j}\xi_{j}^{1}(\tau),\ldots,\sum_{j=1}^{\nu-1}\epsilon^{j}\xi_{j}^{n}(\tau)\right)\right|_{\epsilon=0}\\
 & =\sum_{1\leq\left|\boldsymbol{\gamma}\right|\leq\nu}D_{x}^{\boldsymbol{\gamma}}f_{q}\left(0\right)\sum_{s=1}^{\nu}\sum_{p_{s}\left(\nu,\boldsymbol{\gamma}\right)}\nu!\prod_{j=1}^{s}\frac{\prod_{i=1}^{n}\left[\left(\ell_{j}\right)!\xi_{\ell_{j}}^{i}(\tau)\right]^{k_{ji}}}{\left(\mathbf{k}_{j}\right)!\left[\left(\ell_{j}\right)!\right]^{\left|\mathbf{k}_{j}\right|}}\\
 & =\sum_{1\leq\left|\boldsymbol{\gamma}\right|\leq\nu}D_{x}^{\boldsymbol{\gamma}}f_{0q}\left(0\right)\sum_{s=1}^{\nu}\sum_{p_{s}\left(\nu,\boldsymbol{\gamma}\right)}\nu!\prod_{j=1}^{s}\frac{\prod_{i=1}^{n}\left[\xi_{\ell_{j}}^{i}(\tau)\right]^{k_{ji}}}{\prod_{i=1}^{n}k_{ji}!},\quad q=1,\ldots,n,\\
\frac{d^{\nu-1}}{d\epsilon^{\nu-1}}H_{1}\left(0\right) & =\left.D_{\epsilon}^{\nu-1}\tilde{f}_{1q}\left(\sum_{j=1}^{\nu-1}\epsilon^{j}\xi_{j}^{1}(\tau),\ldots,\sum_{j=1}^{\nu-1}\epsilon^{j}\xi_{j}^{n}(\tau),\tau\right)\right|_{\epsilon=0}\\
 & =\sum_{1\leq\left|\boldsymbol{\gamma}\right|\leq\nu-1}D_{x}^{\mathbf{\boldsymbol{\gamma}}}\tilde{f}_{1q}\left(0,\tau\right)\sum_{s=1}^{\nu-1}\sum_{p_{s}\left(\nu-1,\boldsymbol{\gamma}\right)}\left(\nu-1\right)!\prod_{j=1}^{s}\frac{\prod_{i=1}^{n}\left[\left(\ell_{j}\right)!\xi_{\ell_{j}}^{i}(\tau)\right]^{k_{ji}}}{\left(\mathbf{k}_{j}\right)!\left[\left(\ell_{j}\right)!\right]^{\left|\mathbf{k}_{j}\right|}}.
\end{align*}

Substitution of these expressions into formula (\ref{eq:x_12-2-1})
then gives the final recursive formulas
\begin{align}
\xi_{\nu}(t) & =\sum_{1\leq\left|\boldsymbol{\gamma}\right|\leq\nu}\,\sum_{s=1}^{\nu}\sum_{p_{s}\left(\nu,\boldsymbol{\gamma}\right)}\int_{-\infty}^{t}e^{A\left(t-\tau\right)}\left[\frac{\partial^{\left|\boldsymbol{\gamma}\right|}f_{0}\left(0\right)}{\partial x_{1}^{\gamma_{1}}\cdots\partial x_{n}^{\gamma_{n}}}\prod_{j=1}^{s}\frac{\prod_{i=1}^{n}\left[\xi_{\ell_{j}}^{i}(\tau)\right]^{k_{ji}}}{\prod_{i=1}^{n}k_{ji}!}\right]d\tau\label{eq:x_12-2-1-1}\\
 & +\sum_{1\leq\left|\boldsymbol{\gamma}\right|\leq\nu-1}\,\sum_{s=1}^{\nu-1}\sum_{p_{s}\left(\nu-1,\boldsymbol{\gamma}\right)}\int_{-\infty}^{t}e^{A\left(t-\tau\right)}\left[\frac{\partial^{\left|\boldsymbol{\gamma}\right|}\tilde{f}_{1}\left(0,\tau\right)}{\partial x_{1}^{\gamma_{1}}\cdots\partial x_{n}^{\gamma_{n}}}\prod_{j=1}^{s}\frac{\prod_{i=1}^{n}\left[\xi_{\ell_{j}}^{i}(\tau)\right]^{k_{ji}}}{\prod_{i=1}^{n}k_{ji}!}\right]d\tau,\quad\nu\geq1,\nonumber 
\end{align}
where $k_{ji}$ is the $i^{th}$ component of the integer vector \textbf{$\mathbf{k}_{j}\in\mathbb{N}^{n}-\left\{ \mathbf{0}\right\} $
}appearing\textbf{ }in the index set
\[
p_{s}\left(\nu,\boldsymbol{\gamma}\right)=\left\{ \left(\mathbf{k}_{1},\ldots,\mathbf{k}_{s},\ell_{1},\ldots,\ell_{s}\right):\mathbf{k}_{i}\in\mathbb{N}^{n}-\left\{ \mathbf{0}\right\} ,\ell_{i}\in\mathbb{N},0<\ell_{1}<\cdots<\ell_{s},\sum_{i=1}^{s}\mathbf{k}_{i}=\boldsymbol{\gamma},\sum_{i=1}^{s}\left|\mathbf{k}_{i}\right|\ell_{i}=\nu\right\} .
\]
 We have also used the notational convection $\frac{\partial^{\left|\boldsymbol{\gamma}\right|}}{\partial x_{1}^{\gamma_{1}}\cdots\partial x_{n}^{\gamma_{n}}}=I$
for $\boldsymbol{\gamma}=0$. 

Substitution of (\ref{eq:x_12-2-1-1}) into the expansion (\ref{eq:formal expansion for anchor point})
gives
\begin{align}
x^{*}(t) & =\sum_{\nu\geq1}x_{\nu}(t),\label{eq:x_nu stable case}\\
x_{\nu}(t)= & \sum_{1\leq\left|\boldsymbol{\gamma}\right|\leq\nu}\,\sum_{s=1}^{\nu}\sum_{p_{s}\left(\nu,\boldsymbol{\gamma}\right)}\int_{-\infty}^{t}e^{A\left(t-\tau\right)}\left[\frac{\partial^{\left|\boldsymbol{\gamma}\right|}f_{0}\left(0\right)}{\partial x_{1}^{\gamma_{1}}\cdots\partial x_{n}^{\gamma_{n}}}\frac{\prod_{i=1}^{n}\left[x_{\ell_{j}}^{i}(\tau)\right]^{k_{ji}}}{\prod_{i=1}^{n}k_{ji}!}\right]d\tau\nonumber \\
 & +\sum_{1\leq\left|\boldsymbol{\gamma}\right|\leq\nu-1}\,\sum_{s=1}^{\nu-1}\sum_{p_{s}\left(\nu-1,\boldsymbol{\gamma}\right)}\int_{-\infty}^{t}e^{A\left(t-\tau\right)}\left[\frac{\partial^{\left|\boldsymbol{\gamma}\right|}f_{1}\left(0,\tau\right)}{\partial x_{1}^{\gamma_{1}}\cdots\partial x_{n}^{\gamma_{n}}}\prod_{j=1}^{s}\frac{\prod_{i=1}^{n}\left[x_{\ell_{j}}^{i}(\tau)\right]^{k_{ji}}}{\prod_{i=1}^{n}k_{ji}!}\right]d\tau,\quad\nu\geq1.\nonumber 
\end{align}
These results cover anchor trajectories of like-mode SSMs of stable
hyperbolic fixed points but do not cover anchor trajectories for mixed-mode
SSMs of such fixed points or anchor trajectories of like-mode SSMs
of unstable hyperbolic fixed points.

We now extend formula (\ref{eq:x_12-2-1-1}) to mixed-mode SSMs by
weakening the stability assumption (\ref{eq: stable fixed point assumption})
back to our general hyperbolicity assumption (\ref{eq:hyperbolicity assumption}).
In this case, the matrix $A$ has an exponential dichotomy, as described
by the inequalities (\ref{eq:dichotomy for autonomous ODE}). The
dichotomy exponent $\kappa>0$ can be selected as in (\ref{eq:alpha_choice}).
We use the matrix $T$ defined in \eqref{eq:T definition} in the
coordinate transformation
\[
x=Ty,\quad y=\left(y^{s},y^{u}\right)^{\mathrm{T}}\in\mathbb{R}^{s}\times\mathbb{R}^{u},
\]
which brings system system (\ref{eq: nonlinear system}) to the form
\begin{align*}
\dot{y}^{s} & =A^{s}y^{s}+f^{s}(y,t),\\
\dot{y}^{u} & =A^{u}y^{u}+f^{u}(y,t),
\end{align*}
 with $\left(f^{s}(y,t),f^{u}(y,t)\right)^{\mathrm{T}}=T^{-1}f(Ty,t)$
and
\begin{equation}
\mathrm{Re}\left[\mathrm{spect}\left(A^{s}\right)\right]\subset\left(-\infty,0\right),\quad\mathrm{Re}\left[\mathrm{spect}\left(A^{u}\right)\right]\subset\left(0,\infty\right).\label{eq:two spectra}
\end{equation}
 In these new coordinates, with the notation 
\[
\left(\begin{array}{c}
\hat{\xi}_{\nu}^{s}(t)\\
\hat{\xi}_{\nu}^{u}(t)
\end{array}\right)=T^{-1}\xi_{\nu}(t),
\]
the system of ODEs (\ref{eq:xi_j dot1}) becomes 
\begin{align}
\dot{\hat{\xi}}_{\nu}^{s}(t) & =A^{s}\hat{\xi}_{\nu}^{s}(t)+\frac{1}{\nu!}\left.D_{\epsilon}^{\nu}f^{s}\left(T^{-1}x_{\epsilon}(t),t\right)\right|_{\epsilon=0},\nonumber \\
\dot{\hat{\xi}}_{\nu}^{u}(t) & =A^{u}\hat{\xi}_{\nu}^{u}(t)+\frac{1}{\nu!}\left.D_{\epsilon}^{\nu}f^{u}\left(T^{-1}x_{\epsilon}(t),t\right)\right|_{\epsilon=0},\label{eq:xi_uj dot}
\end{align}
whose solutions can be written as 
\begin{align}
\hat{\xi}_{\nu}^{s}(t) & =e^{A^{s}\left(t-t_{0}^{s}\right)}\hat{\xi}_{\nu}^{s}(t_{0}^{s})+\frac{1}{\nu!}\int_{t_{0}^{s}}^{t}e^{A^{s}\left(t-\tau\right)}\left.D_{\epsilon}^{\nu}f^{s}\left(T^{-1}x_{\epsilon}(\tau),\tau\right)\right|_{\epsilon=0}d\tau,\nonumber \\
\hat{\xi}_{\nu}^{u}(t) & =e^{A^{u}\left(t-t_{0}^{u}\right)}\hat{\xi}_{\nu}^{u}(t_{0}^{u})+\frac{1}{\nu!}\int_{t_{0}^{u}}^{t}e^{A^{u}\left(t-\tau\right)}\left.D_{\epsilon}^{\nu}f^{u}\left(T^{-1}x_{\epsilon}(\tau),\tau\right)\right|_{\epsilon=0}d\tau.\label{eq:xi_uj dot-1}
\end{align}
 Based on the spectral properties of $A^{s}$ and $A^{u}$ listed
in (\ref{eq:two spectra}) and the uniform boundedness of $\hat{\xi}_{\nu}^{s,u}(t_{0}^{s,u})$,
we can take the limits $t_{0}^{s}\to-\infty$ and $t_{0}^{u}\to+\infty$
to obtain 
\begin{align}
\hat{\xi}_{\nu}^{s}(t) & =\frac{1}{\nu!}\int_{-\infty}^{t}e^{A^{s}\left(t-\tau\right)}\left.D_{\epsilon}^{\nu}f^{s}\left(T^{-1}x_{\epsilon}(\tau),\tau\right)\right|_{\epsilon=0}d\tau,\nonumber \\
\hat{\xi}_{\nu}^{u}(t) & =-\frac{1}{\nu!}\int_{t}^{\infty}e^{A^{u}\left(t-\tau\right)}\left.D_{\epsilon}^{\nu}f^{u}\left(T^{-1}x_{\epsilon}(\tau),\tau\right)\right|_{\epsilon=0}d\tau.\label{eq:xi_uj dot-1-1}
\end{align}
Therefore, introducing the Green's function (\ref{eq:Green's function definition}),
we can write the solution $\xi_{\nu}(t)$ in the $y$ coordinates
as 
\[
\xi_{\nu}(t)=\frac{1}{\nu!}\int_{-\infty}^{\infty}G(t-\tau)\left.D_{\epsilon}^{\nu}f\left(x_{\epsilon}(\tau),\tau\right)\right|_{\epsilon=0}d\tau,\quad\nu\geq1.
\]
 This then leads to the final formula (\ref{eq:x_nu hyperbolic case})
for a general hyperbolic trajectory $x^{*}(t)$, in analogy with formula
(\ref{eq:x_nu stable case}) for an asymptotically stable hyperbolic
trajectory.

\section{Proofs of Theorems \ref{thm:non-automous SSM} and \ref{thm:computation of non-automous SSM}}

\setcounter{equation}{0}

\subsection{Proof of Theorem \ref{thm:non-automous SSM}: Existence of non-autonomous
spectral submanifolds \label{subsec:Proof of existence of SSM}}

Under the conditions of Theorem \ref{thm:x^* under conditions}, we
can shift coordinates by letting
\[
y=x-x^{*}(t),
\]
which transforms system (\ref{eq: nonlinear system}) to the form
\begin{align}
\dot{y} & =\left[A+\partial_{x}f(x^{*}(t),t)\right]y+g(y,t),\label{eq:perturbed, transformed}
\end{align}
 where
\begin{align}
g(y,t) & =f(x^{*}(t)+y)-f(x^{*}(t))-\partial_{x}f(x^{*}(t),t)y=\mathcal{O}\left(\left|y\right|^{2}\right).\label{eq:def of g(y,t)}
\end{align}

Next we introduce a perturbation parameter $0\leq\epsilon<\delta$
and the scalings
\begin{equation}
x=\epsilon\xi,\quad y=\epsilon\eta=\epsilon\xi-\epsilon\xi^{*}(t),\quad f_{0}(x)=\epsilon^{2}\tilde{f}_{0}\left(\xi;\epsilon\right),\quad f_{1}(x,t)=\epsilon\tilde{f}_{1}(x,t).\label{eq:epsilon scalings}
\end{equation}
These scalings imply
\begin{align*}
f_{0}(x^{*}(t)+y) & =\epsilon^{2}\tilde{f}_{0}(\xi^{*}(t)+\eta;\epsilon),\quad D_{x}f_{0}(x^{*}(t))=\epsilon D_{\xi}\tilde{f}_{0}(\xi^{*}(t);\epsilon),\\
\partial_{x}f_{1}(x^{*}(t),t) & =\epsilon\partial_{x}\tilde{f}_{1}(x^{*}(t),t),\quad g(y,t)=\epsilon^{2}\tilde{g}(\eta,t;\epsilon).
\end{align*}
With these expressions, we can rewrite eq. (\ref{eq:perturbed, transformed})
as
\begin{align*}
\dot{\eta} & =\left[A+\epsilon D_{\xi}\tilde{f}_{0}(\xi^{*}(t);\epsilon)+\epsilon\partial_{x}\tilde{f}_{1}(\epsilon\xi^{*}(t),t)\right]\eta+\epsilon\tilde{g}(\eta,t;\epsilon),\quad\tilde{g}(\eta,t)=\mathcal{O}\left(\left|\eta\right|^{2}\right),\\
\dot{\epsilon} & =0.
\end{align*}
 We can further rewrite these equations as

\begin{align}
\dot{Y} & =\mathcal{A}(t)Y+\mathcal{G}(Y,t),\qquad\mathcal{G}(Y,t)=\mathcal{O}\left(\left|Y\right|^{2}\right),\qquad Y=\left(\begin{array}{c}
\eta\\
\epsilon
\end{array}\right)\in\mathbb{R}^{n+1},\label{eq:dot Y eq def}\\
\mathcal{A}(t) & =\left(\begin{array}{cc}
A & 0\\
0 & 0
\end{array}\right),\qquad\mathcal{G}(Y,t)=\left(\begin{array}{c}
\epsilon\partial_{x}\tilde{f}_{0}(\xi^{*}(t);\epsilon)\eta+\epsilon\partial_{x}\tilde{f}_{1}(\epsilon\xi^{*}(t),t)\eta+\epsilon\tilde{g}(\eta,t)\\
0
\end{array}\right).\label{eq:A G def}
\end{align}

While linearization results do not apply to eq. (\ref{eq:dot Y eq def})
due to the non-hyperbolicity of the origin, the invariant manifold
results of \citet{kloeden11} do apply. Their main condition stated
in our context is
\begin{equation}
\mathcal{G}(0,t)=0,\qquad\lim_{Y\to0}\sup_{t\in\mathbb{R}}\left|D_{Y}\mathcal{G}(Y,t)\right|=0,\label{eq:lim sup condition}
\end{equation}
of which the first one is already satisfied and the second one requires
the local uniform boundedness of the first derivatives of $\mathcal{G}(Y,t)$
in a neighborhood of $Y=0$. By inspection of formula (\ref{eq:A G def}),
we see that the second condition in (\ref{eq:lim sup condition})
is satisfied if the following four conditions are fulfilled:

(a) $\left|\partial_{x}f_{0}(x^{*}(t))\right|$ is uniformly bounded

(b) $\left|\partial_{x}f_{1}(x,t)\right|$ and $\left|\partial_{x}^{2}f_{1}(x,t)\right|$
are uniformly bounded in the $B_{\delta}$ neighborhood of $x=0$ 

(c) $\left|g(y,t)\right|$ is uniformly bounded in a neighborhood
of $y=0$ and 

(d) $\partial_{y}\left|g(y,t)\right|$ is uniformly bounded in a neighborhood
of $y=0$. 

Given the definition of $g(y,t)$ in (\ref{eq:def of g(y,t)}), conditions
(c)-(d) are satisfied if $\left|\partial_{x}f(x^{*}(t),t)\right|$
and
\begin{equation}
\left|\partial_{x}f(x^{*}(t)+y,t)-\partial_{x}f(x^{*}(t),t)\right|\label{eq:cond on f}
\end{equation}
are uniformly bounded in $B_{\delta}$. 

First, note that condition (a) is satisfied because $x^{*}(t)$ is
uniformly bounded. Second, since $x^{*}(t)$ stays in the $B_{\delta}$
ball under the conditions (\ref{eq:conditions for x^*}), we obtain
that $\left|\partial_{x}f_{0}(x^{*}(t)+y)-\partial_{x}f_{0}(x^{*}(t))\right|$
in uniformly bounded in a compact neighborhood of $y=0.$ Therefore,
for condition (\ref{eq:cond on f}) to hold (and hence to ensure that
(c)-(d) to hold), it remains to require that $\left|\partial_{x}f_{1}(x^{*}(t)+y)-\partial_{x}f_{1}(x^{*}(t))\right|$
remain uniformly bounded in a neighborhood of $y=0$. By the mean
value theorem, this holds if $\left|\partial_{x}^{2}f_{1}(x,t)\right|$
remains uniformly bounded in $B_{\delta}$, which is just condition
(b) above. Therefore, in addition to our assumptions in (\ref{eq:conditions for x^*}),
we need to assume condition (\ref{eq:Hessian of f_1}) for (a)-(d)
to hold. In summary, under the conditions, (\ref{eq:conditions for x^*})
and (\ref{eq:Hessian of f_1}), the local invariant manifold results
of \citet{kloeden11} are applicable near the $Y=0$ fixed point of
system (\ref{eq:dot Y eq def})-(\ref{eq:A G def}). 

Specifically, under conditions (\ref{eq:conditions for x^*}) and
(\ref{eq:Hessian of f_1}), the results in Section 6.3 of \citet{kloeden11}
apply to general non-autonomous systems of differential equations
of the form (\ref{eq:dot Y eq def}). These results in turn build
on classic result of \citet{sacker78} that establish the existence
of a dichotomy spectrum $\Sigma$ for $\mathcal{A}(t)$ that consists
up to $n+1$ disjoint closed intervals of the form
\begin{equation}
\Sigma=\left[a_{1},b_{1}\right]\cap\ldots\cap\left[a_{m},b_{m}\right],\quad m\leq m+1.\label{eq:dichotomy spectrum}
\end{equation}
(By the definition of the dichotomy spectrum, the linear system of
ODE $\dot{Y}=$ $\left(\mathcal{A}(t)-\lambda I\right)Y$ admits no
exponential dichotomy for any $\lambda\in\Sigma$.) Assuming $m\geq2$,
one can therefore select constants $\kappa_{j}^{+},\kappa_{-j}^{-}\in\mathbb{R}$
for $j=1,\ldots,m-1$ from the gaps among the closed spectral subintervals
in (\ref{eq:dichotomy spectrum}) as follows:
\[
b_{j}<\kappa_{j}^{+}<\kappa_{j}^{-}<a_{j+1},\quad j=1,\ldots,m-1,
\]
such that for appropriate constants $K>0$ and projection maps $P_{\pm}^{j}(t_{0})\in\mathbb{R}^{(n+1)\times\left(n+1\right)}$
with $P_{\pm}^{j}(t_{0})P_{\pm}^{j}(t_{0})=P_{\pm}^{j}(t_{0})$, the
normalized fundamental matrix solution $\Phi\left(t,t_{0}\right)$
of $\dot{Y}=\mathcal{A}(t)Y$satisfies
\begin{align}
\left\Vert \Phi\left(t,t_{0}\right)P_{-}^{j}(t_{0})\right\Vert  & \leq Ke^{\kappa_{j}^{+}(t-t_{0})},\quad t\geq t_{0},\nonumber \\
\left\Vert \Phi\left(t,t_{0}\right)P_{+}^{j}(t_{0})\right\Vert  & \leq Ke^{\kappa_{j}^{-}(t-t_{0})},\quad t\leq t_{0},\label{eq:exponential dichotomy}
\end{align}
for all $j=1,\ldots,m-1$. 

Then, by Theorem 6.10 and Remark 6.6 of \citet{kloeden11}, for each
$j=1,\ldots,m-1$, there exist two non-autonomous invariant manifolds
$W_{j}^{\pm}(t)$ for system (\ref{eq:dot Y eq def}) such that:
\begin{description}
\item [{(i)}] $\mathcal{W}_{j}^{\pm}(t)$ contain the $Y=0$ fixed point
of system (\ref{eq:dot Y eq def}). 
\item [{(ii)}] In a neighborhood $U\subset\mathbb{R}^{n+1},$ the manifolds
$\mathcal{W}_{j}^{\pm}(t)$ can described as graphs with the help
of $C^{0}$ functions $w_{j}^{\pm}\colon U\times\mathbb{\mathbb{R}}\to\mathbb{R}^{n+1}$
such that
\begin{equation}
\mathcal{W}_{j}^{\pm}(t)=\left\{ Y=Z+w_{j}^{\pm}(Z,t)\in U:\,\,Z\in\mathrm{range}\left(P_{\pm}^{j}(t)\right),\,\,\,\,w_{j}^{\pm}(Z,t)\in\mathrm{range}\left(P_{\mp}^{j}(t)\right)\right\} .\label{eq:wjdef}
\end{equation}
\item [{(iii)}] $w_{j}^{\pm}(Z,t)$ is uniformly $o$$\left(\left|Z\right|\right)$,
i.e., $\lim_{u\to0}\frac{\left\Vert w_{j}^{\pm}(Z,t)\right\Vert }{\left\Vert Z\right\Vert }=0$,
uniformly in $t$. 
\item [{(iv)}] If 
\begin{equation}
m_{j}^{+}\kappa_{j}^{+}<\kappa_{j}^{-}\label{eq:first spectral gap inequality}
\end{equation}
 holds for a positive integer $m_{j}^{+}$, then $\mathcal{W}_{j}^{+}(t)$
is of class $C^{m_{j}^{+}}.$ Similarly, if 
\begin{equation}
\kappa_{j}^{+}<m_{j}^{-}\kappa_{j}^{-}\label{eq:second spectral gap inequality}
\end{equation}
holds for a positive integer $m_{j}^{-}$ , then $\mathcal{W}_{j}^{-}(t)$
is of class $C^{m_{j}^{-}}.$
\item [{(v)}] If $\kappa_{j}^{+}<0$ then for all $\gamma>\kappa_{j}^{+}$,
we have $\sup_{t\geq0}\left\Vert Y(t,t_{0},Y_{0})\right\Vert e^{-\gamma t}<\infty$
for all $Y_{0}\in\mathcal{W}_{j}^{+}(t_{0})\cap U$ in a small enough
neighborhood $U$ of $Y=0$. Similarly, if $\kappa_{j}^{-}>0$ then
for all $\gamma<\kappa_{j}^{-}$, we have $\sup_{t\geq0}\left\Vert Y(t,t_{0},Y_{0})\right\Vert e^{-\gamma t}<\infty$
for all $Y_{0}\in\mathcal{W}_{j}^{+}(t_{0})\cap U$ in a small enough
neighborhood $U$ of $Y=0$.
\end{description}
Based on these results, further non-autonomous invariant manifolds
can be obtained by letting
\begin{equation}
\mathcal{W}_{i,j}(t)=\mathcal{W}_{i}^{+}(t)\cap\mathcal{W}_{j}^{-}(t),\quad1\leq j<i\leq m-1.\label{eq:W_i,j}
\end{equation}

Note that Theorem 6.10 and Remark 6.6 of \citet{kloeden11} assume
no hyperbolicity for the $Y=0$ fixed point of system (\ref{eq:dot Y eq def}),
which makes them applicable to the $Y=0$ fixed point of the extended
system (\ref{eq:dot Y eq def}). Specifically, the dichotomy spectrum
$\Sigma$ of the matrix $\mathcal{A}(t)$ defined in eq. (\ref{eq:A G def})
is discrete and given by
\[
\Sigma=\left\{ \mu_{1},\mu_{2},\ldots,\mu_{c}=0,\ldots,\mu_{m+1}\right\} =\mathrm{Re}\left[\mathrm{spect}\left(A\right)\right]\cup\left\{ 0\right\} .
\]
 As \citet{kloeden11} point out, the definition (\ref{eq:W_i,j})
yields a hierarchy of invariant manifolds, which in turn implies the
hierarchy of invariant manifolds shown in Table 1 for the original
system (\ref{eq: nonlinear system}) arising from this construct for
the extended system.
\begin{table}[H]
\centering{}%
\begin{tabular}{ccccccccc}
$\mathcal{W}_{1}^{+}(t)$ & $\subset$ & $\mathcal{W}_{2}^{+}(t)$ & $\subset$ & $\cdots$ & $\subset$ & $\mathcal{W}_{n-1}^{+}(t)$ & $\subset$ & $\mathbb{R}^{n+1}$\tabularnewline
 &  & $\cup$ &  & $\cup$ &  & $\cup$ &  & $\cup$\tabularnewline
 &  & $\mathcal{W}_{2,1}(t)$ & $\subset$ & $\cdots$ & $\subset$ & $\mathcal{W}_{n-1,1}(t)$ & $\subset$ & $\mathcal{W}_{1}^{-}(t)$\tabularnewline
 &  &  &  &  &  & $\cup$ &  & $\cup$\tabularnewline
 &  &  &  & $\ddots$ &  & $\vdots$ &  & $\vdots$\tabularnewline
 &  &  &  &  &  & $\mathcal{W}_{n-1,n-2}(t)$ & $\subset$ & $\mathcal{W}_{n-2}^{-}(t)$\tabularnewline
 &  &  &  &  &  &  &  & $\cup$\tabularnewline
 &  &  &  &  &  &  &  & $\mathcal{W}_{n-1}^{-}(t)$\tabularnewline
\end{tabular}\caption{Hierarchy of invariant manifolds for the extended system (\ref{eq:dot Y eq def}).\label{tab:Hierarchy-of-invariant-manifolds}}
\end{table}

For any index $j\leq c$, all invariant manifolds $\mathcal{W}_{i,j}^{+}(t)$
defined in formula (\ref{eq:W_i,j}) are graphs over a spectral subspace
that contains the $\mu_{c}=0$ center direction and hence they are
as smooth in $\epsilon$ as they are in the other variables along
that spectral subspace. This smoothness property is in turn shared
by all invariant manifolds of the form
\begin{align*}
\mathcal{W}_{i,c}(t) & =\mathcal{W}_{i}^{+}(t)\cap\mathcal{W}_{c}^{-}(t),\quad1\leq c<i\leq m-1.\\
\end{align*}

The $\epsilon=const.$ slices of $\mathcal{W}_{i}^{+}(t)$, $\mathcal{W}_{j}^{-}(t)$
and $\mathcal{W}_{i,j}(t)$ then provide similar invariant manifolds
$W_{i}^{+}(t;\epsilon)$, $W_{j}^{-}(t;\epsilon)$ and $W_{i,j}(t;\epsilon)$
of one less dimension along the $x^{*}(t;\epsilon)$ trajectory of
the original system (\ref{eq: nonlinear system}) for small enough
with $\epsilon>0$ under the assumed scaling (\ref{eq:epsilon scalings}).
The degree of smoothness of $W_{i}^{+}(t;\epsilon)$, $W_{j}^{-}(t;\epsilon)$
and $W_{i,j}(t;\epsilon)$ in $x$ and $\epsilon$ coincides with
the general degree of smoothness that can be inferred from the above
construct for their extended counterparts, $\mathcal{W}_{j}^{+}(t)$,
$\mathcal{W}_{j}^{-}(t)$ and $\mathcal{W}_{i,j}(t)$. Specifically,
the following non-autonomous invariant manifolds can be inferred from
Table \ref{tab:Hierarchy-of-invariant-manifolds} for system (\ref{eq: nonlinear system}):

(A) $W_{i}^{+}(t;\epsilon)$: the time-dependent smooth continuations
of the fastest decaying $i$ modes for any $1\leq i\leq n-1$. These
manifolds satisfy (\ref{eq:first spectral gap inequality}) for arbitrary
large $m$ and hence are as smooth as the original dynamical system.
They are generally $C^{0}$ in $\epsilon$ which can be improved to
$C^{r}$ for pseudo-stable manifolds, i.e., for all $W_{i}^{+}(t;\epsilon)$
with $\mu_{i}>0$.

(B) $W_{j}^{-}(t;\epsilon)$: the time-dependent smooth continuations
of the slowest decaying (or even growing) $n-1-j$ modes. If $\mu_{j-1}/\mu_{j}>m_{j}^{-}$
for some positive integer $m_{j}^{-}$, then these manifolds satisfy
(\ref{eq:second spectral gap inequality}) with $m_{j}^{-1}$ and
hence are of smoothness class $C^{m_{j}^{-}}$. For all $j$ with
$\mu_{j}<0$, $W_{j}^{-}(t;\epsilon)$ is also $C^{m_{j}^{-}}$ smooth
in $\epsilon$. Any manifold $W_{j}^{-}(t;\epsilon)$ with $\mu_{j}>0$
(fast unstable manifolds) can only be concluded to be $C^{0}$ in
$\epsilon$ from this argument.

(C) $W_{i,j}(t;\epsilon)$; the continuation of any (like-mode or
mixed-mode) spectral subspace that is $r_{i,j}-$normally hyperbolic.
In that case, $W_{i,j}(t;\epsilon)$ is also of smoothness class $C^{r_{i,j}}$
in $\epsilon$. 

\subsection{Proof of Theorem 4: Approximation of the SSM, $W\left(E,t\right)$,
anchored at $x_{\epsilon}^{*}(t)$ \label{subsec:  Proof of approximation of SSM}}

\subsubsection{Invariance equation}

Next, we derive an approximation for the non-autonomous SSMs anchored
to the hyperbolic trajectory $x^{*}(t)$. We perform a linear change
of coordinates 
\[
\left(\begin{array}{c}
u\\
v
\end{array}\right)=P^{-1}\left(x-x_{\epsilon}^{*}(t)\right),
\]
where $P=[e_{1},\ldots,e_{n}]\in\mathbb{C}^{n}$ contains the complex
eigenvectors corresponding to the ordered eigenvalues (\ref{eq:eigenvalue ordering})
of $A$, and $\left(u,v\right)\in\mathbb{C}^{d}\times\mathbb{C}^{n-d}$.
Rewriting the scaled version (\ref{eq:original system 1}) of system
\eqref{eq: nonlinear system} in these complex coordinates, we obtain
\begin{align}
\left(\begin{array}{c}
\dot{u}\\
\dot{v}
\end{array}\right) & =P^{-1}\left(\dot{x}-\dot{x}_{\epsilon}^{*}(t)\right)=P^{-1}\left[Ax+f(x,t)-\dot{x}_{\epsilon}^{*}(t)\right]\nonumber \\
 & =P^{-1}AP\left(\begin{array}{c}
u\\
v
\end{array}\right)+P^{-1}Ax_{\epsilon}^{*}(t)+P^{-1}\left[f\left(x_{\epsilon}^{*}(t)+P\left(\begin{array}{c}
u\\
v
\end{array}\right),t\right)-\dot{x}_{\epsilon}^{*}(t)\right]\nonumber \\
 & =\left(\begin{array}{cc}
A^{u} & 0\\
0 & A^{v}
\end{array}\right)\left(\begin{array}{c}
u\\
v
\end{array}\right)+\hat{f}(u,v,\epsilon;t),\label{eq:(u,v) system}
\end{align}
where
\begin{align}
\hat{f}(u,v,\epsilon;t) & =P^{-1}\left[f_{0}\left(x_{\epsilon}^{*}(t)+P\left(\begin{array}{c}
u\\
v
\end{array}\right)\right)+\epsilon\tilde{f}_{1}\left(x_{\epsilon}^{*}(t)+P\left(\begin{array}{c}
u\\
v
\end{array}\right),t\right)+Ax_{\epsilon}^{*}(t)-\dot{x}_{\epsilon}^{*}(t)\right],\label{eq:fhat definition}
\end{align}
 and 
\begin{align*}
 & A^{u}=\left(\begin{array}{ccc}
\lambda_{1} & 0 & 0\\
0 & \ddots & 0\\
0 & 0 & \lambda_{d}
\end{array}\right),\quad A^{v}=\left(\begin{array}{ccc}
\lambda_{d+1} & 0 & 0\\
0 & \ddots & 0\\
0 & 0 & \lambda_{n}
\end{array}\right).
\end{align*}
 Note that under the scaling (\ref{eq:original system 1}), the anchor
trajectory $x^{*}(t)$ becomes $\epsilon$-dependent, which is reflected
by our modified notation $x_{\epsilon}^{*}(t)$ for the same trajectory. 

By definition, we have
\[
x_{0}^{*}(t)\equiv0,
\]
 and

\[
f_{0}\left(x_{\epsilon}^{*}(t)\right)+Ax_{\epsilon}^{*}(t)-\dot{x}_{\epsilon}^{*}(t)+\epsilon\tilde{f}_{1}\left(x_{\epsilon}^{*}(t),t\right)\equiv0,
\]
therefore
\begin{equation}
D_{\epsilon}^{p}\left[f_{0}\left(x_{\epsilon}^{*}(t)\right)+Ax_{\epsilon}^{*}(t)-\dot{x}_{\epsilon}^{*}(t)+\epsilon\tilde{f}_{1}\left(x_{\epsilon}^{*}(t),t\right)\right]=D_{\epsilon}^{p}\hat{f}(0,0,\epsilon;t)\equiv0,\quad p\in\mathbb{N},\quad t\in\mathbb{R}.\label{eq:epsilon derivatives}
\end{equation}
By the definition of $\hat{f}$ in (\ref{eq:fhat definition}) and
by formula (\ref{eq:epsilon derivatives}) with $p=1$ , we also have
\begin{equation}
\hat{f}(0,0,0;t)=0,\quad D_{u}\hat{f}(0,0,0;t)=0,\quad D_{v}\hat{f}(0,0,0;t)=0,\quad D_{\epsilon}\hat{f}(0,0,0;t)=0.\label{eq:vanishing derivatives}
\end{equation}

For $\epsilon=0$, under the nonresonance conditions (\ref{eq:external non-resonance condition}),
we have a unique, primary spectral submanifold 
\begin{equation}
W_{0}\left(E\right)=\left\{ \left(u,v\right)\in U\subset\mathbb{R}^{n}:\,v=h_{0}(u)=\sum_{\left|\mathbf{k}\right|\geq0}h^{\mathbf{k}}u^{\mathbf{k}}\right\} ,\label{eq:W_0(E)}
\end{equation}
where $h_{0}(u)=\mathcal{O}\left(\left|u\right|^{2}\right)\in C^{r}$
defines the primary SSM as a smooth graph over $E$ with a quadratic
tangency to $E$ at $x=0$ (see \citet{haller16}). For $\epsilon>0$
small, under the conditions of statement (i) of Theorem \ref{thm:non-automous SSM},
this primary SSM persist in the form of a (generally non-unique) invariant
manifold
\begin{equation}
W_{\epsilon}\left(E,t\right)=\left\{ \left(u,v\right)\in U\subset\mathbb{R}^{n}:\,v=h_{\epsilon}(u,t)=\sum_{\left|\mathbf{k}\right|,p\geq0}h^{\mathbf{k}p}(t)u^{\mathbf{k}}\epsilon^{p}\right\} ,\label{eq:h double expansion}
\end{equation}
with coefficients $h^{\mathbf{k}p}(t)$ that are uniformly bounded
in $t$ within a $B_{\delta}$ ball around the origin $x=0$ for $0\leq\epsilon\leq\delta$,
and with the notation
\[
u^{\mathbf{k}}=u_{1}^{k_{1}}\cdot\ldots\cdot u_{d}^{k_{d}}.
\]
Note that this expansion is only valid up to the order of smoothness
of the SSM. For general members of the $W_{\epsilon}\left(E,t\right)$,
this smoothness can only guaranteed to be class $C^{\rho}$, but exceptional
members may admit higher-order expansions, just as primary SSMs do
in the case of autonomous, time-periodic and time-quasiperiodic SSMs
(see \citet{haller16}). 

Also note that for the smooth persistence of $W_{0}(E)$ as $W_{\epsilon}\left(E,t\right)$,
a possible $1\colon1$ resonance between the an eigenvalue inside
$E$ and another one outside $E$ has to be excluded in order to secure
the normal hyperbolicity of $W_{0}(E)$. This is reflected by the
lowering of the lower index in the summation in eq. \eqref{eq:strenghtened nonresonance under epsilon perturbation}
from $2$ to $1$ relative to what one requires for the existence
of $W_{0}\left(E\right)$. 

Differentiating the definition (\ref{eq:h double expansion}) of the
invariant manifold $W_{\epsilon}\left(E,t\right)$ in time and using
the system of ODEs (\ref{eq:(u,v) system}), we obtain
\begin{align}
\dot{v} & =D_{u}h_{\epsilon}(u,t)\dot{u}+D_{t}h_{\epsilon}(u,t)\nonumber \\
 & =D_{u}h_{\epsilon}(u,t)\left[A^{u}u+\hat{f}^{u}(u,h_{\epsilon}(u,t),t;\epsilon)\right]+D_{t}h_{\epsilon}(u,t)\label{eq:vdoteq1}
\end{align}
 At the same time, substitution into (\ref{eq:(u,v) system}) gives
\begin{align}
\dot{v} & =A^{v}h_{\epsilon}(u,t)+\hat{f}^{v}(u,h_{\epsilon}(u,t),t;\epsilon).\label{eq:vdoteq2}
\end{align}
Comparing (\ref{eq:vdoteq1}) and (\ref{eq:vdoteq2}) give the invariance
PDE satisfied by $W_{\epsilon}\left(E,t\right)$:
\begin{equation}
D_{u}h_{\epsilon}\left[A^{u}u+\hat{f_{0}^{u}}+\epsilon\hat{f_{1}^{u}}\right]+D_{t}h_{\epsilon}=A^{v}h_{\epsilon}+\hat{f_{0}^{v}}+\epsilon\hat{f_{1}^{v}}.\label{eq:invariance PDE}
\end{equation}

As the origin $(u,v)=(0,0)$ is a fixed point of system (\ref{eq:(u,v) system})
for all $t$ and $\epsilon$, we must have $h(0,t;\epsilon)\equiv0$
for all $\epsilon\geq0$ small and all $t\in\mathbb{R}$. This, in
turn, implies, 
\begin{equation}
h^{\mathbf{0}p}(t)\equiv0,\quad t\in\mathbb{R},\quad p\in\mathbb{N}.\label{eq:origin is invariant}
\end{equation}
 Furthermore, for $\epsilon=0$, the non-autonomous SSM $W_{\epsilon}\left(E,t\right)$
becomes the autonomous SSM $W(E)$, and hence we must have 
\begin{equation}
h^{\mathbf{k}0}(t)\equiv h^{\mathbf{k}},\quad t\in\mathbb{R},\quad\mathbf{k}\in\mathbb{N}^{d},\qquad h^{\mathbf{k}0}(t)\equiv h^{\mathbf{k}}=0,\quad\left|\mathbf{k}\right|=1,\label{eq:epsilon=00003D0 limit of h}
\end{equation}
 with the time-independent Taylor series coefficients $h^{\mathbf{k}}$
in the expansion for $W_{0}(E)$ in eq. (\ref{eq:W_0(E)}). 

\subsubsection{Structure of the invariance equation}

Substitution of (\ref{eq:h double expansion}) into the invariance
equation (\ref{eq:invariance PDE}) gives
\[
\sum_{\left|\mathbf{k}\right|,p\geq0}h^{\mathbf{k}p}(t)\epsilon^{p}D_{u}u^{\mathbf{k}}\left[A^{u}u+\hat{f}^{u}\left(u,\sum_{\left|\mathbf{k}\right|,p\geq0}h^{\mathbf{k}p}(t)u^{\mathbf{k}}\epsilon^{p},\epsilon,t\right)\right]+\sum_{\left|\mathbf{k}\right|,p\geq0}\dot{h}^{\mathbf{k}p}(t)u^{\mathbf{k}}\epsilon^{p}
\]
 
\begin{equation}
=\sum_{\left|\mathbf{k}\right|\geq0,p\geq0}A^{v}h^{\mathbf{k}p}(t)u^{\mathbf{k}}\epsilon^{p}+\hat{f}^{v}\left(u,\sum_{\left|\mathbf{k}\right|,p\geq0}h^{\mathbf{k}p}(t)u^{\mathbf{k}}\epsilon^{p},\epsilon,t\right).\label{eq:nonautonomous invariance eq.-1}
\end{equation}
We observe that 
\begin{align}
A^{v}h^{\mathbf{k}p}(t)u^{\mathbf{k}}-h^{\mathbf{k}p}(t)D_{u}u^{\mathbf{k}}A^{u}u & =\left(\begin{array}{ccc}
\lambda_{d+1} & 0 & 0\\
0 & \ddots & 0\\
0 & 0 & \lambda_{n}
\end{array}\right)\left(\begin{array}{c}
h_{1}^{\mathbf{k}p}\\
\vdots\\
h_{n-d}^{\mathbf{k}p}
\end{array}\right)u^{\mathbf{k}}\nonumber \\
 & -\left(\begin{array}{c}
h_{1}^{\mathbf{k}p}\\
\vdots\\
h_{n-d}^{\mathbf{k}p}
\end{array}\right)\left(\begin{array}{ccc}
\frac{k_{1}u^{\mathbf{k}}}{u_{1}} & \cdots & \frac{k_{d}u^{\mathbf{k}}}{u_{d}}\end{array}\right)\left(\begin{array}{c}
\lambda_{1}u_{1}\\
\vdots\\
\lambda_{d}u_{d}
\end{array}\right)\nonumber \\
 & =\left(\begin{array}{c}
\lambda_{d+1}h_{1}^{\mathbf{k}p}\\
\vdots\\
\lambda_{n}h_{n-d}^{\mathbf{k}p}
\end{array}\right)u^{\mathbf{k}}-\left(\begin{array}{ccc}
h_{1}^{\mathbf{k}p}\frac{k_{1}u^{\mathbf{k}}}{u_{1}} & \cdots & h_{1}^{\mathbf{k}p}\frac{k_{d}u^{\mathbf{k}}}{u_{d}}\\
\vdots & \ddots & \vdots\\
h_{n-d}^{\mathbf{k}p}\frac{k_{1}u^{\mathbf{k}}}{u_{1}} & \cdots & h_{n-d}^{\mathbf{k}p}\frac{k_{d}u^{\mathbf{k}}}{u_{d}}
\end{array}\right)\left(\begin{array}{c}
\lambda_{1}u_{1}\\
\vdots\\
\lambda_{d}u_{d}
\end{array}\right)\nonumber \\
 & =\left(\begin{array}{c}
\lambda_{d+1}h_{1}^{\mathbf{k}p}\\
\vdots\\
\lambda_{n}h_{n-d}^{\mathbf{k}p}
\end{array}\right)u^{\mathbf{k}}-\left(\begin{array}{c}
h_{1}^{\mathbf{k}p}\sum_{j=1}^{d}k_{j}\lambda_{j}\\
\vdots\\
h_{n-d}^{\mathbf{k}p}\sum_{j=1}^{d}k_{j}\lambda_{j}
\end{array}\right)u^{\mathbf{k}}\nonumber \\
 & =\left(\begin{array}{ccc}
\lambda_{d+1}-\sum_{j=1}^{d}k_{j}\lambda_{j} & \cdots & 0\\
\vdots & \ddots & \vdots\\
0 & \cdots & \lambda_{n}-\sum_{j=1}^{d}k_{j}\lambda_{j}
\end{array}\right)h^{\mathbf{k}p}(t)u^{\mathbf{k}}.\label{eq:leading order terms in invariance equation-1}
\end{align}
Therefore, the invariance equation (\ref{eq:nonautonomous invariance eq.-1})
can be rewritten as
\begin{equation}
\sum_{\left|\mathbf{k}\right|,p\geq0}\dot{h}^{\mathbf{k}p}(t)u^{\mathbf{k}}\epsilon^{p}=\sum_{\left|\mathbf{k}\right|,p\geq0}A_{\mathbf{k}}h^{\mathbf{k},p}(t)u^{\mathbf{k}}\epsilon^{p}+\sum_{\left|\mathbf{k}\right|,p\geq0}M^{\mathbf{k}p}(t,h^{\mathbf{j}m})u^{\mathbf{k}}\epsilon^{p},\label{eq:simplified invariance eq.-1}
\end{equation}
where
\begin{align}
A_{\mathbf{k}} & =\mathrm{diag}\left[\lambda_{\ell}-\sum_{j=1}^{d}k_{j}\lambda_{j}\right]_{\ell=d+1}^{n}\in\mathbb{C}^{\left(n-d\right)\times\left(n-d\right)},\nonumber \\
\sum_{\left|\mathbf{k}\right|,p\geq0}M^{\mathbf{k}p}(t,h^{\mathbf{j}m})u^{\mathbf{k}}\epsilon^{p} & =M(u,h^{\mathbf{j}m},\epsilon;t)=\hat{f}^{v}\left(u,\sum_{\left|\mathbf{k}\right|,p\geq0}h^{\mathbf{k}p}(t)u^{\mathbf{k}}\epsilon^{p},\epsilon;t\right)\label{eq:M^k-1}
\end{align}
\[
-\sum_{\left|\mathbf{k}\right|,p\geq0}\epsilon^{p}\left(\begin{array}{ccc}
h_{1}^{\mathbf{k}p}(t)\frac{k_{1}u^{\mathbf{k}}}{u_{1}} & \cdots & h_{1}^{\mathbf{k}p}(t)\frac{k_{d}u^{\mathbf{k}}}{u_{d}}\\
\vdots & \ddots & \vdots\\
h_{n-d}^{\mathbf{k}p}(t)\frac{k_{1}u^{\mathbf{k}}}{u_{1}} & \cdots & h_{n-d}^{\mathbf{k}p}(t)\frac{k_{d}u^{\mathbf{k}}}{u_{d}}
\end{array}\right)\hat{f}^{u}\left(u,\sum_{\left|\mathbf{k}\right|,p\geq0}h^{\mathbf{k}p}(t)u^{\mathbf{k}}\epsilon^{p},\epsilon;t\right).
\]

Recall from eq. (\ref{eq:vanishing derivatives}) that $\hat{f}^{v}$
and $\hat{f}^{u}$ vanish for $u,v,\epsilon=0$ and have no $\mathcal{O}\left(\left|v\right|\right)$
terms. Therefore, when $\sum_{\left|\mathbf{k}\right|,p\geq0}h^{\mathbf{k}p}(t)u^{\mathbf{k}}\epsilon^{p}$
is substituted into the terms in the Taylor expansion of $\hat{f}^{v}\left(u,v,\epsilon;t\right)$,
then it is multiplied in each case by at least the first power of
$u$ or at least the first power of $\epsilon.$ (When $\sum_{\left|\mathbf{k}\right|,p\geq0}h^{\mathbf{k}p}(t)u^{\mathbf{k}}\epsilon^{p}$
is substituted into terms of $\mathcal{O}\left(\left|v\right|^{2}\right)$
then it is multiplied by itself, but $h^{\mathbf{0}0}(t)\equiv0$,
and hence the lowest order term multiplying $\sum_{\left|\mathbf{k}\right|,p\geq0}h^{\mathbf{k}p}(t)u^{\mathbf{k}}\epsilon^{p}$
will be again at least the first power of $u$ or at least the first
power of $\epsilon.$) As a result, $M^{\mathbf{k}p}(h^{\mathbf{j}m},t)$
will only depend on $h^{\mathbf{j}m}$ that are lower in order, i.e.,
\[
D_{h^{\mathbf{j}m}}M^{\mathbf{k}p}(h^{\mathbf{j}m},t)=0,\quad\left|(\mathbf{j},m)\right|\geq\left|(\mathbf{k},p)\right|.
\]

Equating coefficients of equal powers of $u$ in eq. (\ref{eq:simplified invariance eq.-1}),
we therefore obtain the recursively solvable linear system of inhomogeneous
linear ODEs
\begin{equation}
\dot{h}^{\mathbf{k}p}(t)=A_{\mathbf{k}}h^{\mathbf{k}p}(t)+M^{\mathbf{k}p}(t,h^{\mathbf{j}m}),\quad\left|(\mathbf{j},m)\right|<\left|(\mathbf{k},p)\right|,\label{eq:system of inhomogeneous ODEs-1}
\end{equation}
with $M^{\mathbf{k}p}(t,h^{\mathbf{j}m})$ defined in formula (\ref{eq:M^k-1-2-1}).
So far, we know from eqs. (\ref{eq:origin is invariant}) and (\ref{eq:epsilon=00003D0 limit of h})
that
\begin{equation}
h^{\mathbf{0}p}(t)\equiv0,\quad t\in\mathbb{R},\quad p\in\mathbb{N},\qquad h^{\mathbf{k}0}(t)\equiv h^{\mathbf{k}},\quad t\in\mathbb{R},\quad\mathbf{k}\in\mathbb{N}^{d},\label{eq:elementary identities about indices}
\end{equation}
which is a homogeneous system of ODEs in (\ref{eq:system of inhomogeneous ODEs-1})
for $\left|\mathbf{k}\right|>0$. 

\subsubsection{Solution of the invariance equation}

A first observation about solving the family (\ref{eq:system of inhomogeneous ODEs-1})
of linear ODEs is that for arbitrary $\left|\mathbf{k}\right|$ and
$p=0$, we obtain from eqs. (\ref{eq:epsilon=00003D0 limit of h})
and (\ref{eq:system of inhomogeneous ODEs-1}) the algebraic equations
\begin{align*}
0 & =A_{\mathbf{k}}h^{\mathbf{k}0}+M^{\mathbf{k}0}(h^{\mathbf{j}0}).
\end{align*}
We know from classic SSM theory that this algebraic system of equations
is a recursively solvable linear system of equations, because $\left|\mathbf{j}\right|<\left|\mathbf{k}\right|$
holds for all $M^{\mathbf{k}0}(h^{\mathbf{j}0})$, and hence 
\begin{equation}
h^{\mathbf{k}0}=-A_{\mathbf{k}}^{-1}M^{\mathbf{k}0}(h^{\mathbf{j}0}),\quad\left|\mathbf{j}\right|<\left|\mathbf{k}\right|.\label{eq:coeffs in autonomus limit}
\end{equation}

Next, we note that the nonresonance assumption (\ref{eq:full nonresonance conditions})
implies that the coefficient matrix $A_{\mathbf{k}}$ of the homogeneous
part of eq. (\ref{eq:system of inhomogeneous ODEs-1}) is nonsingular.
We additionally require now this homogenous part to admit a hyperbolic
fixed point, which leads to the stronger nonresonance condition (\ref{eq:strict nonresonance})
of Theorem \ref{thm:non-automous SSM}. This is the same non-resonance
condition that arises in the work of \citet{haro06} for quasiperiodic
forcing, which is a subset of the general forcing class we are considering
here.

Under the nonresonance condition (\ref{eq:strict nonresonance}),
all diagonal elements of $A_{\mathbf{k}}$ have nonzero real parts.
We can therefore uniquely split $A_{\mathbf{k}}$ into the sum of
a diagonal matrix $A_{\mathbf{k}}^{-}$ which contains the stable
eigenvalues of $A_{\mathbf{k}}$ as well as zeros, and a diagonal
matrix $A_{\mathbf{k}}^{+}$ which contains the unstable eigenvalues
of $A_{\mathbf{k}}$ as well as zeros:
\begin{align*}
A_{\mathbf{k}} & =A_{\mathbf{k}}^{-}+A_{\mathbf{k}}^{+},\\
A_{\mathbf{k}}^{-}(j,j) & =\left\{ \begin{array}{ccc}
A_{\mathbf{k}}(j,j), &  & \mathrm{Re}A_{\mathbf{k}}(j,j)<0,\\
\\
0, &  & \mathrm{otherwise},
\end{array}\right.\quad A_{\mathbf{k}}^{+}(j,j)=\left\{ \begin{array}{ccc}
A_{\mathbf{k}}(j,j), &  & \mathrm{Re}A_{\mathbf{k}}(j,j)>0,\\
\\
0, &  & \mathrm{otherwise}.
\end{array}\right.
\end{align*}
 We define the corresponding splitting for the vectors $h^{\mathbf{k}p}(t)$
and $M^{\mathbf{k}p}(t,h^{\mathbf{j}m}(t))$ as 
\begin{align*}
h^{\mathbf{k}p} & =h^{\mathbf{k}p-}+h^{\mathbf{k}p+},\quad\\
h_{j}^{\mathbf{k}p-} & =\left\{ \begin{array}{ccc}
h_{j}^{\mathbf{k}p}, &  & \mathrm{Re}A_{\mathbf{k}}(j,j)<0,\\
\\
0, &  & \mathrm{otherwise},
\end{array}\right.\quad h_{j}^{\mathbf{k}p+}=\left\{ \begin{array}{ccc}
h_{j}^{\mathbf{k}p}, &  & \mathrm{Re}A_{\mathbf{k}}(j,j)>0,\\
\\
0, &  & \mathrm{otherwise},
\end{array}\right.\\
\end{align*}
\begin{align*}
M^{\mathbf{k}p} & =M^{\mathbf{k}p-}+M^{\mathbf{k}p+},\\
M_{j}^{\mathbf{k}p-} & =\left\{ \begin{array}{ccc}
M_{j}^{\mathbf{k}p}, &  & \mathrm{Re}A_{\mathbf{k}}(j,j)<0,\\
\\
0, &  & \mathrm{otherwise},
\end{array}\right.\quad M_{j}^{\mathbf{k}p+}=\left\{ \begin{array}{ccc}
M_{j}^{\mathbf{k}p}, &  & \mathrm{Re}A_{\mathbf{k}}(j,j)>0,\\
\\
0, &  & \mathrm{otherwise}.
\end{array}\right..\\
\end{align*}
We then then split the expression for the solution of system (\ref{eq:system of inhomogeneous ODEs-1})
as
\begin{align}
h^{\mathbf{k}p-}(t) & =e^{A_{\mathbf{k}}^{-}\left(t-t_{0}\right)}h^{\mathbf{k}p-}(t_{0})+\int_{t_{0}}^{t}e^{A_{\mathbf{k}}^{-}\left(t-s\right)}M^{\mathbf{k}p-}\left(s,h^{\mathbf{j}m}(s)\right)\,ds,\nonumber \\
h^{\mathbf{k}p+}(t) & =e^{A_{\mathbf{k}}^{+}\left(t-t_{0}\right)}h^{\mathbf{k}p+}(t_{0})+\int_{t_{0}}^{t}e^{A_{\mathbf{k}}^{+}\left(t-s\right)}M^{\mathbf{k}p+}\left(s,h^{\mathbf{j}m}(s)\right)\,ds.\label{eq:hkpm}
\end{align}
If $h^{\mathbf{k}p}(t)$ is a uniformly bounded solution of system
(\ref{eq:system of inhomogeneous ODEs-1}) for all forward and backward
times, then using the signs of the nonzero diagonal entries of the
matrices $A_{\mathbf{k}}^{\pm}$, we obtain from the formulas (\ref{eq:hkpm})
that
\begin{align*}
h^{\mathbf{k}p-}(t) & =\lim_{t_{0}\to-\infty}\left[e^{A_{\mathbf{k}}^{-}\left(t-t_{0}\right)}h^{\mathbf{k}p-}(t_{0})+\int_{t_{0}}^{t}e^{A_{\mathbf{k}}^{-}\left(t-s\right)}M^{\mathbf{k}p-}\left(s,h^{\mathbf{j}m}(s)\right)\,ds\right]\\
 & =\int_{-\infty}^{t}e^{A_{\mathbf{k}}^{-}\left(t-s\right)}M^{\mathbf{k}p-}(s,h^{\mathbf{j}m}(s))\,ds,\\
h^{\mathbf{k}p-}(t) & =\lim_{t_{0}\to+\infty}\left[e^{A_{\mathbf{k}}^{+}\left(t-t_{0}\right)}h^{\mathbf{k}+}(t_{0})+\int_{t_{0}}^{t}e^{A_{\mathbf{k}}^{+}\left(t-s\right)}M^{\mathbf{k}p+}\left(s,h^{\mathbf{j}m}(s)\right)\,ds\right]\\
 & =-\int_{t}^{+\infty}e^{A_{\mathbf{k}}^{+}\left(t-s\right)}e^{A_{\mathbf{k}}^{+}\left(t-s\right)}M^{\mathbf{k}p+}\left(s,h^{\mathbf{j}m}(s)\right)\,ds,
\end{align*}
which gives 
\begin{align*}
h^{\mathbf{k}p}(t) & =h^{\mathbf{k}p-}(t)+h^{\mathbf{k}p-}(t)=\int_{-\infty}^{t}e^{A_{\mathbf{k}}^{-}\left(t-s\right)}M^{\mathbf{k}p-}(s,h^{\mathbf{j}m}(s))\,ds\\
 & -\int_{t}^{+\infty}e^{A_{\mathbf{k}}^{+}\left(t-s\right)}e^{A_{\mathbf{k}}^{+}\left(t-s\right)}M^{\mathbf{k}p+}\left(s,h^{\mathbf{j}m}(s)\right)\,ds.
\end{align*}
Therefore, introducing the Green's function as the diagonal matrix
$G_{\mathbf{k}}(t)\in\mathbb{C}^{\left(n-d\right)\times\left(n-d\right)}$
defined in formula (\ref{eq:G_k definition}), we can write the unique
globally bounded solution of eq. (\ref{eq:system of inhomogeneous ODEs-1})
in the form
\begin{equation}
h^{\mathbf{k}p}(t)=\int_{-\infty}^{\infty}G_{\mathbf{k}}(t-s)M^{\mathbf{k}p}(s,h^{\mathbf{j}m}(s))\,ds,\quad\left|(\mathbf{j},m)\right|<\left|(\mathbf{k},p)\right|,\quad t\in\mathbb{R},\label{eq:h_k formula}
\end{equation}
with $M^{\mathbf{k}p}(t,h^{\mathbf{j}m}(t))$ defined in formula (\ref{eq:M^k-1-2-1}).
As we have already seen in eqs. \ref{eq:elementary identities about indices})
and (\ref{eq:coeffs in autonomus limit}), we specifically have
\begin{equation}
h^{\mathbf{0}p}(t)\equiv0,\quad t\in\mathbb{R},\quad p\in\mathbb{N},\qquad h^{\mathbf{k}0}(t)\equiv-A_{\mathbf{k}}^{-1}M^{\mathbf{k}0}(h^{\mathbf{j}0}),\quad\left|\mathbf{j}\right|<\left|\mathbf{k}\right|,\quad h^{\mathbf{0}}=0.\label{eq:oelementary identities about indices-1}
\end{equation}

\subsubsection{Reduced dynamics}

We obtain the reduced dynamics on $W_{\epsilon}\left(E,t\right)$
by restricting the $u$-component of system (\ref{eq:(u,v) system})
to $W_{\epsilon}\left(E,t\right)$: 
\begin{equation}
\dot{u}=A^{u}u+\hat{f}^{u}\left(u,h_{\epsilon}\left(u,t\right)\right).\label{eq:final reduced dynamics}
\end{equation}
 We obtain the form (\ref{eq:final reduced dynamics-2}) of the reduced
dynamics by substituting the definition of $\hat{f}^{u}$ into eq.
(\ref{eq:final reduced dynamics}), using the defining relationship
\[
\dot{x}_{\epsilon}^{*}(t)=Ax_{\epsilon}^{*}(t)+f_{0}\left(x_{\epsilon}^{*}(t)\right)+\epsilon\tilde{f}_{1}\left(x_{\epsilon}^{*}(t),t\right)
\]
of the anchor trajectory $x_{\epsilon}^{*}(t)$, and noting that the
rows of $P^{-1}$ are the appropriately normalized left eigenvectors
of $A$, and hence $Q_{u}\in\mathbb{C}^{d\times n}$ in formula (\ref{eq:final reduced dynamics-2})
contains the first $d$ rows of $P^{-1}$. The end result is then:\emph{
\begin{align*}
\dot{u} & =A^{u}u+\hat{f}^{u}\left(u,h_{\epsilon}(u,t),\epsilon;t\right)\\
 & =A^{u}u+Q_{u}\left[f_{0}\left(x_{\epsilon}^{*}(t)+P\left(\begin{array}{c}
u\\
h_{\epsilon}(u,t)
\end{array}\right)\right)+\epsilon\tilde{f}_{1}\left(x_{\epsilon}^{*}(t)+P\left(\begin{array}{c}
u\\
h_{\epsilon}(u,t)
\end{array}\right),t\right)+Ax_{\epsilon}^{*}(t)-\dot{x}_{\epsilon}^{*}(t)\right]\\
 & =A^{u}u\\
 & +Q_{u}\left[f_{0}\left(x_{\epsilon}^{*}(t)+P\left(\begin{array}{c}
u\\
h_{\epsilon}(u,t)
\end{array}\right)\right)-f_{0}\left(x_{\epsilon}^{*}(t)\right)+\epsilon\tilde{f}_{1}\left(x_{\epsilon}^{*}(t)+P\left(\begin{array}{c}
u\\
h_{\epsilon}(u,t)
\end{array}\right),t\right)-\epsilon\tilde{f}_{1}\left(x_{\epsilon}^{*}(t),t\right)\right],
\end{align*}
}as claimed in statement (ii) of Theorem \ref{thm:computation of non-automous SSM}. 

The $\left(u,v\right)$ coordinates are measured from the perturbed
anchor trajectory $x^{*}(t)$. We can also express the reduced dynamics
on $W_{\epsilon}\left(E,t\right)$ in coordinates that are aligned
with the subspace $E$ and emanate from the original $x=0$ fixed
point of system (\ref{eq: nonlinear system}) for $\epsilon=0.$ Let
\[
\left(\begin{array}{c}
\xi\\
\eta
\end{array}\right)=P^{-1}x,
\]
which gives
\begin{equation}
\left(\begin{array}{c}
u\\
v
\end{array}\right)=P^{-1}\left(x-x_{\epsilon}^{*}(t)\right)=\left(\begin{array}{c}
\xi\\
\eta
\end{array}\right)-P^{-1}x_{\epsilon}^{*}(t),\label{eq:coordinate change to (u,v)-1}
\end{equation}
which implies
\[
\xi=u+Q_{u}x_{\epsilon}^{*}(t),\quad\eta=v+Q_{v}x_{\epsilon}^{*}(t)\qquad P^{-1}=\left[\begin{array}{c}
Q_{u}\\
Q_{v}
\end{array}\right],\quad Q_{u}\in\mathbb{C}^{d\times n},\quad Q_{u}\in\mathbb{C}^{(n-d)\times n}.
\]
 Note that 
\[
P^{-1}AP=\left[\begin{array}{cc}
A^{u} & 0_{d\times n}\\
0_{\left(n-d\right)\times n} & A^{v}
\end{array}\right],
\]
which implies
\[
\left[\begin{array}{c}
Q_{u}A\\
Q_{v}A
\end{array}\right]=P^{-1}A=\left[\begin{array}{cc}
A^{u} & 0_{d\times n}\\
0_{\left(n-d\right)\times n} & A^{v}
\end{array}\right]P^{-1}=\left[\begin{array}{cc}
A^{u} & 0_{d\times n}\\
0_{\left(n-d\right)\times n} & A^{v}
\end{array}\right]\left[\begin{array}{c}
Q_{u}\\
Q_{v}
\end{array}\right]=\left[\begin{array}{c}
A^{u}Q_{u}\\
A^{v}Q_{v}
\end{array}\right],
\]
yielding the identity
\begin{equation}
Q_{u}A=A^{u}Q_{u}.\label{eq:linear identity}
\end{equation}
Given that \emph{
\begin{align}
\dot{u} & =A^{u}u+Q_{u}\left[f_{0}\left(x_{\epsilon}^{*}(t)+P\left(\begin{array}{c}
u\\
h_{\epsilon}(u,t)
\end{array}\right)\right)+\epsilon\tilde{f}_{1}\left(x_{\epsilon}^{*}(t)+P\left(\begin{array}{c}
u\\
h_{\epsilon}(u,t)
\end{array}\right),t\right)+Ax_{\epsilon}^{*}(t)-\dot{x}_{\epsilon}^{*}(t)\right],\label{eq:final reduced dynamics-2-2}
\end{align}
}we obtain 
\begin{align*}
\dot{\xi} & =\dot{u}+Q_{u}\dot{x}_{\epsilon}^{*}(t)\\
 & =A^{u}u+Q_{u}\left[f_{0}\left(x_{\epsilon}^{*}(t)+P\left(\begin{array}{c}
u\\
h_{\epsilon}(u,t)
\end{array}\right)\right)+\epsilon\tilde{f}_{1}\left(x_{\epsilon}^{*}(t)+P\left(\begin{array}{c}
u\\
h_{\epsilon}(u,t)
\end{array}\right),t\right)+Ax_{\epsilon}^{*}(t)\right]\\
 & =A^{u}\left(\xi-Q_{u}x_{\epsilon}^{*}(t)\right)+Q_{u}Ax_{\epsilon}^{*}(t)\\
 & +Q_{u}\left[f_{0}\left(x_{\epsilon}^{*}(t)+P\left(\begin{array}{c}
u\\
h_{\epsilon}(u,t)
\end{array}\right)\right)+\epsilon\tilde{f}_{1}\left(x_{\epsilon}^{*}(t)+P\left(\begin{array}{c}
u\\
h_{\epsilon}(u,t)
\end{array}\right),t\right)\right]\\
 & =A^{u}\xi+\left[Q_{u}A-A^{u}Q_{u}\right]x_{\epsilon}^{*}(t)\\
 & +Q_{u}\left[f_{0}\left(x_{\epsilon}^{*}(t)+P\left(\begin{array}{c}
\xi-Q_{u}x_{\epsilon}^{*}(t)\\
h_{\epsilon}(\xi-Q_{u}x_{\epsilon}^{*}(t),t)
\end{array}\right)\right)+\epsilon\tilde{f}_{1}\left(x_{\epsilon}^{*}(t)+P\left(\begin{array}{c}
\xi-Q_{u}x_{\epsilon}^{*}(t)\\
h_{\epsilon}(\xi-Q_{u}x_{\epsilon}^{*}(t),t)
\end{array}\right),t\right)\right]\\
 & =A^{u}\xi+Q_{u}f_{0}\left(x_{\epsilon}^{*}(t)+P\left(\begin{array}{c}
\xi-Q_{u}x_{\epsilon}^{*}(t)\\
h_{\epsilon}(\xi-Q_{u}x_{\epsilon}^{*}(t),t)
\end{array}\right)\right)\\
 & +Q_{u}\epsilon\tilde{f}_{1}\left(x_{\epsilon}^{*}(t)+P\left(\begin{array}{c}
\xi-Q_{u}x_{\epsilon}^{*}(t)\\
h_{\epsilon}(\xi-Q_{u}x_{\epsilon}^{*}(t),t)
\end{array}\right),t\right)
\end{align*}
where we have used the identity (\ref{eq:linear identity}). This
completes the proof of statement (iii) of Theorem \ref{thm:computation of non-automous SSM}.

\section{Proof of Theorems \ref{thm:adiabatic anchor}-\ref{thm:computation of adiabatic SSM}}

\setcounter{equation}{0}

\subsection{Proof of Theorem \ref{thm:adiabatic anchor}: Expansion for the hyperbolic
slow manifold $\mathcal{L}_{\epsilon}$ \label{subsec:Proof of approximation of slow manifold}}

We now assume that $f$ and $x_{0}^{\prime}(\alpha)$ have $r\geq1$
uniformly bounded derivatives over a closed neighborhood of $x_{0}(\alpha)$
for all $\alpha$. Then, by the results of \citet{eldering13} on
the persistence of non-compact, normally hyperbolic invariant manifolds,
a unique and uniformly bounded slow manifold $\mathcal{L}_{\epsilon}$
exists near $\mathcal{L}_{0}$ for $\epsilon$ small enough. These
results are applicable because the $\mathcal{O}\left(\epsilon\right)$
perturbation term in (\ref{eq:slow-fast system}) is uniformly bounded
in the $C^{1}$ norm. The slow manifold $\mathcal{L}_{\epsilon}$
is $C^{r}$-diffeomorphic to $\mathcal{L}_{0}$ and is as smooth in
the $\epsilon$ parameter as system (\ref{eq:slow-fast system}).
As a consequence, we have an asymptotic expansion 
\begin{equation}
\mathcal{L}_{\epsilon}=\left\{ (x,\alpha)\in\mathbb{R}^{n}\times\mathbb{R}\colon\,\,\,\,\,x=x_{\epsilon}(\alpha)=\sum_{j=0}^{r}\epsilon^{j}x_{j}(\alpha)+o\left(\epsilon^{r}\right)\right\} ,\label{eq:expansion for x(alpha)}
\end{equation}
where the functions $x_{j}(\alpha)$ are uniformly bounded in $\alpha$.
By the definition of $x_{\epsilon}(\alpha)$ and by eq. (\ref{eq:alpha derivatives vanish}),
we have 
\begin{equation}
D_{\epsilon}^{p}\left[f\left(x_{\epsilon}(\alpha),\alpha\right)-\epsilon x_{\epsilon}^{\prime}(\alpha)\right]\equiv0,\quad p\in\mathbb{N},\quad\alpha\in\mathbb{R}.\label{eq:epsilon derivatives-1}
\end{equation}

Substituting the expansion (\ref{eq:expansion for x(alpha)}) into
(\ref{eq:slow-fast system}), we obtain
\[
\sum_{j\geq0}\epsilon^{j+1}x_{j}^{\prime}(\alpha)=f\left(\sum_{j\geq0}\epsilon^{j}x_{j}(\alpha),\alpha\right).
\]
Comparison of equal powers of $\epsilon$ in the last equation gives
\begin{equation}
x_{\nu-1}^{\prime}(\alpha)=\left.\frac{1}{\nu!}\frac{\partial^{\nu}}{\partial\epsilon^{\nu}}f\left(\sum_{j\geq0}\epsilon^{j}x_{j}(\alpha),\alpha\right)\right|_{\epsilon=0},\quad\nu\geq1.\label{eq:general x-nu formula}
\end{equation}
 Specifically, we have 
\[
\epsilon x_{0}^{\prime}(\alpha)+\epsilon^{2}x_{1}^{\prime}(\alpha)+\epsilon^{3}x_{2}^{\prime}(\alpha)+\ldots=f\left(x_{0}(\alpha),\alpha\right)+\epsilon Dfx_{1}(\alpha)+\left.\frac{1}{2}\frac{\partial^{2}}{\partial\epsilon^{2}}f\left(\sum_{p\geq0}\epsilon^{p}x_{p}(\alpha),\alpha\right)\right|_{\epsilon=0}\epsilon^{2}+\ldots.
\]
\linebreak{}
Note that
\begin{align*}
\frac{\partial^{2}}{\partial\epsilon^{2}}f\left(\sum_{j\geq0}\epsilon^{j}x_{j}(\alpha),\alpha\right) & =\frac{\partial}{\partial\epsilon}\left[D_{x}f\left(\sum_{j\geq0}\epsilon^{j}x_{j}(\alpha),\alpha\right)\sum_{j\geq0}j\epsilon^{j-1}x_{j}(\alpha)\right]\\
 & =D_{x}^{2}f\left(\sum_{j\geq0}\epsilon^{j}x_{j}(\alpha),\alpha\right)\otimes\left(\sum_{j\geq1}j\epsilon^{j-1}x_{j}(\alpha)\right)\otimes\left(\sum_{j\geq1}j\epsilon^{j-1}x_{j}(\alpha)\right)\\
 & +D_{x}f\left(\sum_{j\geq0}\epsilon^{j}x_{j}(\alpha),\alpha\right)\sum_{j\geq2}j\left(j-1\right)\epsilon^{j-2}x_{j}(\alpha),
\end{align*}
 therefore, we have
\[
\left.\frac{1}{2}\frac{\partial^{2}}{\partial\epsilon^{2}}f\left(\sum_{j\geq0}\epsilon^{j}x_{j}(\alpha),\alpha\right)\right|_{\epsilon=0}=\frac{1}{2}D_{x}^{2}f\left(x_{0}(\alpha),\alpha\right)\otimes x_{1}(\alpha)\otimes x_{1}(\alpha)+D_{x}f\left(x_{0}(\alpha),\alpha\right)x_{2}(\alpha)
\]
which gives formulas (\ref{eq:adiabatic anchor up to second order})
as solutions up to second order. We note that $A^{-1}(\alpha)=\left[D_{x}f\left(x_{0}(\alpha),\alpha\right)\right]^{-1}$
is known to exist by the hyperbolicity assumption (\ref{eq:hyperbolicity assumption for slo-fast}). 

To obtain the coefficients in eq. (\ref{eq:expansion for x(alpha)})
up to an arbitrary order $\nu$, we need to solve the recursive set
of algebraic equations (\ref{eq:general x-nu formula}) for $x_{j}(\alpha)$.
To solve these equations, we again recall the multi-variate Faá di
Bruno formula (\ref{eq:general Faa di Bruno formula}), for which
we now have $p=1$, wherein we have $H=f_{q}$, $g^{i}=\sum_{p\geq0}\epsilon^{p}x_{p}^{i}(\alpha)$,
$\mathbf{x}=\epsilon\in\mathbb{R},$ $\mathbf{x}^{0}=0\in\mathbb{R}$,
$\mathbf{\boldsymbol{\nu}}=\nu\in\mathbb{N}$, $\boldsymbol{\ell}_{i}=\ell_{i}\in\mathbb{N}$,
and $\mathbf{y}=x\in\mathbb{R}^{n}$. Therefore, for
\[
H(\epsilon)=f_{q}\left(\sum_{j\geq0}\epsilon^{j}x_{j}^{1}(\alpha),\ldots,\sum_{j\geq0}\epsilon^{j}x_{j}^{n}(\alpha),\alpha\right),\quad q=1,\ldots,n,
\]
 formula (\ref{eq:general Faa di Bruno formula}) gives
\begin{align*}
\frac{d^{\nu}}{d\epsilon^{\nu}}H\left(0\right) & =\left.D_{\epsilon}^{\nu}f_{q}\left(\sum_{j\geq0}\epsilon^{j}x_{j}^{1}(\alpha),\ldots,\sum_{j\geq0}\epsilon^{j}x_{j}^{n}(\alpha),\alpha\right)\right|_{\epsilon=0}\\
 & =\sum_{1\leq\left|\boldsymbol{\gamma}\right|\leq\nu}D_{x}^{\mathbf{\boldsymbol{\gamma}}}f_{q}\left(x_{0}(\alpha),\alpha\right)\sum_{s=1}^{\nu}\sum_{p_{s}\left(\nu,\boldsymbol{\gamma}\right)}\nu!\prod_{j=1}^{s}\frac{\prod_{i=1}^{n}\left[\left(\ell_{j}\right)!x_{\ell_{j}}^{i}(\alpha)\right]^{k_{ji}}}{\left(\mathbf{k}_{j}\right)!\left[\left(\ell_{j}\right)!\right]^{\left|\mathbf{k}_{j}\right|}}\\
 & =\sum_{1\leq\left|\boldsymbol{\gamma}\right|\leq\nu}D_{x}^{\boldsymbol{\gamma}}f_{q}\left(x_{0}(\alpha),\alpha\right)\sum_{s=1}^{\nu}\sum_{p_{s}\left(\nu,\boldsymbol{\gamma}\right)}\nu!\prod_{j=1}^{s}\frac{\prod_{i=1}^{n}\left[x_{\ell_{j}}^{i}(\alpha)\right]^{k_{ji}}}{\prod_{i=1}^{n}k_{ji}!}\\
 & =\nu!\left[A(\alpha)x_{\nu}(\alpha)\right]_{q}\\
 & +\sum_{1<\left|\boldsymbol{\gamma}\right|\leq\nu}D_{x}^{\boldsymbol{\gamma}}f_{q}\left(x_{0}(\alpha),\alpha\right)\sum_{s=1}^{\nu}\sum_{p_{s}\left(\nu,\boldsymbol{\gamma}\right)}\nu!\prod_{j=1}^{s}\frac{\prod_{i=1}^{n}\left[x_{\ell_{j}}^{i}(\alpha)\right]^{k_{ji}}}{\prod_{i=1}^{n}k_{ji}!},\quad q=1,\ldots,n.
\end{align*}
Using this last expression in (\ref{eq:general x-nu formula}), we
obtain formula \eqref{eq:recursive formula for adiabatic anchor}.

\subsection{Existence of the adiabatic SSM $\mathcal{M}_{\epsilon}$\label{subsec:Existence-of-the-adiabatic-SSM}}

The classic persistence results of \citet{fenichel71} assume compactness
for the underlying manifold and hence do not guarantee the smooth
persistence of $\mathcal{M}_{0}$ for $\epsilon>0$. To conclude the
persistence of $\mathcal{M}_{0}$, we first employ the ``wormhole''
construct from Proposition B1 of \citet{eldering18} that extends
$\mathcal{M}_{0}$ smoothly over its boundary so that it becomes a
subset of a $\rho$-normally hyperbolic, normally attracting, class
$C^{r}$ invariant manifold $\mathcal{\tilde{M}}_{0}$ without boundary.
Under this extension, the stable foliation of $W^{s}\left(\mathcal{M}_{0}\right)$
coincides with that part of the stable foliation of $W^{s}\left(\mathcal{\tilde{M}}_{0}\right)$.

Due to the exclusion of a $1:1$ resonance (see eq. (\ref{eq:lowered lower boudn on summation in nonresonance for adiabatic case}))
between eigenvalues inside and outside $\mathcal{\tilde{M}}_{0}$,
, the non-compact, boundaryless extended manifold $\mathcal{\tilde{M}}_{0}$
is a normally attracting invariant manifold. Its persistence can then
be concluded for small enough $\epsilon>0$ from related results by
\citet{eldering13} as long as assumption (\ref{eq:rho-normal hyperbolicity-1})
holds for the same $\rho$, uniformly in $\alpha$. In addition, $\mathcal{M}_{0}$
and $f(x,\alpha)$ must have $r$ derivatives that are uniformly bounded
in $\alpha$ in a small neighborhood of $\mathcal{M}_{0}.$ For $\epsilon>0$
small enough, we then obtain a unique (for a given choice of the wormhole
construct) persistent invariant manifold $\mathcal{M}_{\epsilon}$
that is diffeomorphic to $\mathcal{M}_{0}$, has $r$ uniformly bounded
derivatives and is $\mathcal{O}\left(\varepsilon\right)$ $C^{1}$
-close to $\mathcal{M}_{0}$. The smoothness class of $\mathcal{M}_{\epsilon}$
is $C^{m}$ where $m=\min\left(r,\rho\right)$. 

We do not obtain uniqueness from any of these constructs as they all
involve modifications of the vector field. But the initial conditions
are anyway with probability zero on an SSM that is smoother than the
smoothness implied by the spectral gap. In other words, there is an
inherent non-uniqueness in the choice of $\mathcal{M}_{\epsilon}$
as a $\rho$-normally hyperbolic invariant manifold tangent to $E(\alpha)$
for each $\alpha$, as there are infinitely many different choices
for $W\left(E(\alpha)\right)$ to begin with.

Despite the non-uniqueness of $\mathcal{M}_{\epsilon}$, all persisting
manifolds $\mathcal{M}_{\epsilon}$ must contain the unique, persisting
continuation $\mathcal{L}_{\epsilon}$ of $\mathcal{L}_{0}$. The
reason is that by the results of \citet{eldering13},  $\mathcal{L}_{\epsilon}$
is unique, uniformly bounded and lies fully in a small, inflowing
neighborhood of $\mathcal{M}_{\epsilon}$ whose size is $\mathcal{O}(1)$
in $\epsilon$. Points on $\mathcal{L}_{\epsilon}$ would then be
mapped by the inverse flow map outside that inflowing neighborhood,
unless they are contained in $\mathcal{M}_{\epsilon}$. Therefore,
$\mathcal{L}_{\epsilon}\subset\mathcal{M}_{\epsilon}$ must hold. 

\subsection{Computation of the adiabatic SSM $\mathcal{M}_{\epsilon}$\label{subsec:Computation-of-the-adiabatic-SSM}}

Based on formula (\ref{eq:epsilon derivatives-1}), 
\begin{align}
\hat{f}(u,v,\epsilon;\alpha) & =P^{-1}\left(\alpha\right)\left[f\left(x_{\epsilon}(\alpha)+P\left(\alpha\right)\left(\begin{array}{c}
u\\
v
\end{array}\right),\alpha\right)-A\left(\alpha\right)P\left(\alpha\right)\left(\begin{array}{c}
u\\
v
\end{array}\right)-\epsilon x_{\epsilon}^{\prime}(\alpha)-\epsilon P^{\prime}\left(\alpha\right)\left(\begin{array}{c}
u\\
v
\end{array}\right)\right]\nonumber \\
 & =P^{-1}\left(\alpha\right)\left[f\left(x_{\epsilon}(\alpha),\alpha\right)+D_{x}f\left(x_{\epsilon}(\alpha),\alpha\right)P\left(\alpha\right)\left(\begin{array}{c}
u\\
v
\end{array}\right)\right.\nonumber \\
 & \,\,\,\,\,\,\,\,\,\,\,\,\,\,\,\,\,\,\,\,\,\,\,\,\,\,\,\,\,\,\,\,\left.-A\left(\alpha\right)P\left(\alpha\right)\left(\begin{array}{c}
u\\
v
\end{array}\right)-\epsilon x_{\epsilon}^{\prime}(\alpha)-\epsilon P^{\prime}\left(\alpha\right)\left(\begin{array}{c}
u\\
v
\end{array}\right)+\mathcal{O}\left(\left|u\right|^{2},\left|u\right|\left|v\right|,\left|v\right|^{2}\right)\right]\nonumber \\
 & =P^{-1}\left(\alpha\right)\left[\left[D_{x}f\left(x_{\epsilon}(\alpha),\alpha\right)-D_{x}f\left(x_{0}(\alpha),\alpha\right)\right]P\left(\alpha\right)\left(\begin{array}{c}
u\\
v
\end{array}\right)\right.\nonumber \\
 & \,\,\,\,\,\,\,\,\,\,\,\,\,\,\,\,\,\,\,\,\,\,\,\,\,\,\,\,\,\,\,\,\left.+\mathcal{O}\left(\left|u\right|^{2},\left|u\right|\left|v\right|,\left|v\right|^{2},\epsilon\left|u\right|,\epsilon\left|v\right|\right)\right]\\
 & =P^{-1}\left(\alpha\right)\left[\mathcal{O}\left(\left|u\right|^{2},\left|u\right|\left|v\right|,\left|v\right|^{2},\epsilon\left|u\right|,\epsilon\left|v\right|\right)\right].\label{eq:fhat definition-1-1}
\end{align}
Under the assumptions (\ref{eq:hyperbolicity assumption for slo-fast})
and (\ref{eq:adiabatic spectrum assumption}), and by the definition
of $\hat{f}$ in (\ref{eq:fhat definition-1}), we have
\begin{equation}
\hat{f}(0,0,0;\alpha)=0,\quad D_{u}\hat{f}(0,0,0;\alpha)=0,\quad D_{v}\hat{f}(0,0,0;\alpha)=0,\quad D_{\epsilon}\hat{f}(0,0,0;\alpha)=0.\label{eq:vanishing derivatives-2}
\end{equation}
Specifically, in the $(u,v,\alpha)$ coordinates, the perturbed slow
manifold $\mathcal{L}_{\epsilon}$ satisfies
\[
\mathcal{L}_{\epsilon}=\left\{ (u,v,\alpha)\in\mathbb{R}^{n}\times\mathbb{R}\colon\,\,\,\,\,u=0,\quad v=0\right\} .
\]

We note the similarity between formulas (\ref{eq:(u,v) system-2})-(\ref{eq:vanishing derivatives-2})
and the setting of eqs. (\ref{eq:(u,v) system})-(\ref{eq:vanishing derivatives})
for general non-autonomous SSMs. Based on this similarity, we will
follow the same strategy here that we employed to compute invariant
manifolds in system (\ref{eq:(u,v) system}). Specifically, we will
seek the perturbed invariant manifold $\mathcal{M}_{\epsilon}$ in
system (\ref{eq:(u,v) system-2}) in the form of the asymptotic expansion
(\ref{eq:h double expansion-1}).

On the one hand, differentiating the definition of the invariant manifold
$\mathcal{M}_{\epsilon}$ from eq. (\ref{eq:h double expansion-1})
with respect to $t$ and using the ODE (\ref{eq:(u,v) system-1}),
we obtain 
\begin{align}
\dot{v} & =D_{u}h_{\epsilon}(u,\alpha)\dot{u}+\epsilon D_{\alpha}h_{\epsilon}(u,\alpha)\nonumber \\
 & =D_{u}h_{\epsilon}(u,\alpha)\left[A^{u}(\alpha)u+\hat{f}^{u}(u,h_{\epsilon}(u,\alpha),\alpha;\epsilon)\right]+\epsilon D_{\alpha}h_{\epsilon}(u,\alpha).\label{eq:vdoteq1-1}
\end{align}
On the other hand, substitution of the definition of the invariant
manifold $\mathcal{M}_{\epsilon}$ into \eqref{eq:(u,v) system-2}
gives
\begin{align}
\dot{v} & =A^{v}h_{\epsilon}(u,\alpha)+\hat{f}^{v}(u,h_{\epsilon}(u,\alpha),\alpha;\epsilon).\label{eq:vdoteq2-1}
\end{align}
Comparing (\ref{eq:vdoteq1-1}) and (\ref{eq:vdoteq2-1}) we obtain
\begin{equation}
D_{u}h_{\epsilon}\left[A^{u}u+\hat{f^{u}}\right]+\epsilon D_{\alpha}h_{\epsilon}=A^{v}h_{\epsilon}+\hat{f^{v}}.\label{eq:invariance PDE-1}
\end{equation}

\subsubsection{Structure and solution of the invariance equation}

Substitution of (\ref{eq:h double expansion-1}) into (\ref{eq:invariance PDE-1})
gives
\[
\sum_{\left|\left(\mathbf{k},p\right)\right|\geq1}h^{\mathbf{k}p}(\alpha)\epsilon^{p}D_{u}u^{\mathbf{k}}\left[A^{u}(\alpha)u+\hat{f}^{u}\left(u,\sum_{\left|\left(\mathbf{k},p\right)\right|\geq1}h^{\mathbf{k}p}(\alpha)u^{\mathbf{k}}\epsilon^{p},\alpha;\epsilon\right)\right]+\sum_{\left|\left(\mathbf{k},p\right)\right|\geq1}\left[h^{\mathbf{k}p}\right]^{\prime}(\alpha)u^{\mathbf{k}}\epsilon^{p+1}
\]
 
\begin{equation}
=\sum_{\left|\left(\mathbf{k},p\right)\right|\geq1}A^{v}h^{\mathbf{k}p}(\alpha)u^{\mathbf{k}}\epsilon^{p}+\hat{f}^{v}\left(u,\sum_{\left|\left(\mathbf{k},p\right)\right|\geq1}h^{\mathbf{k}p}(\alpha)u^{\mathbf{k}}\epsilon^{p},\alpha;\epsilon\right).\label{eq:nonautonomous invariance eq.-1-1}
\end{equation}
As we did in the general non-autonomous case, we observe that 
\[
A^{v}(\alpha)h^{\mathbf{k}p}(\alpha)u^{\mathbf{k}}-h^{\mathbf{k}p}(\alpha)D_{u}u^{\mathbf{k}}A^{u}(\alpha)u
\]
\begin{equation}
=\left(\begin{array}{ccc}
\lambda_{d+1}(\alpha)-\sum_{j=1}^{d}k_{j}\lambda_{j}(\alpha) & \cdots & 0\\
\vdots & \ddots & \vdots\\
0 & \cdots & \lambda_{n}(\alpha)-\sum_{j=1}^{d}k_{j}\lambda_{j}(\alpha)
\end{array}\right)h^{\mathbf{k}p}(\alpha)u^{\mathbf{k}}.\label{eq:leading order terms in invariance equation-1-1}
\end{equation}
Therefore, the invariance equation (\ref{eq:nonautonomous invariance eq.-1-1})
can be rewritten as
\begin{equation}
\sum_{\left|\left(\mathbf{k},p\right)\right|\geq1}\left[h^{\mathbf{k}p}\right]^{\prime}(\alpha)u^{\mathbf{k}}\epsilon^{p+1}=\sum_{\left|\left(\mathbf{k},p\right)\right|\geq1}A_{\mathbf{k}}(\alpha)h^{\mathbf{k},p}(\alpha)u^{\mathbf{k}}\epsilon^{p}+\sum_{\left|\left(\mathbf{k},p\right)\right|\geq1}M^{\mathbf{k}p}(\alpha,h^{\mathbf{j}m})u^{\mathbf{k}}\epsilon^{p},\label{eq:simplified invariance eq.-1-1}
\end{equation}
where 
\begin{align}
A_{\mathbf{k}} & (\alpha)=\mathrm{diag}\left[\lambda_{\ell}(\alpha)-\sum_{j=1}^{d}k_{j}\lambda_{j}(\alpha)\right]_{\ell=d+1}^{n}\in\mathbb{C}^{\left(n-d\right)\times\left(n-d\right)},\nonumber \\
\sum_{\left|\left(\mathbf{k},p\right)\right|\geq1}M^{\mathbf{k}p}(\alpha,h^{\mathbf{j}m})u^{\mathbf{k}}\epsilon^{p} & =M(u,\alpha,h^{\mathbf{j}m},\epsilon)=\hat{f}^{v}\left(u,\sum_{\left|\left(\mathbf{k},p\right)\right|\geq1}h^{\mathbf{k}p}(\alpha)u^{\mathbf{k}}\epsilon^{p},\alpha;\epsilon\right)\label{eq:M^k-1-1}
\end{align}
\[
-\sum_{\left|\left(\mathbf{k},p\right)\right|\geq1}\epsilon^{p}\left(\begin{array}{ccc}
h_{1}^{\mathbf{k}p}(\alpha)\frac{k_{1}u^{\mathbf{k}}}{u_{1}} & \cdots & h_{1}^{\mathbf{k}p}(\alpha)\frac{k_{d}u^{\mathbf{k}}}{u_{d}}\\
\vdots & \ddots & \vdots\\
h_{n-d}^{\mathbf{k}p}(\alpha)\frac{k_{1}u^{\mathbf{k}}}{u_{1}} & \cdots & h_{n-d}^{\mathbf{k}p}(\alpha)\frac{k_{d}u^{\mathbf{k}}}{u_{d}}
\end{array}\right)\hat{f}^{u}\left(u,\sum_{\left|\left(\mathbf{k},p\right)\right|\geq1}h^{\mathbf{k}p}(\alpha)u^{\mathbf{k}}\epsilon^{p},\alpha;\epsilon\right).
\]

Recall from eq. (\ref{eq:vanishing derivatives-2}) that $\hat{f}^{v}$
and $\hat{f}^{u}$ vanish for $u,v,\epsilon=0$ and have no $\mathcal{O}\left(\left|v\right|\right)$
terms. Then, as we concluded in the general non-autonomous case, $M^{\mathbf{k}p}(\alpha,h^{\mathbf{j}m})$
will only depend on $h^{\mathbf{j}m}$ that are lower in order, i.e.,
\begin{equation}
D_{h^{\mathbf{j}m}}M^{\mathbf{k}p}(\alpha,h^{\mathbf{j}m})u^{\mathbf{k}}=0,\quad\left|(\mathbf{j},m)\right|\geq\left|(\mathbf{k},p)\right|.\label{eq:zero Djm}
\end{equation}

Equating coefficients of equal powers of $u$ in eq. (\ref{eq:simplified invariance eq.-1-1}),
we then obtain the system of equations
\begin{equation}
\left[h^{\mathbf{k}(p-1)}\right]^{\prime}(\alpha)=A_{\mathbf{k}}(\alpha)h^{\mathbf{k}p}(\alpha)+M^{\mathbf{k}p}(\alpha,h^{\mathbf{j}m}),\quad\left|(\mathbf{j},m)\right|<\left|(\mathbf{k},p)\right|.\label{eq:system of inhomogeneous ODEs-1-1}
\end{equation}
This is a recursively defined set of linear algebraic equations for
$h^{\mathbf{k}p}(\alpha)$, which can be uniquely solved as long as
$A_{\mathbf{k}}(\alpha)$ is nonsingular, i.e., the non-resonance
conditions (\ref{eq:strict nonresonance-1}) are satisfied. In that
case, the recursive solution of (\ref{eq:system of inhomogeneous ODEs-1-1})
starts from 
\begin{equation}
h^{\mathbf{0}p}(\alpha)\equiv0,\quad p\geq0,\qquad h^{\mathbf{k}0}(\alpha)\equiv h^{\mathbf{k}}(\alpha),\quad\mathbf{k}\in\mathbb{N}^{d},\qquad h^{\mathbf{k}0}(\alpha)\equiv h^{\mathbf{k}}(\alpha)=0,\quad\left|\mathbf{k}\right|=1,\label{eq:known k^=00007Bkp=00007D coeffs in adiabatic expansion-1}
\end{equation}
and takes the form (\ref{eq:h_=00007Bkp=00007D  adiabatic case}),
with the quantities in eq. (\ref{eq:M^k-1-2-1-1-1}) obtained from
formulas (\ref{eq:M^k-1-1}) using the relation (\ref{eq:zero Djm}).

\subsubsection{Reduced dynamics}

To obtain the form of the reduced dynamics on the adiabatic SSM $\mathcal{M}_{\epsilon}$,
we consider the $u$ component of the transformation formula (\ref{eq:u-v transformation,  adiabatic case}), 

\emph{
\begin{equation}
u=Q_{u}\left(\alpha\right)\left(x-x_{\epsilon}(\alpha)\right),\label{eq:u part of variable change for adiabatic case}
\end{equation}
}where the rows of $Q_{u}\left(\alpha\right)\in\mathbb{C}^{d\times n}$
are the appropriately scaled unit left eigenvectors of $A(\alpha)$
corresponding to its first $d$ (right) eigenvectors. This scaling
is specified in statement (ii) of the theorem, ensuring that the rows
of $Q_{u}\left(\alpha\right)$ coincide with the first $d$ rows of
$P^{-1}(\alpha)$. 

Differentiation of (\ref{eq:u part of variable change for adiabatic case})
with respect to time gives\emph{
\begin{align*}
\dot{u} & =\epsilon Q_{u}^{\prime}\left(\alpha\right)\left(x-x_{\epsilon}(\alpha)\right)+Q_{u}\left(\alpha\right)\left(\dot{x}-\epsilon x_{\epsilon}^{\prime}(\alpha)\right)\\
 & =\epsilon Q_{u}^{\prime}\left(\alpha\right)P(\alpha)\left(\begin{array}{c}
u\\
v
\end{array}\right)+Q_{u}\left(\alpha\right)\left(f(x,\alpha)-\epsilon x_{\epsilon}^{\prime}(\alpha)\right)\\
 & =Q_{u}\left(\alpha\right)\left(f\left(x_{\epsilon}(\alpha)+P(\alpha)\left(\begin{array}{c}
u\\
v
\end{array}\right),\alpha\right)-\epsilon x_{\epsilon}^{\prime}(\alpha)\right)+\epsilon Q_{u}^{\prime}\left(\alpha\right)P(\alpha)\left(\begin{array}{c}
u\\
v
\end{array}\right).
\end{align*}
}Restricting this last formula to the graph $v=h_{\epsilon}(u,\alpha)$
then proves formula (\ref{eq:adiabatic restriced dynamics}) in statement
(ii) of the theorem.

\bibliographystyle{plainnat}
\bibliography{SSM_bibliography}

\end{document}